\documentclass[11pt,letterpaper]{amsart}

\usepackage{amsbsy,amssymb,amscd,amsfonts,latexsym,amstext,delarray,amsmath,graphicx,color,caption,blkarray,mathtools,comment}
\usepackage{tikz}
\usetikzlibrary{cd}
\usetikzlibrary{matrix,arrows}
\usepackage{graphicx}
\usepackage{hyperref}
\input xy
\xyoption{all}
\usepackage{euscript}
\usepackage{multirow,charter}
\usepackage{etex, pictexwd,dcpic}
\usepackage[normalem]{ulem}

\usetikzlibrary{backgrounds}
\usepackage{amscd}
\usepackage{times,fancyhdr}
\usepackage{array}
\usepackage{epsfig}
\usepackage{graphicx}
\usepackage{color}
\usepackage{mathtools}
\usepackage{mathabx}
\usepackage{euscript}
\usepackage{footnote}
\usepackage[normalem]{ulem}

\newcounter{thecounter}
\numberwithin{thecounter}{section}
\newtheorem{lemma}[thecounter]{Lemma}
\newtheorem{proposition}[thecounter]{Proposition}
\newtheorem{theorem}[thecounter]{Theorem}

\newtheorem{corollary}[thecounter]{Corollary}
\newtheorem{conjecture}[thecounter]{Conjecture}

\theoremstyle{definition}
\newtheorem{defn}[thecounter]{Definition}

\DeclareMathOperator{\SL}{SL}
\DeclareMathOperator{\GL}{GL}
\DeclareMathOperator{\PGL}{PGL}
\DeclareMathOperator{\gl}{\mathfrak{gl}}
\DeclareMathOperator{\Aut}{Aut}
\DeclareMathOperator{\Inn}{Inn}
\DeclareMathOperator{\End}{End}
\DeclareMathOperator{\Exp}{Exp}

\DeclareMathOperator{\Ad}{Ad}
\DeclareMathOperator{\ad}{ad}

\DeclareMathOperator{\re}{{re}}
\DeclareMathOperator{\im}{{im}}

\renewcommand{\a}{\alpha}
\renewcommand{\b}{\beta}
\newcommand{\g}{\gamma}

\newcommand{\Z}{{\mathbb Z}}
\newcommand{\R}{{\mathbb R}}
\newcommand{\C}{{\mathbb C}}
\newcommand{\N}{{\mathbb N}}
\newcommand{\Q}{{\mathbb Q}}
\newcommand{\M}{{\mathbb M}}
\newcommand{\Up}{U^+}
\newcommand{\Um}{U^-}

\newcommand{\Pp}{{{P}^+}}
\newcommand{\Pm}{{{P}^-}}
\newcommand{\Ppm}{{{P}^\pm}}
\newcommand{\Upm}{U^\pm}

\newcommand{\Uhp}{{\widehat{U}^+}}
\newcommand{\Uhm}{{\widehat{U}^-}}
\newcommand{\UhPi}{{\widehat{U}^{\,\Pi}}}

\newcommand{\Php}{{\widehat{P}^+}}
\newcommand{\Phm}{{\widehat{P}^-}}
\newcommand{\Phpm}{{\widehat{P}^\pm}}
\newcommand{\Uhpm}{{\widehat{U}^\pm}}
\newcommand{\Uhimp}{{\widehat{U}_{\im}^+}}

\newcommand{\Uhimpm}{{\widehat{U}_{\im}^\pm}}

\newcommand{\m}{{\mathfrak{m}}}
\newcommand{\mhat}{{\widehat{\mathfrak{m}}}}

\newcommand{\n}{{\mathfrak{n}}}
\newcommand{\nh}{\widehat{\mathfrak{n}}}
\newcommand{\np}{{\mathfrak{n}^+}}
\newcommand{\nhp}{{\widehat{\mathfrak{n}}^+}}
\newcommand{\nm}{{\mathfrak{n}^-}}
\newcommand{\nhm}{{\widehat{\mathfrak{n}}^-}}
\newcommand{\npm}{{\mathfrak{n}^\pm}}
\newcommand{\nimp}{{\mathfrak{n}_{\im}^+}}
\newcommand{\nhimp}{{\widehat{\mathfrak{n}}_{\im}^+}}
\newcommand{\nimm}{{\mathfrak{n}_{\im}^-}}
\newcommand{\nhimm}{{\widehat{\mathfrak{n}}_{\im}^-}}
\newcommand{\nimpm}{{\mathfrak{n}_{\im}^\pm}}
\newcommand{\php}{\widehat{\mathfrak{p}}^+}

\newcommand{\e}{\overline{e}}
\newcommand{\f}{\overline{f}}

\begin{document}

\title[A Lie group analog for the Monster Lie algebra]{A Lie group analog for the Monster Lie algebra}
\author{Lisa Carbone, Elizabeth Jurisich and Scott H.\ Murray}

%%%%%%%%%%%%%%%%%%%%%%%%%%%%%%%%%%%%%%%%%%%%%%%%%%%%%%
% Long and short versions
% NB: rerun bibtex - the long version has references not in the short
\includecomment{longver}\excludecomment{shortver}
%\excludecomment{longver}\includecomment{shortver}

%%%%%%%%%%%%%%%%%%%%%%%%%%%%%%%%%%%%%%%%%%%%%%%%%%%%%%
\begin{abstract}
The Monster Lie algebra $\m$, which admits an action of the Monster finite simple group $\mathbb{M}$, was introduced by Borcherds as part of his work on the Conway--Norton Monstrous Moonshine conjecture.
Here we construct an analog~$G(\frak m)$ of a Lie group or Kac--Moody group
associated to~$\frak m$. The group~$G(\frak m)$ is given by generators and relations, similar to a  construction of a Kac--Moody group given by Tits.   Our group $G(\mathfrak m)$ contains generators corresponding to the
unique real simple root and to a distinguished family of
positive and negative imaginary root vectors.
In the absence of local nilpotence of the adjoint representation  of $\frak m$, we introduce the notion of pro-summability of an infinite sum of operators. 
We consider the completion $\widehat{\mathfrak{m}}=\frak n^-\ \oplus\ \frak h\ \oplus\ \widehat{\frak n}^+$ of~$\mathfrak{m}$,  where $\widehat{\frak n}^+$ is the formal direct product of the positive root spaces of $\frak m$.  
 We construct a group ~$\Uhimp $, which is a group of `smearing automorphisms' of the subalgebra $\nhimp$ of $\widehat{\frak n}^+$ corresponding to the positive imaginary roots of~$\frak m$. This extends to a group $\Uhp$ of automorphisms  of $\nhp$.  We prove that the elements of $\Uhp$  are pro-summable   series with constant term 1, and hence $\Uhp$ is pro-unipotent.
%We give  analogs  $\Exp: \widehat{\mathfrak{n}}^+\to\Uhp$ and  ${\Ad} :\Uhp \to \Aut(\widehat{\frak{n}}^+)$ of the classical exponential map  and  adjoint representation.   
We show that the action of the Monster finite group $\M$ on $\mathfrak m$ induces an action of~$\M$ on~$\widehat{\frak m}$, and that this induces a compatible  action of $\M$ on~$\Uhp$. %We also show that the action of $\M$ on $\widehat{\mathfrak n}^+$ is compatible with the action of  $\widehat{U}^+$ on $\widehat{\mathfrak n}^+$. 
%We construct an analog of the Steinberg group and one-parameter subgroups for $\frak m$.
Although the group $G(\mathfrak m)$ does not act on $\mathfrak m$, 
we show that  rank two linear groups and certain
parabolic and unipotent groups admit natural homomorphisms to
automorphism groups of $\mathfrak m$, of certain
$\mathfrak{gl}_2$ subalgebras of $\mathfrak m$, of the completion
$\widehat{\mathfrak m}$, and of related completions.
We also construct the negative completion $\widehat{\mathfrak{m}}^-$ and as a first step in  constructing a representation of $G(\mathfrak{m})$, we give a homomorphism from $G(\mathfrak{m})$ to a quotient of $\text{Aut}(\widehat{\mathfrak{m}}^+) *_{\GL_2(-1)} \text{Aut}(\widehat{\mathfrak{m}}^-).$
\end{abstract}

\thanks{The first author's research is partially supported by the Simons Foundation, Mathematics and Physical Sciences Collaboration Grants for Mathematicians, Award Number 961267.
The authors would like to thank Jim Lepowsky, Yi-Zhi Huang and Siddhartha Sahi for their interest in this work. The first two authors would like to thank Universit\`a dell'Insubria, where some of this work was undertaken. The first author would like to thank the IAS-SNS for its hospitality and productive work environment during a sabbatical visit. MSC 22E65, 20N99.}

\maketitle

\begin{longver}
\tableofcontents
\end{longver}

%%%%%%%%%%%%%%%%%%%%%%%%%%%%%%%%%%%%%%%%%%%%%%%%%%%%%%
%%%%%%%%%%%%%%%%%%%%%%%%%%%%%%%%%%%%%%%%%%%%%%%%%%%%%%
\section{Introduction}\label{S-intro}
Let $\mathbb{M}$ be the Monster finite simple group and let $V^\natural$ be the Moonshine module  of  Frenkel, Lepowsky and Meurman \cite{FLM}, a  vertex operator algebra with $\Aut(V^\natural)=\mathbb{M}$.    With the construction of~$V^\natural$, Frenkel, Lepowsky and Meurman proved the McKay--Thompson conjecture \cite{Th} that there  exists an infinite dimensional $\mathbb M$-module whose graded dimension is the normalized elliptic modular invariant~$J(q)$.    In their seminal work \cite{FLM}, they also partially settled the Conway--Norton Monstrous Moonshine conjecture \cite{CN} by proving it for  a particular subgroup of $\mathbb M$  that is an extension of the Conway group Co$_1$ by a certain extraspecial 2-group. %\textcolor{red}{$C$ never seems to come up again. why do we need to mention it in the first parargraph.}

The \emph{Monster Lie algebra} $\m$, which also admits an action of the Monster finite simple group $\mathbb{M}$, was constructed by Borcherds \cite{BoPNAS86} as a  quotient of the physical space  of the vertex algebra $V=V^\natural\otimes V_{{1,1}}$, where  $V_{{1,1}}$ is  the vertex  algebra  for the even unimodular  2-dimensional Lorentzian lattice  $\text{II}_{1,1}$.   Borcherds showed that  the Monster Lie algebra is  isomorphic to $\mathfrak{g}(A)/\mathfrak{z},$ where $\mathfrak g(A)$ is the infinite dimensional Lie algebra associated with  the Borcherds Cartan matrix $A$ and $\mathfrak{z}$ is the center of~$\mathfrak g(A)$ (Section~\ref{S-borcherds}). Using $\frak m$  and the No-ghost Theorem from string theory~\cite{GT}, Borcherds proved the remaining cases of the Conway--Norton  conjecture.  The Monster Lie algebra was one of the first examples of a new type of Lie algebra introduced by Borcherds, now known as a Borcherds algebra or generalized Kac--Moody algebra.

In addition to the profound mathematical importance of $\mathfrak{m}$, the appearance of $\mathfrak{m}$ and other Borcherds  algebras  as symmetries in Heterotic string theory has been noted. In particular, 
Harvey and Moore~\cite{HM1996} showed that Borcherds  algebras appear as  gauge symmetry algebras of a compactification of the Heterotic string  to two dimensions. 
More recently, Paquette, Persson and  Volpato~\cite{PPV2016} showed that the Monster Lie algebra is an algebra of spontaneously broken gauge symmetries in their model of the compactified Heterotic string. 

Before the current work, there was no construction of an analog of a reductive Lie group for infinite dimensional Borcherds  algebras (see \cite{Bo99} for the construction of %an integral form of the universal enveloping algebra and 
the analogous group for the fake Monster Lie algebra).
%Here we address this question of associating an analog of a Lie group to the Monster Lie algebra. 
%
There are certain difficulties with constructing groups for Borcherds  algebras in general. A fundamental difference between  Kac--Moody algebras and Borcherds  algebras is that, in the Kac--Moody case,  the  simple roots are all real, that is, they have positive squared norm. 
Hence the entire Lie algebra is generated by real root vectors, namely those associated to the simple root vectors and their negatives. 
For Borcherds  algebras this is no longer the case, as there are simple roots which are imaginary, that is, with non-positive squared norm. 

One construction of a Kac--Moody group for a Kac--Moody algebra $\frak g$ is as a group generated by certain automorphisms of an integrable representation of $\frak g$, such as an integrable highest weight or adjoint representation. The  root vectors associated with real roots of a Kac--Moody or Borcherds algebra are, by definition, locally nilpotent 
in their 
action on an integrable module. Thus, when considering the adjoint representation,  there are automorphisms of the form $\exp(u\,\ad(e_{\a}))$ for $u\in\C$ and $e_{\a}$ a root vector corresponding to a real root~$\a$. This provides an infinite-dimensional Kac--Moody group, which is analogous to  the Chevalley group for a finite dimensional semisimple Lie algebra. 
However, this method  breaks down for Borcherds  algebras, since root vectors  corresponding to  imaginary simple roots do not act locally nilpotently on the adjoint representation or on any other faithful weight module. 
In fact, the adjoint representation of the Monster Lie algebra is not integrable  \cite{jurisich1996exposition}. Hence elements of the form $\exp(u\,\ad(e_{\a}))$ do not represent well-defined automorphisms of~$\m$.

Our approach to overcoming these obstacles is two-fold. Our first step is to 
introduce the notion of \emph{pro-summability}, which allows us to define exponentials without the local nilpotence condition (Subsection~\ref{SS-graded} and \cite{CJM}). In analogy with  the Kac--Moody case, where it is natural and useful to associate a group to the completion of the underlying Kac--Moody algebra, we construct a group of automorphisms of a completion $\widehat{\mathfrak{m}}$ of $\frak m$.   We start with the triangular decomposition 
$\mathfrak{m}=\mathfrak{n}^-\oplus\mathfrak{h}\oplus\mathfrak{n^+}$,
where $\Delta_\pm$ denotes the set of positive (respectively negative) roots, $\frak m_{\alpha}$ is the root space associated to $\a$,
and $\frak n^{\pm}= \oplus_{\alpha\in \Delta_\pm}\frak m_{\alpha}$. We also define $\nimpm = \bigoplus_{\a\in\Delta^{\im}_\pm} \frak m_\a.$
We obtain the formal completion of $\mathfrak{n}^+$ by replacing the infinite direct sum with the infinite direct product 
$\nhp = \prod_{\a\in\Delta_+}\mathfrak{m}_\a$, that is, by allowing elements with infinitely many nonzero components. 
%We also set $ \nhimp=  \prod_{\a\in\Delta_+} \frak m^{\im}_{\a}$ (Section~\ref{S-unip}).
We define the completion of $\mathfrak{m}$ as
$$\widehat{\mathfrak{m}}=\nm\oplus\mathfrak{h}\oplus\nhp.$$
Following the methods for Kac--Moody algebras  in~\cite{kumar2012kac}, we show that $\nhp$ is a pro-nilpotent Lie algebra (Section~\ref{S-unip}).
We construct a group $\Uhimp$ associated to the imaginary roots of $\mathfrak m$  (Section~\ref{S-unip}). 
The group $\Uhimp$  is now constructed as a group of \emph{smearing  automorphisms} of $\nhimp=\prod_{\a\in\Delta^{\im}_+} \frak m_{\a}$. In Subsection~\ref{SS-pro}, we show that the group $\Uhimp$ is naturally  pro-unipotent. 
This construction naturally extends to a pro-unipotent group construction 
$\Uhp \leq \Aut(\widehat{\mathfrak{n}}^+)\leq \Aut(\widehat{\mathfrak{m}})$ for $\widehat{\mathfrak{n}}^+$ (Section~\ref{S-related}, see also \cite{CJM}).
We also construct  analogs of the exponential map 
 $\Exp: \widehat{\mathfrak{n}}^+\to\Uhp$
 and  the adjoint representation
${\Ad} :\Uhp \to \Aut(\widehat{\frak{n}}^+)$ and show that they satisfy the familiar relationship $\Exp\left(\Ad(g)\right)(x)=g\Exp(x) g^{-1}$ (Section~\ref{SS-adjoint}).

We show that the action of $\M$ on $\mathfrak m$ induces an action of $\M$ on $\widehat{\frak m}$, and that this in turn induces an action of $\M$ on~$\Uhp$ (Theorem~\ref{action}). We show that the action of $\M$ on $\widehat{\mathfrak n}^+$ commutes with
the action of  $\widehat{U}^+$ on $\widehat{\mathfrak n}^+$. In fact, these actions also commute when extended to the parabolic subalgebra $\widehat{\mathfrak p}^+$ and parabolic group $\widehat{P}^+$ containing $\widehat{\mathfrak n}^+$ and $\widehat{U}^+$, respectively (Corollary~\ref{compat}).

Our second approach (Section~\ref{S-Grels}) is to construct a group $G(\frak m)$ given by generators and relations.  This is  an analog for $\frak m$ of  the group presentation for a symmetrizable Kac--Moody algebra given by Tits (as in~\cite{Ti87}).  The Monster Lie algebra $\mathfrak{m}$ is symmetrizable since the Borcherds generalized Cartan matrix for $\mathfrak{m}$ is symmetric. However, our approach is fundamentally different to that of \cite{Ti87}, as we include explicit generators for  positive imaginary roots.

We utilize the fact that $\mathfrak m$ is generated by $\mathfrak{gl}_2$ subalgebras corresponding to  both real and imaginary simple root vectors.
In fact, we find it convenient to extend this to include some non-simple imaginary roots, as this gives a set of free generators for $\nhimp$.
The group ${G}(\frak m)$ is generated by a group $\GL_2(-1)$ corresponding to the unique real positive simple root $(-1,1)$,  and an infinite family of groups $\GL_2(\ell,j,k)$ corresponding to imaginary  roots $(\ell+1,j-\ell)\in \text{II}_{1,1}$ for $j\geq 1$ and $1\leq k\leq c(j)$, where $c(j)$ is the coefficient of~$q^j$ in the normalized elliptic modular function $J(q)$. 
We determine the relations of $G(\frak m)$ by analyzing the corresponding $\mathfrak{gl}_2$ subalgebras of $\frak m$ (see Section~\ref{SS-identities}).
Since the adjoint representation of the Monster Lie algebra is not integrable, it is not possible to construct all the group elements of $G(\frak m)$ as automorphisms of $\frak m$. 
However, we show that the  rank \(2\) linear groups associated to
real and imaginary roots act as automorphism groups of the corresponding $\frak{gl}_2$ subalgebras of $\mathfrak{m}$ (Subsections ~\ref{SS-gl2re} and~\ref{SS-gl2im}).

In the Kac--Moody case, it suffices to consider the group generated by root groups corresponding to real roots. We must consider roots that are  not real in order to get a group $G(\mathfrak{m})$ that reflects the full structure of $\mathfrak{m}$.  For example, if we consider the group generated only by  locally $\ad$-nilpotent elements of $\frak{m}$, this would  give $\GL_2$ as an analog of a Kac--Moody group associated to $\frak{m}$ (see also \cite{Kum22}).

Although the group $G(\mathfrak m)$ does not act on $\mathfrak m$, 
we show in Section~\ref{S-Grels} that  rank two linear groups and certain
parabolic and unipotent groups admit natural homomorphisms to
automorphism groups of $\mathfrak m$, of certain
$\mathfrak{gl}_2$ subalgebras of $\mathfrak m$, of the completion
$\widehat{\mathfrak m}$, and of related completions.
In Section~\ref{rep}, we construct the negative completion $\widehat{\mathfrak{m}}^-$ and as a first step in  constructing a representation of $G(\mathfrak{m})$, we give a homomorphism from $G(\mathfrak{m})$ to a quotient of $\text{Aut}(\widehat{\mathfrak{m}}^+) *_{\GL_2(-1)} \text{Aut}(\widehat{\mathfrak{m}}^-).$
We conjecture that this homomorphism is injective.
 
Our two main group constructions are connected.
The structure of  $\Uhp$ gives rise to relations in the group~$G(\frak m)$.   Conversely, we show in Section~\ref{S-Grels} that the relevant group relations in $G(\frak m)$ are satisfied in $\Uhp$. 
 
For any Borcherds algebra $\frak g$, whose nontrivial highest weight and adjoint representations are not  integrable, we expect that it is not possible in general to construct a Lie group analog for the whole Lie algebra $\frak g$ that acts as automorphisms of $\frak g$. However, as we show in \cite{ACJM2}, the construction of a pro-summable  complete pro-unipotent group $\Uhp$ holds for all Borcherds algebras $\frak g$ and $\Uhp$  acts on a completion~$\widehat{\frak g}$ of~$\frak g$ as automorphisms. 
It should also be possible to construct an analog~$G(\frak g)$ of our group~$G(\frak m)$ for any Borcherds algebra~$\frak g$, relative to a Borcherds  Cartan matrix~$A$. We hope to consider this in future work.

%If $A$ has zeroes on the diagonal, corresponding to isotropic imaginary simple roots, a suitable modification of our construction would be needed.}
%these would give rise to 3-dimensional Heisenberg subgroups of $G(\frak g)$ corresponding to isotropic imaginary roots. As with our group $G(\frak m)$, it should be possible to construct groups of automorphisms of finite dimensional Heisenberg subalgebras of $\frak g$ corresponding to these subgroups. Non-zero entries on the diagonal of $A$  corresponding to real or imaginary roots would give rise to $\SL_2$ subgroups of $G(\frak g)$ and these should also correspond to groups of automorphisms of the relevant $\frak{sl}_2$ subalgebras of $\frak g$.

Our results in this work use the construction of $\mathfrak{m}$ as $\mathfrak{m}=\mathfrak{g}(A)/\mathfrak{z},$ where $\mathfrak g(A)$ is the infinite dimensional Lie algebra associated with  the Borcherds Cartan matrix $A$ (Section~\ref{S-borcherds}) and  $\mathfrak{z}$ is the center of~$\mathfrak g(A)$. In a future work~\cite{ACJKM2} we will give a construction  of a Lie group analog $\mathcal{G}(\mathfrak{m})$ using the equivalent construction of 
$\mathfrak{m}$ as $\mathfrak{m}=P_1/R$, where  $P_1$ is the physical space of $V=V^\natural\otimes V_{{1,1}}$ and $R$ is the radical of a natural bilinear form on $P_1$. Our construction of $\mathcal{G}(\mathfrak{m})$
encodes the action of $\M$ on primary vectors in $V$ induced from the~\cite{FLM}-action of~$\M$ on~$V^\natural$ and $\mathcal{G}(\mathfrak{m})$  is conjecturally a proper subgroup of the group~$G(\frak m)$.

The authors would like to thank Abid Ali,  Yi-Zhi Huang, James Lepowsky, Ugo Moschella and Siddhartha Sahi for their interest in this project and for helpful discussions. We thank Darlayne Addabbo for pointing out that Theorem~\ref{action} holds in our setting.
Much of the early part of this work was undertaken at Universit\'a degli studi dell'Insubria. The first two authors gratefully acknowledge the hospitality and discussions with Ugo Moschella that made this work possible. The first author would like to thank the IAS-SNS for its hospitality and productive work environment during a sabbatical visit.

%%%%%%%%%%%%%%%%%%%%%%%%%%%%%%%%%%%%%%%%%%%%%%%%%%%%%%
%%%%%%%%%%%%%%%%%%%%%%%%%%%%%%%%%%%%%%%%%%%%%%%%%%%%%%
\section{The Monster Lie algebra}\label{S-borcherds}
In this section, we describe the Monster Lie algebra $\mathfrak{m}$ and its central extension $\mathfrak{g}^e$. Most of the preliminary material here is routine. More detail is given in \cite{CarXiv}.
Note that we use the convention that $\N=\{1,2,3,\dots\}$.

%%%%%%%%%%%%%%%%%%%%%%%%%%%%%%%%%%%%%%%%%%%%%%%%%%%%%%
\subsection{Presentation of the Monster Lie algebra}\label{SS-pres}
Borcherds \cite{BoInvent}
 gave two equivalent constructions of the Monster Lie algebra:
\begin{enumerate} 
\item  Let $V^\natural$ be the Moonshine module of  Frenkel, Lepowsky and Meurman \cite{FLM}, let $V_{{1,1}}$ denote the vertex  algebra  for the even unimodular two-dimensional Lorentzian lattice  $\textrm{II}_{1,1}$, and let $V=V^\natural\otimes V_{{1,1}}$. The {\emph{physical space}} of $V$ is
$$P_1=\{\psi\in V^\natural\otimes V_{1,1}\mid L(0)\psi=\psi,\ L(j)\psi=0 \text{ for all $j> 0$}\},$$
where $L(j)$ is  the $j$th Virasoro generator.
Then the Monster Lie algebra is $P_1/R$, where  $R$ is the radical of a natural bilinear form on $P_1$. 

\item Let $\mathfrak g(A)$ be the Lie algebra associated with  the Borcherds Cartan matrix $A$ given below. Then the Monster Lie algebra is  $\mathfrak{g}(A)/\mathfrak{z},$ where $\mathfrak{z}$ is the center of~$\mathfrak g(A)$. 
\end{enumerate}
The proof  that these two Lie algebras are isomorphic uses %, in part,  
the No-ghost Theorem \cite{GT}. Here we consider only the second construction.

Following \cite{BoJAlg},  \cite{JLW},  \cite{jurisich1996exposition} and \cite{JurJPAA},  the Borcherds Cartan matrix for the Monster Lie algebra $\mathfrak m$ is
\begin{equation*}%https://tex.stackexchange.com/questions/10122/bordermatrix-with-blocks
{A=\smaller \begin{blockarray}{cccccccccc}
 & & \xleftrightarrow{c(-1)}  &   \multicolumn{3}{c}{$\xleftrightarrow{\hspace*{0.7cm}c(1)\hspace*{0.7cm}}$}   &  \multicolumn{3}{c}{$\xleftrightarrow{\hspace*{0.7cm} { c(2)}\hspace*{0.7cm}}$}   & \\
\begin{block}{cc(c|ccc|ccc|c)}
  &\multirow{1}{*}{$c(-1)\updownarrow$} & 2 & 0 & \dots & 0 & -1 & \dots & -1 & \dots \\ \cline{3-10}
    &\multirow{3}{*}{ $ \,\,c(1)\,\,\left\updownarrow\vphantom{\displaystyle\sum_{\substack{i=1\\i=0}}}\right.$}& 0 & -2 & \dots & -2 & -3 & \dots & -3 &   \\
      & & \vdots & \vdots & \ddots & \vdots & \vdots & \ddots & \vdots & \dots  \\
        & & 0 & -2 & \dots & -2 & -3 & \dots & -3 &   \\ \cline{3-10}
         & \multirow{3}{*}{ $\,\,c(2)\,\, \left\updownarrow\vphantom{\displaystyle\sum_{\substack{i=1\\i=0}}}\right.$}  & -1 & -3 & \dots & -3 & -4 & \dots & -4 &   \\
    && \vdots & \vdots & \ddots & \vdots & \vdots & \ddots & \vdots & \dots  \\
        && -1 & -3 & \dots & -3 & -4 & \dots & -4 &   \\ \cline{3-10}
         && \vdots &  & \vdots &  &  & \vdots & \vdots &   \\
\end{block}
\end{blockarray}}\;\;,
\end{equation*}
The block form of $A$ uses the coefficients of $q$-series expansion for the normalized modular function $J(q)$, that is, $c(j)$ denotes  the coefficient of $q^j$ in %the  modular function
$$J(q)=
\sum_{j\geq -1}c(j)q^j= {1 \over q}  + 196884 q + 21493760 q^2 + 864299970 q^3  + \cdots.$$
So $c(-1) = 1$, $c(0) = 0$, $c(1) = 196884$, $\dots $.
We define the index set 
%$$I = \left\{(j,k)\mid j,k\in\Z,\; 1 \leq k\leq c(j) \right\} $$
$$I = \left\{(j,k)\mid j\in\{-1,1,2,3,\dots\},\; 1 \leq k\leq c(j) \right\} $$%}
to reflect the block form of~$A$,
so that %Borcherds generalized Cartan matrix can be written
$$ A = \left( a_{jk,pq} \right)_{(j,k),(p,q)\in {I}}$$ where $a_{jk,pq} = -(j+p)$.
For another approach to this indexing see~\cite{Kang2018}. Note that in the notation $ a_{jk,pq}$, $jk$ and $pq$ are pairs of indices not a product. 

%\begin{comment}
%Let $I$ be a finite or countable index set. 
%A \emph{generalized Cartan matrix} $A = (a_{ij})_{i,j \in I}$ is a matrix over $\mathbb R$ satisfying:
%\begin{align*}
%\tag{C1}\label{Csymm} a_{ij}&=a_{ji},\\
%\tag{C2}\label{Cneg} a_{ij}&\le 0\qquad\text{if $i\ne j$},\\
%\tag{C3}\label{Cint} \frac{2a_{ij}}{a_{ii}}&\in\Z \qquad\text{if $a_{ii}>0$},
%\end{align*}
%for all $i,j\in I$.  Define
%$$m_{ij} := 1-\frac{2a_{ij}}{a_{ii}}$$
%for $i\ne j$ with $a_{ii}>0$.
%
%We define $\mathfrak{g}(A)$ as the Lie algebra over $\C$ with generating set $\{h_i, e_i, f_i\mid i \in I\}$ and defining
%relations:
%\begin{align*}
%\tag{R1}\label{Rhh} [h_i,h_j]&=0,\\
%\tag{R2}\label{Rhe} [h_i,{e}_j]&=a_{ij}e_j ,\\
%\tag{R3}\label{Rhf} [h_i,{f}_j]&=-a_{ij} f_j,\\
%\tag{R4}\label{Ref} [e_i,f_j]&=\delta_{ij} h_i,\\ 
%\tag{R5}\label{Ree} (\ad e_i)^{m_{ij}} \,e_j=(\ad f_i)^{m_{ij}}\, f_j&=0\qquad\text{if $i\ne j$ and $a_{ii}>0$},\\
%\tag{R6}\label{Reeorth} [e_i, e_j]=[f_i,f_j]&=0 \qquad\text{if  $a_{ii}=0$},
%\end{align*}
%for all $i,j\in I$.
%\end{comment}

The \emph{Serre--Chevalley generators} of $\mathfrak{g}(A)$ are 
$e_{jk}$, $f_{jk}$, $h_{jk}$ for all $(j,k)\in  I$, with
 defining relations 
\begin{align*}
\tag{R:1}\label{Rmhh} \left[h_{jk},h_{pq}\right]&=0,\\
\tag{R:2}\label{Rmhe} \left[h_{jk},{e}_{pq} \right]&=a_{jk,pq}{e}_{pq} =-(j+p) {e}_{pq},\\
\tag{R:3}\label{Rmhf} \left[h_{jk},{f}_{pq} \right]&=-a_{jk,pq}{f}_{pq} =(j+p) {f}_{pq},\\
\tag{R:4}\label{Rmef} \left[{e}_{jk},{f}_{pq} \right]&=\delta_{jp}\delta_{kq}h_{jk},\\ 
\tag{R:5}\label{Rmee} \left(\ad {e}_{-1\,1} \right)^j \,{e}_{jk}&= \left(\ad {f}_{-1\,1} \right)^j \,{f}_{jk}=0,
\end{align*}
for all $(j,k),\,(p,q) \in {I}$. 
%From here on, we usually write $e_{-1}:={e}_{-1\,1}$ and $f_{-1}:={f}_{-1\,1}$.

The \emph{Cartan subalgebra} of $\mathfrak{g}(A)$ is $\mathfrak{h}_{A}=\sum_{(j,k)\in I} \C h_{jk}$.
%, and we denote the subalgebra generated by $\left\{e_{jk}\mid (j,k) \in I\right\}$ by $\frak n^+$ and the subalgebra generated by $\left\{f_{jk}\mid  (j,k)\in I\right\}$ by $\nm$. 
The \emph{Cartan involution} $\eta:\mathfrak{g}(A)\to\mathfrak{g}(A)$  acts as~$-1$ on~$\frak h_A$ and interchanges~$e_{jk}$ and~$f_{jk}$. 

The \emph{Monster Lie algebra} is  $\mathfrak{m}=\mathfrak{g}(A)/\frak{z}$, where $\frak z$ is the center of $\frak g(A)$.
Note that $\mathfrak{z}$ is contained in $\mathfrak{h}_A$ and so the Cartan subalgebra of $\mathfrak{m}$ is 
$ \mathfrak{h}:= \mathfrak{h}_A/\mathfrak{z}$.
The Cartan involution induced on $\mathfrak{m}$ is also denoted by~$\eta$. 
%Note that the spaces  $\mathfrak{n}^\pm$  intersect $\mathfrak{h}$ trivially, so we can identify them with subsets of~$\mathfrak{m}$.
The matrix~$A$ has rank~2 and so $\mathfrak{h}$ has dimension~2 (see Proposition~\ref{P-Monster} below for details).

Define the following elements in $\frak m = \frak g(A)/\frak z$:
\begin{align*}
h_1 &:= \frac12(h_{-1\,1}-h_{1\,1})+\mathfrak{z}, & h_2 &:= -\frac12(h_{-1\,1}+h_{1\,1})+\mathfrak{z},\\
e_{-1} &:= {e}_{-1\,1}+\mathfrak{z}, & f_{-1} &:= {f}_{-1\,1}+ \mathfrak{z},
\end{align*}
 and write $e_{jk}$ for $e_{jk}+\mathfrak{z}$ and $f_{jk}$ for ${f}_{jk}+ \mathfrak{z}$, 
 for all $(j,k)\in  I-\left\{(-1,1) \right\}$. 
The following proposition now gives explicit generators and relations for $\mathfrak{m}$.
\begin{proposition}\label{P-Monster} \cite{JurJPAA}
The Serre--Chevalley generators of $\mathfrak{m}$ are
$h_{1}$, $h_2$, $e_{-1}$, $f_{-1}$, and $e_{jk}$, $f_{jk}$ for all
$(j,k)\in  I-\left\{(-1,1) \right\}$,
with defining relations:
\begin{align*}
\tag{M:1}\label{Mhh} \left[h_{1},h_2\right]&=0,\\
\tag{M:2a}\label{Mhe-} \left[h_1,e_{-1} \right] &= e_{-1},&              
              \left[h_2, e_{-1} \right] &= -e_{-1},\\
\tag{M:2b}\label{Mhe} \left[h_1,e_{jk} \right]&= e_{jk},&
              \left[h_2, e_{jk} \right] &= j e_{jk},\\
\tag{M:3a} \label{Mhf-} \left[h_1,f_{-1} \right] &= -f_{-1},&
             \left[h_2,f_{-1} \right] &= f_{-1},\\
\tag{M:3b}\label{Mhf} \left[h_1,f_{jk} \right]&= - f_{jk},&
              \left[h_2,f_{jk} \right] &= -j f_{jk},\\
\tag{M:4a}\label{Me-f-} \left[e_{-1},f_{-1} \right]&=h_1-h_2, \\
\tag{M:4b}\label{Mef-} \left[e_{-1},f_{jk} \right]&=0,  & \left[e_{jk},f_{-1} \right]&=0,\\
\tag{M:4c}\label{Mef} \left[e_{jk},f_{pq} \right]&=-\delta_{jp}\delta_{kq} \left(jh_1 + h_2\right),\\ 
\tag{M:5}\label{Mee}
  \left(\ad e_{-1} \right)^j e_{jk}&=0,&\qquad \left(\ad f_{-1} \right)^j f_{jk}&=0,
\end{align*}
for all $(j,k),\,(p,q) \in  I-\left\{(-1,1) \right\}$. Also ${\frak h} = \C h_1\oplus \C h_2 = \mathfrak{h}_A/\mathfrak{z}$, the Cartan subalgebra of~$\m$.
\end{proposition}

\begin{proof}
Now $\mathfrak{m}=\mathfrak{g}(A)/\mathfrak{z}$ and  the center $\mathfrak{z}$ is spanned by
$(p-j)h_{-1\,1}-(p+1)h_{jk}+(j+1)h_{pq}$.
So 
$h_{-1\,1}+\mathfrak{z}=h_1-h_2$; %\\
$\;h_{1\,k} + \mathfrak{z}=-h_1-h_2$ %\\
for $(1,k)\in I$; and %}
\begin{align*}
h_{jk} + \mathfrak{z} &= h_{jk} + \frac{1}{2}\left((1-j)h_{-1\,1}-(1+1)h_{jk}+(j+1)h_{1\,1} \right) +\mathfrak{z} \\
&=\frac{1-j}{2}\,h_{-1\,1} +\frac{1+j}{2}\, h_{1\,1} +\mathfrak{z} = \frac12\left(h_{-1\,1}+h_{1\,1} \right) + \frac{j}2\left(-h_{-1\,1}+h_{1\,1} \right)+{\frak z}=-jh_1-h_2
%-h_1-jh_2,
\end{align*}
when $j\ne 1$. Hence $h_1$ and $h_2$ span ${\frak h}$. The relations  \eqref{Mhh}--\eqref{Mee} follow from the corresponding relations \eqref{Rmhh}--\eqref{Rmee} for ${\frak g}(A)$. Finally $h_1$ and $h_2$ are linearly independent by \eqref{Mhe-} and \eqref{Mhe}.
\end{proof}

Before defining the root system, we construct an extension of the Lie algebra~$\mathfrak{g}(A)$ in order to ensure that the simple roots form a linearly independent set (see also \cite{JLW}).
The row space of $A$ is spanned by the two rows indexed by~$\{(\pm1,1)\}$.
Thus, for $(j,k)\in I-\left\{(\pm1,1) \right\}$,  let  $D_{jk}$  be the derivation of $\mathfrak{g}(A)$ defined by
$$D_{jk} (e_{pq}) = \delta_{jp}\delta_{kq}e_{jk}, \quad D_{jk} (f_{pq}) = -\delta_{jp}\delta_{kq}f_{jk},\quad D_{jk} (h_{pq})=0,$$
for all $(p,q) \in I$.
The span of the $D_{jk}$ is an abelian Lie algebra of derivations of $\mathfrak{g}(A)$, denoted $\mathfrak{d}_0$. 
Now $\mathfrak{d}_0+\mathfrak{h}_A=\mathfrak{d}_0\oplus\mathfrak{h}_A$ is an abelian subalgebra of 
$\mathfrak{d}_0 \ltimes \mathfrak{g}(A)$.
%The functionals in $\left\{\ad(e_{jk})\mid (j,k)\in I\right\}$ are now
%linearly independent when restricted to $\mathfrak{d}_0\oplus\mathfrak{h}_A$.
We  define the \emph{extended Monster Lie algebra} to be $\mathfrak{g}^e(A)=\mathfrak{g}^e:=\mathfrak{d}_0 \ltimes \mathfrak{g}(A)$,
with Cartan subalgebra $\frak h^e=\frak d_0 \oplus \frak h_A$.

%%%%%%%%%%%%%%%%%%%%%%%%%%%%%%%%%%%%%%%%%%%%%%%%%%%%%%
\subsection{Decompositions of $\mathfrak m$}\label{SS-decomp}
%We now give two triangular decompositions for the Monster Lie algebra: the usual decomposition via roots, and a decomposition into a $\gl_2$ subalgebra and two free Lie algebras. %LC: `$\gl_2$ subalgebra and two free Lie algebras' is not a triangular decomposition in the usual sense. No need to add this sentence.
 
We use $\mathfrak{h}^e\subseteq \mathfrak{g}^e$ to define the roots of $\mathfrak{g}(A)$.  In Subsection~\ref{SS-mroots}, we show
 how to obtain the  root system of $\frak m$.
For $\alpha \in (\frak h^e)^*$, define  
$$\mathfrak{m}_\alpha = \left\{x \in \mathfrak{g}(A)  \mid [h,x] = \alpha (h)x \mbox{ for all } h \in \mathfrak{h}^e \right\}.$$ 
Nonzero elements $\alpha \in (\mathfrak{h}^e)^*$ such that  $\mathfrak{m}_\alpha \neq 0$ are called \emph{roots} of $\mathfrak{g}(A)$; the set of all roots is denoted $\Delta\subseteq (\mathfrak{h}^e)^*$. 
The root space $\mathfrak{m}_{\a}$ is contained in $\mathfrak{g}(A)$, but can be considered to be a subset of $\mathfrak{g}^e$ or $\frak m$,
since $\mathfrak{m}_\a \cap \frak h ^e=\{0\}$.

The {\em simple root} $\alpha_{jk} \in(\mathfrak{h}^e)^*$, for $(j,k)\in I$, is defined by the condition
$\left[h,e_{jk}\right]=\alpha_{jk}(h)e_{jk}$ for all $h \in \mathfrak{h}^e$.
Now $\alpha_{jk}(h_1)=1$, $\alpha_{jk}(h_2)=j$,
and $\alpha_{jk}(D_{pq})=\delta_{jp}\delta_{kq}$ for  $(p,q)\in I-\left\{(\pm1,1) \right\}$, so the simple roots are linearly independent in~$(\mathfrak{h}^e)^*$.

Every root $\a\in\Delta$ is an integral sum $\a=\sum_{(j,k)\in I} a_{jk}\a_{jk}$, so is contained in the 
\emph{root lattice}  $Q := \bigoplus_{(j,k)\in I}\Z\a_{jk}\cong \Z^I$.
We call $\a$ a \emph{positive root} if every $a_{jk}\ge0$ and  a \emph{negative root} if every $a_{jk}\le0$.
The set of all positive (respectively negative) roots is denoted $\Delta_+$ (respectively $\Delta_-$).  
%{\bf{\color{blue}  In Section~\ref{S-roots}, we will identify $\alpha$ with $(m,n)\in II_{1,1}$ - not true!}}
We have subalgebras
$$ \npm:= \bigoplus_{\a\in\Delta_\pm} \mathfrak{m}_\alpha.$$
% generated by $\left\{e_{jk}\mid (j,k) \in I\right\}$ by $\frak n^+$ and the subalgebra generated by $\left\{f_{jk}\mid  (j,k)\in I\right\}$ by $\nm$. 
% $\a\in\left(\mathfrak{h}^e \right)^*$, the root space $\mathfrak{m}_{\a}$ is contained in $\mathfrak{g}(A)$, but can be considered a subset of $\mathfrak{g}^e$ or $\frak m$, since $\mathfrak{m}_\a \cap \frak h ^e=0$. 

Some elementary properties of $\mathfrak{m}$ are summarized 
in the following, which follows from Proposition~1.5 of~\cite{jurisich1996exposition}. 
\begin{proposition} \label{P-all}
The Borcherds algebra $\mathfrak{m}=\mathfrak{g}(A)/\mathfrak{z}$ has the following properties:
\begin{enumerate}
\item $\mathfrak{m}= \frak n^{-} \oplus \frak h \oplus \frak n^{+}= \frak h \oplus \bigoplus_{\alpha\in \Delta}\frak g_{\alpha}$.
\item  $\left\{h_{jk}\mid (j,k)\in I \right\}$ is a basis of ${\frak h_A}$.
\item $\frak n^+ = \bigoplus_{\a \in \Delta_+} \frak m_\a$
   and $\frak n^-  = \bigoplus_{\a \in \Delta_-} \frak m_\a$ .%are pro-nilpotent,
\item $\Delta = \Delta_+ \cup \Delta_-$ (disjoint union).
\item $\left[\frak m_\a, \frak m_\b \right] \subset \frak m_{\a+\b}$ for all $\a,\b \in \left(\mathfrak{h}^e \right)^*$.
\item $\eta \left(\frak m_\a \right) = \frak m_{- \a}$ and $\dim \frak m_\a = \dim \frak m_{-\a} < \infty$ for all $\a \in (\frak h^e)^*$.
\item $\frak m_{\alpha_{jk}} = \C e_{jk}$  and $\frak m_{-\alpha_{jk}} = \C f_{jk}\,$ for all  $(j,k) \in I$.
%\item For each $i\in I_0$ every $\adj \frak u_i$-%
%stable subspace of $\frak g$ is a direct sum of finite-dimensional
%$\frak u_i$-modules. 
\item $\alpha_{jk} \left(h_{pq} \right) =  a_{jk,pq}$ $\text{ for all } (j,k),\,(p,q) \in I$.
\item Every nonzero ideal of $\mathfrak{g}^e$ intersects nontrivially with $\mathfrak{h}^e$  and every nonzero ideal of $\mathfrak{m}$ intersects nontrivially with $\mathfrak{h}$.
%\item If the (possibly infinite) matrix
%$A'$ 
%is obtained from the matrix  $A$ by a
%permutation of the rows and a corresponding permutation of columns, then
%there is a natural isomorphism $\frak g(A) \cong \frak g(A^\prime)$.
%\item If $A$ has the form $\pmatrix B &  0 \\ 0 &  C \endpmatrix$, 
%there is a 
%natural isomorphism $\frak g(A) \cong \frak g(B) \times \frak g(C)$. 
\end{enumerate}
\end{proposition}

%\begin{comment}
%Finally $\mathfrak{g}^e$ has an invariant symmetric bilinear form $(\circ,\circ)$  defined by
%\begin{align*}
%\tag{B1}\label{B-hh}  (h_{i},h_{j}) &= a_{ij},\\
%\tag{B2}\label{B-ef}  ({e}_{i},{f}_{j}) &= \delta_{ij}, \\
%\tag{B3}\label{B-ee-ff}  (e_i,e_j)=(f_i,f_j)&=0,\\
%\tag{B4}\label{B-hef}  (h_i, e_j)=(h_i, e_j)&=0,\\
%\tag{B5}\label{B-dd-hd}  (D_{i},D_{j}) =   (h_i,D_j) &= \delta_{ij},\\
%\tag{B6}\label{B-def}  (D_i ,e_j)=(D_i, f_j)&=0,
%\end{align*}
%for $i,j\in {I}$.  This induces a symmetric bilinear form on $(\frak h^e)^*$ with
%$$(\a_i,\a_j)=a_{ij}.$$
%\end{comment}

Following Section~1.3 of~\cite{jurisich1996exposition}, 
we define a  symmetric bilinear form on $(\mathfrak{h}^e)^*\supseteq Q$ with
$$ \left(\a_{jk},\a_{pq} \right)_{Q}:=a_{jk,pq} = -(j+p) $$
(note that this formula includes $\a_{-1}=\a_{-1\,1}$).
This induces a form on $\mathfrak{h}^e$, which in turn extends to an
invariant symmetric bilinear form $\left(\cdot ,\cdot \right)_{\mathfrak{g}^e}$ on $\mathfrak{g}^e$ satisfying, for all $(j,k)$, $(p,q)\in {I}$, 
\begin{align*}
%\tag{B1} (h_{-1,1},h_{jk}) &= -(j-1),\\
\tag{B:1}\label{B-gAe-hh}
\left(h_{jk},h_{pq} \right)_{\mathfrak{g}^e} &= - (j+p),\\
\tag{B:2}\label{B-gAe-ef}
  \left({e}_{jk},{f}_{pq} \right)_{\mathfrak{g}^e} &= \delta_{jp} \delta_{kq}, \\
\tag{B:3}\label{B-gAe-gg}
\left(\mathfrak{m}_{\a}, \mathfrak{m}_{\b} \right)_{\mathfrak{g}^e}&=0 \quad\text{if $\a,\b\in\Delta$ with $\a+\b\ne0$,}\\
\tag{B:4}\label{B-gAe-hg}
\left(\mathfrak{h}^e, \mathfrak{m}_{\a} \right)_{\mathfrak{g}^e}&=0 \quad\text{if $\a\in\Delta$,}\\
\tag{B:5}\label{B-gAe-dd-hd}
\left(D_{jk},D_{pq} \right)_{\mathfrak{g}^e} =   \left(h_{jk},D_{pq} \right)_{\mathfrak{g}^e} &= \delta_{jp} \delta_{kq},\\
\tag{B:6}\label{B-gAe-def}
  \left(D_{jk} ,e_{pq} \right)_{\mathfrak{g}^e}=\left(D_{jk}, f_{pq} \right)_{\mathfrak{g}^e}&=0.
\end{align*}  
This also induces symmetric invariant bilinear forms on $\mathfrak{g}(A)$ and on $\mathfrak{m}$, denoted $ \left(\circ,\circ \right)_{\mathfrak{g}(A)}$ and $ \left(\circ,\circ \right)_{\mathfrak{m}}$.

We call $\a\in\Delta$ a \emph{real root} if $(\alpha, \alpha)_Q>0$ and an \emph{imaginary root} if $(\alpha, \alpha)_Q\le0$. 
The set of real (respectively imaginary) roots is denoted $\Delta^{\re}$ (respectively $\Delta^{\im}$); 
we also denote the set of  positive real roots (respectively  negative real roots) by $\Delta^{\re}_{\pm}$ and the set of  positive imaginary roots (respectively   negative imaginary roots) by $\Delta^{\im}_{\pm}$.
Hence $\a_{-1}:=\a_{-1\,1}$ is a real root with squared norm $2$, while $\a_{jk}$ is an imaginary root
with squared norm $-2j$ for $(j,k)\in I$ with $j>0$.
Hence  we define
\begin{align*}
I^{\re}&:= \left\{ (j,k)\in I \mid a_{jk,jk}>0 \right\}=\left\{(-1,1) \right\}, \text{ and} \\
I^{\im} &:= \left\{ (j,k)\in I \mid a_{jk,jk}\le0 \right\}= \left\{(j,k)\mid j\in\N,\; 1 \leq k\leq c(j) \right\}.
\end{align*}

We now describe the  decomposition of $\mathfrak m$ as a direct sum of a $\gl_2$ subalgebra and two free subalgebras as in~\cite{JurJPAA}.
Define an \emph{extended index set}
\begin{align*}
 E &= \{(\ell,j,k)\mid( j,k)\in I^{\im},\, 0\leq \ell<j \}\\ &= \{(\ell,j,k)\mid j\in\N,\, 1 \leq k\leq c(j),\, 0\leq \ell<j \}.
 \end{align*}
and set
$$
   e_{\ell,jk}:= \frac{\left(\ad e_{-1} \right)^\ell e_{jk}}{\ell!}\quad\text{and}\quad f_{\ell,jk}:= \frac{\left(\ad f_{-1} \right)^\ell f_{jk}}{\ell!},
$$
for $ (\ell,j,k) \in E$.
%Note that the factorials are only included to simplify the relations in Section~\ref{SS-identities}.
%As mentioned in the previous section, the structure of the root spaces $\frak g_\a$ is not changed by taking a quotient by subalgebras of $\frak h^e$. 
%Hence, the root vectors $e_{\ell,jk}$ and $f_{\ell,jk}$ can be considered as elements of either $\mathfrak{g}^e$, $\frak g(A)$, or $\frak m$.
% It follows that the corresponding roots $\a_{\ell,jk}$ can be viewed as roots of $\frak g(A)$, and also $\frak m$. We will sometimes abuse notation and write $\a_{\ell,jk}$ for a root of $\frak m$, rather than $\overline\a_{\ell,jk}$. Similarly, we may sometimes write $\frak g_\a$ for root spaces of $\frak m$ rather than $\frak m_\a$.
%Define ${\frak g}_{\re}:=\C e_{-1}\oplus\C(h_1-h_2)\oplus\C f_{-1} \cong \mathfrak{sl}_2$ and ${\frak h}_{\im}:=\C(h_1+h_2)$.

\begin{theorem}[\cite{JurJPAA}]\label{T-m-free}
Let $\frak m$ be the Monster Lie algebra and 
let $\mathfrak{gl}_2(-1)$ be the Lie subalgebra generated by $\left\{e_{-1}, f_{-1}, h_{1}, h_2 \right\}$. Then
$$\mathfrak{m} = \nimm \oplus \mathfrak{gl}_2(-1)  \oplus \nimp$$
where
$\mathfrak{gl}_2(-1) %=\mathfrak{g}_{\re}\oplus\mathfrak{h}_{\im}
\cong \mathfrak{gl}_2$,  $\;\nimp$ is a subalgebra freely generated by  $\left\{e_{\ell,jk}\mid (\ell,j,k)\in E \right\}$,  and
$\nimm$ is a subalgebra freely generated by  $\left\{f_{\ell,jk}\mid (\ell,j,k)\in E \right\}$.
\end{theorem}
%
%See \cite{JLW} for an equivalent characterization of $\frak m$ as a maximal local Lie algebra. Page 137 of  \cite{JLW} gives a  construction of $\mathfrak{gl}_2(-1)\cong \mathfrak{gl}_2$ in terms of the vertex operator algebra that~$\frak m$ is constructed from.

For fixed $(j,k)\in I^{\im}$,  note that $\left\{e_{\ell,jk}\right\}_{\ell=0}^{j-1}$ is a basis for a lowest weight $\frak{gl}_2(-1)$-module $V_{jk}^+$ with lowest weight vector $e_{jk}$,  and $\left\{f_{\ell,jk}\right\}_{\ell=0}^{j-1}$ is a basis for a 
highest weight $\frak{gl}_2(-1)$-module $V_{jk}^-$ with highest weight vector $f_{jk}$.

%%%%%%%%%%%%%%%%%%%%%%%%%%%%%%%%%%%%%%%%%%%%%%%%%%%%%%
\subsection{Roots and reflections for ${\mathfrak g}(A)$}\label{SS-gAeroots}
In this subsection, we explicitly describe the root system $\Delta$ of ${\frak g}(A)$ with respect to $\frak h^e$.
Recall that the extended Monster Lie algebra is $\frak g^e=\mathfrak{d}_0 \ltimes \mathfrak{g}(A)$
with Cartan subalgebra $\frak h^e=\frak d_0 \oplus \frak h_A$, and that $\frak g^e$ and $\frak h^e=\frak d_0 \oplus \frak h_A$ were constructed so that the simple roots are linearly independent in $(\mathfrak{h}^e)^*$. In Subsection~\ref{SS-mroots}, we show
 how to obtain the more familiar root system of $\frak m$ in the  lattice $\text{II}_{1,1}$.

For $\a\in\left(\mathfrak{h}^e \right)^*$,  recall that the root space $\mathfrak{m}_{\a}$ can be considered a subset of either $\mathfrak{g}^e$,  $\mathfrak{g}(A)$, or $\frak m$. 
The simple roots ${\alpha}_{jk} \in (\mathfrak{h}^e)^*$ are then defined for $(j,k) \in I$ by the condition
$[h,e_{jk}]=\alpha_{jk}(h)e_{jk}$, for all $h \in \frak h^e$.
%When restricting the roots to a subspace of ${\mathfrak h}^e)^*$, and when considering the action of the roots on $\frak h^e/\frak z$ we continue to refer to the resulting functionals as roots, when no confusion will arise. For example, we wish to consider the action of a root $\alpha$ on $\frak h$, and in turn on $\frak h/\frak z$. Because $\frak h$ is two dimensional, the value of the root $\alpha$ is then determined by it's value on the elements denoted $h_{-1,1}$ and $h_{1,1}$ above, resp. there images in $\frak m$. 
%For convenience we write $\a_{-1}:= \a_{-1\,1}$. % for the only real simple root.
The root lattice of $\mathfrak{g}^e$  is 
$ Q = \Z\a_{-1} \oplus \bigoplus_{(j,k)\in I^{\im} } \Z \a_{jk}\subseteq \left(\mathfrak{h}^e \right)^*$.

We have  one dimensional simple root spaces:
$${\mathfrak m}_{\a_{-1}} = \C {e}_{-1},\quad {\mathfrak m}_{-\a_{-1}} = \C {f}_{-1},\quad {\mathfrak m}_{\a_{jk}} =\C {e}_{jk},\quad {\mathfrak m}_{-\a_{jk}} =\C {f}_{jk}.$$

%{\bf\color{blue} Theorem~\ref{T-m-free} implies that (with abuse of notation)  $\ell \a_{-1}+\a_{jk}$ is a root of $\frak m$ for $(\ell,j,k)\in E$ with}
%Also
%${e}_{\ell,jk}\in {\mathfrak g}_{\ell\a_{-1}+\a_{jk}}  \quad\text{and}\quad 
 %   {f}_{\ell,jk}\in{\mathfrak g}_{-(\ell\a_{-1}+\a_{jk})}. $
    
If bases for  the root spaces of $\mathfrak{g}(A)$ are needed, they may be obtained using Theorem~\ref{T-m-free} and the Shirshov basis \cite{shirshov2009bases} for a free Lie algebra. However, for our purposes, 
a description of the set of all positive imaginary roots suffices:
\begin{lemma}\label{L-rts}  Let $$\alpha=\ell\a_{-1} + \sum_{(j,k)\in I^{\im}} c_{jk} \a_{jk}$$ be an element of the root lattice $Q$.
Then $\alpha\in{\Delta}^{\im}_+$ if and only if 
\begin{enumerate}
\item\label{Delta-pos} $\ell\ge0$ and $c_{jk}\ge0$ for all $(j,k)\in I^{\im}$ with at least one $c_{jk}>0$; 
\item\label{Delta-ell} $\displaystyle \ell \le \sum_{(j,k)\in I^{\im}} (j-1)c_{jk}$; and 
\item\label{Delta-norepeat}   if only one of the coefficients $c_{jk}$ appearing in $\alpha$ is nonzero, then either $c_{jk}=1$ or $ 0<\ell<(j-1)c_{jk}$.
\end{enumerate}
\end{lemma}

\begin{longver}
\begin{proof}  Elements in the Shirshov basis \cite{shirshov2009bases} are Lie products of terms
$$ e_{\ell_1,j_1 k_1}, e_{\ell_2,j_2 k_2},\dots,e_{\ell_n,j_n k_n},$$
for any choice of indices $\left(\ell_1,j_1, k_1 \right),\dots,\left(\ell_n,j_n, k_n \right)$ in $E$,  allowing repeats but  such that the indices are not \emph{all} equal.
Such a basis element is in $\frak g_\a$,  for $\mathfrak{g}=\mathfrak g(A)$, where
$$\a= \left(\ell_1+\cdots+\ell_n \right) \a_{-1}+\a_{j_1k_1}+\cdots+\a_{j_nk_n}.$$
The result now follows, with \eqref{Delta-ell} following from the fact that each $\ell_i\le j_i-1$ and  \eqref{Delta-norepeat} coming from the condition that the indices are not all equal.
\end{proof}
\end{longver}
Note that  
 $$\C e_{\ell,jk}\subseteq{\mathfrak m}_{\ell\a_{-1}+\a_{jk}}  \quad\text{and}\quad 
    \C f_{\ell,jk}\subseteq{\mathfrak m}_{-(\ell\a_{-1}+\a_{jk})}, $$
for $\ell=0,1,\dots, j-1$, but equality only holds for $\ell=0$ or $\ell=j-1$.

Let $\R Q$ be the $\R$-span of the roots and note that the bilinear form on $Q$ extends to a form $(\circ,\circ)_{\R Q}$. 
For a root $\a$ of nonzero norm, we define the \emph{reflection in $\a$} by
$$ w_\a:\quad \R Q \to \R Q, \quad x\mapsto x-\frac{2(\a,x)_{\R Q}}{(\a,\a)_{\R Q}}\a.$$
The reflection $w_{{\a}_{-1}}$ preserves the root lattice $Q$ with
$${w}_{\a_{-1}}\;  :\; \a_{-1}\mapsto-\a_{-1}, \qquad \ell \a_{-1}+\a_{jk}\mapsto (j-\ell-1) \a_{-1}+\a_{jk}.$$
Note that $w_{\a_{-1}}$ also preserves $\Delta$ and $\Delta^{\im}_{\pm}$.
For $q=1,\dots,c(1)$, the reflection $w_{\a_{1q}}$ also preserves $Q$ with
$${w}_{\a_{1q}}\;  :\; \a_{-1}\mapsto \a_{-1}, \qquad \ell \a_{-1}+\a_{jk}\mapsto\ell \a_{-1}+\a_{jk}-(j+1)\a_{1q}.$$
However ${w}_{\a_{1q}}$ does not preserve $\Delta$, since $\ell \a_{-1}+\a_{jk}-(j+1)\a_{1q}$ is never a root, as it has both positive and negative coefficients with respect to the simple roots.
Every other reflection of the form $w_{m\a_{-1}+\a_{pq}}$ fails to preserve $Q$.
We define the \emph{Weyl group} 
as $\langle w_{\a_{-1}}\rangle$, which  has order $2$.

%%%%%%%%%%%%%%%%%%%%%%%%%%%%%%%%%%%%%%%%%%%%%%%%%%%%%%
\subsection{Roots and reflections for $\mathfrak m$}\label{SS-mroots}
The root system $\Delta$ gives a %grading of $\frak m$  with 
decomposition
$$\frak m = \frak h_A \oplus \bigoplus_{\a\in\Delta} \frak m_\a.$$
In this subsection, we compare this with the coarser decomposition given by the roots with respect to  $\frak h:={\frak h}_A/{\frak z} = \C h_1 \oplus \C h_2$.

Recall from~\cite{conway2013sphere} that the unique even 2-dimensional unimodular Lorentzian lattice~$\text{II}_{1,1}$ is the lattice $\Z\oplus\Z$ 
equipped with the bilinear form with  matrix~
$\left(\begin{smallmatrix}  ~0 & -1\\-1 & ~0
 \end{smallmatrix}\right)$.
We have a $\Z$-linear, inner product preserving, surjective \emph{specialization} map $Q \to \text{II}_{1,1}, \a\mapsto\overline\a$ defined by
$$\overline\a_{-1}=(1,-1)\quad\text{and}\quad\overline\a_{jk}=(1,j),$$
(see Corollary 5.4 of \cite{JurJPAA}) and specialized root system $\overline\Delta=\left\{\overline\a\mid\a\in\Delta \right\}$. 
Only finitely many roots in~$\Delta$ specialize to a given element of~$\overline\Delta$ by Lemma~\ref{L-rts}.
%The specialized root lattice is $\Z\oplus\Z$ which is the unique even 2-dimensional unimodular Lorentzian lattice~$\text{II}_{1,1}$,whose elements are  pairs $(m, n)\in\Z\oplus\Z$ with norm $-2mn$ (see~\cite{conway2013sphere})
The specialized root space decomposition for $\frak m$ can  be written
$${\mathfrak m}= {\mathfrak h}\oplus \bigoplus_{\b\in\overline\Delta} {\mathfrak m}_{\b},$$
where
$$\frak m_{\b} = \left\{ x \in {\mathfrak m} \mid [h,x] = \b(h)x \text{ for all $h\in{\mathfrak h}$} \right\} 
= \bigoplus_{\substack{\a\in\Delta\text{ s.t.}\\ \overline\a=\b}} \frak m_\a.$$

For $(\ell,j,k)\in E$, the root $\a_{\ell,jk}:=\ell\a_{-1}+\a_{jk}\in\Delta$ specializes to
$\overline\a_{\ell,jk} = (\ell+1,j-\ell)\in  \text{II}_{1,1}.$
We can now see that the positive specialized root system is
\begin{align*}
  \overline\Delta_{+} &=\left\{(1,-1) \right\} \cup \left\{(x,y)\mid x,y\in \N \right\},
\end{align*} 
and further $\overline\Delta_{-}=-\overline\Delta_{+}$,  $\overline\Delta= \overline\Delta_{+}\cup\overline\Delta_{-}$. 
\begin{figure}
\includegraphics[height=9.5cm]{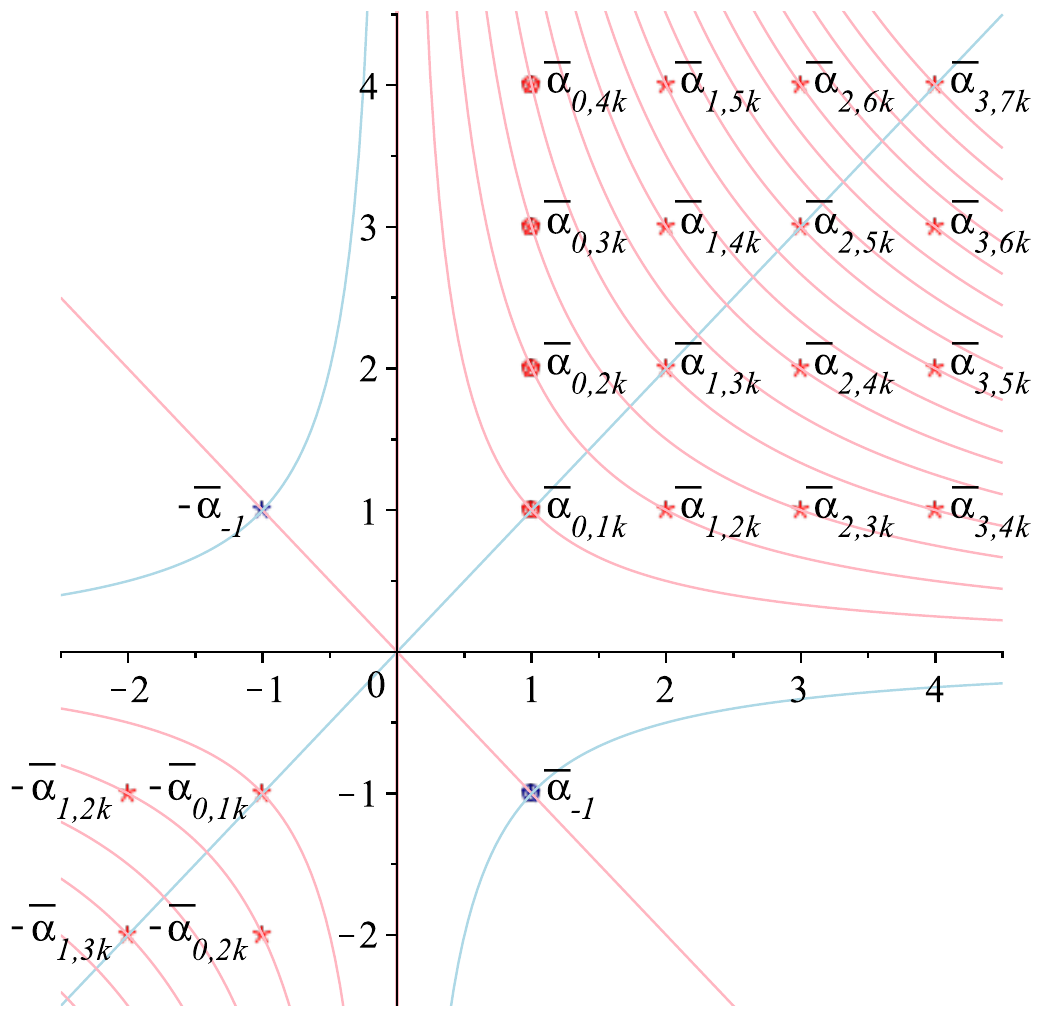} 
\caption{Roots for the Monster Lie algebra $\frak m$}\label{F-rts}
 \end{figure}
See Figure~\ref{F-rts}, where the real roots are denoted by blue asterisks and the imaginary roots by red asterisks; the blue hyperbola indicates elements of squared norm $2$ while the red hyperbolas indicate elements with squared norms in $-2\N$.

It is shown in \cite{BoInvent} that   $\dim(\mathfrak{m}_{(m,n)})=c(mn)$.
This follows from   the denominator identity of this specialized root grading of $\frak m$, which takes the elegant form of a product formula for the modular ~$J$-function:
\begin{equation*}\label{eq:iprod}
u\left(J(u) -J(v) \right) = \prod_{ i\in \N \atop j\in  \{-1\}\cup\N }  \left(1-u^iv^j\right)^{c(ij)}. \end{equation*}

We have  specialized root spaces
\begin{align*}
 \mathfrak{m}_{(1,-1)}& =\C e_{-1}, &
\mathfrak{m}_{(1,j)}&=\bigoplus_{k=1}^{c(j)}\,\C e_{jk}, &
\mathfrak{m}_{(j,1)}&=\bigoplus_{k=1}^{c(j)}\,\C e_{j-1,jk}, \\
 \mathfrak{m}_{(-1,1)}&= \C f_{-1},&
 \mathfrak{m}_{(-1,-j)}&=\bigoplus_{k=1}^{c(j)} \,\C f_{jk},&
 \mathfrak{m}_{(-j,-1)}&=\bigoplus_{k=1}^{c(j)} \,\C f_{j-1,jk}.
 \end{align*}
We also have 
$e_{\ell,jk}\in  %\mathfrak{m}_{ \a_{\ell,j}}=
\mathfrak{m}_{(\ell+1,j-\ell)}$
%\qquad\text{and}\qquad
and $f_{\ell,jk}\in  %\mathfrak{m}_{ -\a_{\ell,j}}=
\mathfrak{m}_{(-\ell-1,-j+\ell)}.$

Now $w_{\a_{-1}}$ induces the reflection $w_{\re}$ %on $ \text{II}_{1,1}$ %$$\left(\begin{array}{rr}  0 & 1\\1 & 0 \end{array}\right),$$
in the blue line in Figure~\ref{F-rts}
and 
$${w}_{\re}\;  :\; \overline\a_{-1}\mapsto-\overline\a_{-1}, \qquad \overline\a_{\ell,jk}\mapsto \overline\a_{j-\ell-1,jk}.$$
For $1\le q\le c(1)$,  $w_{\a_{1q}}$ induces the new reflection $w_{\im}$
%$$\left(\begin{array}{rr}  0 & -1\\-1 & 0 
%\end{array}\right)$
in the red line in Figure~\ref{F-rts}
and 
$${w}_{\im}\;  :\; \overline\a_{-1}\mapsto\overline\a_{-1}, \qquad \overline\a_{\ell,jk}\mapsto -\overline\a_{\ell,jk}.$$
Note that $w_{\im}$ preserves the specialized root system $\overline\Delta$, even though ${w}_{\a_{1q}}$ doesn't preserve $\Delta$.
If $p=2m+1$, then $w_{\a_{m,pq}}$ also induces $w_{\im}$ on the specialized root lattice $ \text{II}_{1,1}$.
In all other cases, the reflection $w_{\a_{m,pq}}$ doesn't preserve the specialized root lattice.
%So we can define the Weyl group of $\frak m$ to be
%$$ W =\left\langle w_{\re},w_{\im} \right\rangle.$$
%{\bf{\color{blue} We are changing the definition of the Weyl group. This requires further comments. It should have a different notation. Are we going to use this larger Weyl group in what follows?}} Note that this is strictly larger than the Weyl group $W^e=\left\langle w_{\a_{-1}} \right\rangle$ for $\mathfrak{g}(A)$.

%%%%%%%%%%%%%%%%%%%%%%%%%%%%%%%%%%%%%%%%%%%%%%%%%%%%%%
\subsection{Root strings and root spans}\label{SS-strings}
In this subsection, we describe some of the properties of root strings in the root system $\Delta$ for the Lie algebra $\mathfrak{g}(A)$ and the spans of certain pairs of roots, which will be needed in future sections.
\begin{proposition}
Every root string of $\Delta$ in the direction of $\a_{-1}$ is  finite.
\end{proposition}

\begin{longver}
\begin{proof}
Let $\a\in\Delta-\left\{\pm\a_{-1} \right\}$.  The root string through $\a$ in the direction of $\a_{-1}$ consists of the roots
of the form $\a+i\a_{-1}\in \Delta$ for $i\in\Z$.
When we specialize to $\overline\Delta$, this becomes $\overline\a+i\overline\a_{-1}= \overline\a+i(1,-1)\in
\overline\Delta$ for $i\in\Z$, which is clearly finite (see Figure~\ref{F-rts}). So the original root string in $\Delta$ is also finite.
\end{proof}
\end{longver}

As before, we set  $\a_{\ell,jk}= \ell\a_{-1}+\a_{jk}$.
We have root strings in $\Delta$ in the direction of $\a_{-1}$:
\begin{align*}
\a_{jk}=\a_{0,jk},\; \a_{1,jk},\;\dots,\; \a_{j-1,jk}, &\quad \text{and}\\
-\a_{j-1,jk}, \;\dots,\; -\a_{1,jk},\;-\a_{jk}=-\a_{0,jk},&
\end{align*}
for $(j,k)\in I^{\im}$.
After specializing to $\mathfrak{m}$, these can be considered as the weights of the $\mathfrak{gl}_2(-1)$-representations 
$V_{jk}^\pm$ spanned by the corresponding weight vectors
\begin{align*}
e_{jk}=e_{0,jk},\; e_{1,jk},\;\dots,\; e_{j-1,jk}, &\quad \text{and}\\
f_{j-1,jk},\; \dots,\; f_{1,jk},\;f_{jk}=f_{0,jk},&
\end{align*}
respectively. %Note that  $e_{j,jk}=f_{j,jk}=0$.
These representations have highest weights $\overline{\a}_{j-1,jk}$ and $-\overline{\a}_{0,jk}$, respectively.
The Weyl group of  $\mathfrak{gl}_2(-1)$ is $\left\langle w_{-1} \right\rangle$ and the reflection $w_{-1}$ reverses these root strings. 

%\textcolor{blue}{We could use other $\gl_2$ subalgebras to describe infinite dimensional representations corresponding to these root strings. I think the roots would have fractional coefficients wrt the fundamental weights of our $\gl_2$ subalgebras. Does this help to explain the actions of the reflections that don't preserve the root lattice?}

Given any root system $\Delta$ and $\a,\b\in\Delta$, the \emph{root span} is $\Sigma(\a,\b)=\bigcup_{n=1}^\infty \Sigma_n$, where we inductively define 
\begin{align*}
 \Sigma_1&= \{\a+\b\}\cap\Delta; &
 \Sigma_{n} &= \left\{\g+\a,\g+\b \mid \g\in\Sigma_{n-1} \right\}\cap\Delta.
\end{align*}
More generally, if $A\subseteq\Delta$, then $\Sigma(A)=\bigcup_{n=1}^\infty \Sigma_n$, where 
$\Sigma_1= \{\a+\b\mid \a,\b\in A\}\cap\Delta$ and
$\Sigma_{n} = \left\{\g+\a \mid \g\in\Sigma_{n-1}, \a\in A \right\}\cap\Delta$.
Note that $\Sigma(\a,\b)=\emptyset$ if and only if $\a+\b\notin\Delta$.
%In general, $\Sigma(\a,\b)\subseteq \left\{a\a+b\b \mid a,b\in\N \right\}\cap\Delta$. 
%This inclusion is always an equality for $\Delta$ finite, but we will see in the following proposition that we can get strict inequality in our root system.
Also note that $\Sigma(\b,\a)= \Sigma(\a,\b)$ and $\Sigma(-\a,-\b)=-\Sigma(\a,\b)$.
The following classification can be proved inductively by repeated application of Lemma~\ref{L-rts}:
\begin{proposition}\label{P-Sigma}
 Let $\Delta$ be the root system for the Lie algebra $\mathfrak{g}(A)$. Then we have the following.
\begin{enumerate}
\item $\displaystyle  \Sigma\left(\a_{-1},\a_{-1} \right)= \Sigma\left(\a_{-1},-\a_{-1} \right)=\emptyset$.
\item $\displaystyle  \Sigma\left(\a_{\ell,jk},\a_{-1} \right) = \begin{cases} 
 \{(a\ell+b)\a_{-1} + a\a_{jk} \mid&\!\! a,b\in\N, b<a(j-\ell)\}, \\
  &\text{\;\;if $0\le\ell<j-1$,}\\ 
  \emptyset&\text{\;\;if $\ell=j-1$.}\end{cases}$
\item $\displaystyle  \Sigma\left(\a_{\ell,jk},-\a_{-1} \right) = \begin{cases} 
 \{(a\ell-b)\a_{-1} + a\a_{jk} \mid&\!\! a,b\in\N, b\le a\ell\},\\ &\text{\;\;if $0<\ell<j$,}\\ 
  \emptyset&\text{\;\;if $\ell=0$.}\end{cases}$
\item $\displaystyle \Sigma\left(\a_{\ell,jk},\a_{m,pq} \right) =\left\{(a\ell+bm)\a_{-1}+a\a_{jk}+b\a_{pq} \mid a,b\in\N \right\}$.
\item $\displaystyle \Sigma\left(\a_{\ell,jk}, -\a_{m,jk} \right) =$ \\  $\begin{cases} 
%  \{(a\ell-bm)\overline\a_{-1}+a\overline\a_{jk}-b\overline\a_{pq} \mid a,b\in\N \} & \text{if $0<\ell-m<j-p$ or $0<m-\ell<p-j$},\\
 \{\a_{-1}\}\cup\{(x\ell+y)\a_{-1} + x\a_{jk}\mid &\!\!\!\! x\in\Z, y\in\N, y<x(j-\ell)\},\\ & \text{\;\;if $\ell-m=1$},\\ 
 \{-\a_{-1}\}\cup\{(x\ell-y)\a_{-1} + x\a_{jk}&\!\!\!\mid x\in\Z, y\in\N, y\le x\ell\},\\& \text{\;\;if $\ell-m=-1$},\\
  \emptyset & \text{\;\;otherwise.} \end{cases}$
\item  $\displaystyle \Sigma\left(\a_{\ell,jk}, -\a_{m,pq} \right)=\emptyset\;$ if $(j,k)\ne(p,q)$.
\end{enumerate}
\end{proposition}
\begin{corollary}\label{C-Sigma}
$\Sigma(\pm\a_{0,2k},\mp\a_{1,2k})=\{\mp\a_{-1}\}
$
but, in all other cases,
the root span of two roots of the form $\pm\a_{-1}$ or $\pm\a_{\ell,jk}$ is either empty or infinite.
\end{corollary}
As we see in Theorem~\ref{T-commrel}, $\Sigma(\a,\b)$ is the set of roots needed to compute the group commutator of the pro-unipotent generators corresponding to $\a$ and $\b$.

We contrast this situation with real roots $\a,\b$ of a Kac--Moody root system, where the set 
$$S(\a,\b):=\left\{a\a+b\b \mid a,b\in\N \right\}\cap\Delta$$ is finite if and only if $\a$ and $\b$ are a prenilpotent pair,
and furthermore $S(\a,\b)=\Sigma(\a,\b)$.
But $\Sigma(\a,\b)$ can be strictly smaller 
than $S(\a,\b)$ in the Monster root system $\Delta$
(or, indeed, in a Kac--Moody root system if we consider imaginary roots $\a$ and $\b$).
For example, $\Sigma\left(\a_{-1},\a_{j-1,jk} \right)$ is empty while $S(\a_{-1},\a_{j-1,jk})$ is infinite.

%%%%%%%%%%%%%%%%%%%%%%%%%%%%%%%%%%%%%%%%%%%%%%%%%%%%%%
\subsection{Positive root systems}\label{SS-pos}
We now define positive subsystems of a root system $\Delta$, following~\cite{morita80} and~\cite{loosneher11}:
\begin{defn}
We say $\Pi\subseteq\Delta$ is a \emph{positive subsystem} if $\Pi\cap(-\Pi)=\emptyset$ and
$\a+\b\in\Pi$ whenever $\a,\b\in\Pi$ and $\a+\b\in\Delta$.
A root $\g\in\Pi$ is called \emph{indecomposable (with respect to $\Pi$)} if it cannot be written as a 
sum $\g=\a+\b$ for $\a,\b\in\Pi$.
If $A\subseteq \Delta$ and $\Sigma(A)\cap\Sigma(-A)=\emptyset$, then $\Sigma(A)$ is a positive subsystem and is called
the positive subsystem \emph{generated by $A$}.
\end{defn}

Note that a indecomposable root is just a minimal element of $\Pi$ with respect to the 
partial ordering with $\g\prec\a$ when $\a-\g\in \Pi$. 
If $\Pi$ is finite, then the set of indecomposable roots are just the simple roots corresponding to $\Pi$.
If $\Pi$ is infinite, then the indecomposable roots may not generate $\Delta$.
For example, in the specialized root system $\overline{\Delta}$, the set
$$ \Pi := \{ (a,b)\in \Delta_{\im} \mid \sqrt{2}a<b\}$$
is easily shown to be a positive root system containing no indecomposable roots.
We say that $\Pi$ is \emph{based} if it is generated by the set of indecomposable roots with respect to $\Pi$.
If $\Pi$ is based then the indecomposable roots are called \emph{simple} roots.

\begin{lemma}
If a positive subsystem $\Pi$ is a finitely generated, then it is based.
In particular, if $\a,\b\in\Delta$ with $\a\ne q\b$ for $q\in\Q$,
then $\Sigma(\a,\b)$ is a based positive subsystem with simple roots $\a$ and $\b$.
\end{lemma}

Not every based positive subsystem is finitely generated however.
For example, $\Delta^+$ is a based positive subsystem with simple roots
$\{\a_{jk} \mid (j,k)\in I\}$, and  $\{-\a_{-1}\}\cup\Delta^+_{\im}$ is a positive root system with simple roots
$\{-\a_{-1},\a_{jk} \mid (j,k)\in I^{\im}\}$.

If $\Pi$ is a positive root system, then $-\Pi$ is also a positive root system.
For every positive root system, we have a corresponding subalgebra
$$
  \mathfrak{n}^\Pi := \bigoplus_{\a\in\Pi} \mathfrak{m}_\a,
$$
and decomposition
$$
  \m^\Pi := \mathfrak{n}^{-\Pi}\oplus\mathfrak{h}\oplus \mathfrak{n}^\Pi .
$$
Note that $\m^{\Delta^+}=\m$, but $\m^{\Sigma(\a,\b)}$ is strictly contained in $\m$.

%\textcolor{red}{\bf[I don't think we need this result any more.]}
%We need the following straightforward results for constructing positive root systems:
%\begin{lemma}\label{L-posrtfunctional}
%If $\pi : \R Q \to \R$ is in the restricted dual of $\R Q$, then $\pi^{-1}(0)\cap{\Delta}=\emptyset$, then 
%$\Delta_\pi:=$\left\{\a\in\Delta\mid \pi({\a})>0 \right\}$
%is a positive root system.
%\end{lemma}

%\textcolor{red}{\bf[This result is now obvious by taking %$\Pi=\Sigma(\a,\b)$]}
%\begin{lemma}\label{L-posrtpair}
%If $\a,\b\in\Delta$ with $\a\ne -q\b$ for $q\in\Q^+$,
%then there is a positive root system containing $\a$ and~$\b$.
%\end{lemma}

%%%%%%%%%%%%%%%%%%%%%%%%%%%%%%%%%%%%%%%%%%%%%%%%%%%%%%
\subsection{Further identities in $\mathfrak{m}$}\label{SS-identities}
In this section, we give some basic identities in $\mathfrak{m}$ that are needed in the rest of the paper.
\begin{theorem}\label{T-identities-v2}
The following identities hold in  $\mathfrak{m}$ for all $(\ell,j,k), (m,p,q)\in E$. \\
%Relations in $\mathfrak{gl}_2(-1)$:
%\begin{align*}
%\tag{$\gl_2$-L1e}\label{gl2-L-h1h2}
%  [h_1,h_2]&=0,\\
%\tag{$\gl_2$-L4a}\label{gl2-L-h1e}
%  [h_1,e_{-1}] &=  e_{-1},\\
%\tag{$\gl_2$-L4b}\label{gl2-L-h2e}
%  [h_2, e_{-1}] &= -e_{-1},\\
%\tag{$\gl_2$-L4c}\label{gl2-L-h1f}
%  [h_1,f_{-1}] &= - f_{-1},\\
%\tag{$\gl_2$-L4d}\label{gl2-L-h2f1}
%  [h_2,f_{-1}] &= f_{-1},\\
%\tag{$\gl_2$-L5}\label{gl2-L-ef}
%  [e_{-1},f_{-1}]&=h_1-h_2,\\
Relations between the free subalgebras $\nimp$ and $\nimm$:
\begin{align*}
\tag{L:1}\label{L-ef0}
  \left[ e_{\ell,jk}, f_{m,pq}\right]&= 0   \qquad\text{ if $j\neq p$, $k\neq q$, or $| \ell- m|>1$}, \\
\tag{L:2}\label{L-efh} 
  \left[e_{\ell,jk}, f_{\ell,jk} \right] &
   %= ((ad e_{-1} )^\ell e_{jk} , (ad f_{-1} )^\ell f_{jk}   ) ( \ell h_{-1} +h_{jk} ) 
  =  (-1)^{\ell+1} \binom{j-1}{\ell}  \left( (j-\ell) h_1 +(\ell+1)h_2 \right), \\  %\frac{1-j}{2}h_{-1} + \frac{1+j}{2}z
\tag{L:3a}\label{L-efe-}
  \left[e_{\ell,jk}, f_{\ell-1,jk} \right]&= (-1)^{\ell+1}\binom{j-1}{\ell}\,\ell e_{-1} \qquad\text{if $\ell\ne 0$},\\
%\text{?????} (-1)^{\ell+1}(j+\ell-1) \frac{(\ell+1)!(j-1)!}{(j-\ell-1)!}\,  e_{-1} \\
\tag{L:3b}\label{L-eff-}
  \left[e_{\ell-1,jk}, f_{\ell,jk} \right]&=(-1)^{\ell} \binom{j-1}{\ell}\,\ell f_{-1} \qquad\text{if $\ell\ne 0$},\\
\intertext{\it Relations between the root vectors of $\mathfrak{gl}_2(-1)$  and the free subalgebras:}
\tag{L:4a}\label{L-e-e}
  \left[e_{-1},  e_{\ell,jk} \right]&= \begin{cases}  (\ell+1)e_{\ell+1,jk} &\text{if $0\le\ell<j-1$,}\\ 0&\text{if $\ell=j-1$,}\end{cases} \\ 
%&\ (\ad\ e_{-1})^{j-\ell}e_{\ell,jk}=0\\
\tag{L:4b}\label{L-f-e}
  \left[f_{-1}, e_{\ell,jk} \right]&=\begin{cases}0 &\text{if $\ell=0$,}\\ (j-\ell) e_{\ell-1,jk}&\text{if $0<\ell< j$,}\end{cases} \\
\tag{L:4c}\label{L-f-f}
  \left[f_{-1},  f_{\ell,jk} \right]&= \begin{cases} (\ell+1)f_{\ell+1,jk} &\text{if $0\le\ell<j-1$,}\\ 0&\text{if $\ell=j-1$,}\end{cases} \\ 
%&\ (\ad\ f_{-1})^{j-\ell}f_{\ell,jk}=0\\
\tag{L:4d}\label{L-e-f}
  \left[e_{-1}, f_{\ell,jk} \right]&= \begin{cases}0 &\text{if $\ell=0$,}\\ (j-\ell)  f_{\ell-1,jk}&\text{if $0<\ell< j$,}\end{cases} \\
\intertext{{\it Relations between the toral subalgebra $\frak h$ and the free subalgebras:}}
\tag{L:5a}\label{L-h1e}
  \left[h_1, e_{\ell,jk} \right]&=(\ell+1) e_{\ell,jk},\\
\tag{L:5b}\label{L-h1f}
  \left[h_1,f_{\ell,jk} \right]&=-(\ell+1) f_{\ell,jk},\\
\tag{L:6a}\label{L-h2e}
  \left[h_2,  e_{\ell,jk} \right]&= (j-\ell) e_{\ell,jk},\\
\tag{L:6b}\label{L-h2f}
  \left[h_2,  f_{\ell,jk} \right]&= -(j-\ell) f_{\ell,jk}.
\end{align*}
\end{theorem}

\begin{longver}
Note that relations \eqref{L-ef0}, \eqref{L-efh}, \eqref{L-efe-}, and \eqref{L-eff-} can be summarized as
\begin{align*}
 & \left[ e_{\ell,jk}, f_{m,pq} \right]=\\ &\;\; \delta_{jp}\delta_{kq} (-1)^{\ell+1} \binom{j-1}{\ell}  \big(  \delta_{\ell m} \left((j-\ell) h_1 +(\ell+1)h_2 \right)
  -\delta_{\ell-1,m} \ell e_{-1} +\delta_{\ell+1,m} (\ell+1)(j-\ell-1) f_{-1}\big). 
\end{align*}
\end{longver}

\begin{proof}
\begin{shortver}
Most of these relations are straightforward computations. We will give one example here.
\end{shortver}

Relation \eqref{L-ef0} follows from the corresponding relation in $\mathfrak{g}^e$. To see this, 
suppose $$0\ne\left[\e_{\ell,jk},  \f_{m,pq} \right]\in\mathfrak{m}_{\overline\a}$$ for $\overline\a:=(\ell-m )\overline\alpha_{-1} + \overline\alpha_{jk} - \overline\alpha_{pq}$. 
Since the $\overline\alpha_{jk}$ are simple roots of $\overline\Delta$, these coefficients must be  all positive, all negative, or all zero (recalling that  $\mathfrak{m}_0=\mathfrak{h}$).
So we must have $(j,k)=(p,q)$. Now $\overline\a=(\ell-m )\overline\alpha_{-1}$ and so  $\mathfrak{m}_{\overline\a}\ne0$ only if $|\ell-m| \le 1$.
\begin{longver}

Relations \eqref{L-h1e}, \eqref{L-h1f}, \eqref{L-h2e}, and \eqref{L-h2f} follow from the root space decomposition in Subsection~\ref{SS-mroots}.

Relations \eqref{L-e-e} and \eqref{L-f-f} are true by definition of $e_{\ell,jk}$ or $f_{\ell,jk}$, together with \eqref{Rmee}.  \eqref{L-e-f} follows from the fact that  $f_{jk}$ is a highest weight vector for an irreducible representation of $\mathfrak {gl}_2$ of dimension $j$. Similarly, \eqref{L-f-e} follows from the fact that $e_{ik}$ is a lowest weight vector. 

By properties of weight spaces,  $\left[f_{-1}, \left(\ad e_{-1} \right)^\ell e_{jk} \right]= a_{\ell, j}\left(\ad e_{-1} \right)^{\ell-1} e_{jk}$ and $\left[e_{-1}, \left(\ad f_{-1} \right)^\ell f_{jk} \right]= a'_{\ell, j}\left(\ad f_{-1} \right)^{\ell-1} f_{jk}$, for some $\a_{\ell,j}\in\Z$
with $a_{0,j}=a'_{0,j}=0$. Now
\begin{align*}
 \left[f_{-1}, \left(\ad e_{-1} \right)^\ell e_{jk} \right] &=  \left[f_{-1}, \left[e_{-1},\left(\ad e_{-1} \right)^{\ell-1} e_{jk} \right] \right]\\
 &=- \left[e_{-1},\left[\left(\ad e_{-1} \right)^{\ell-1} e_{jk},f_{-1} \right] \right] - \left[\left(\ad e_{-1} \right)^{\ell-1} e_{jk}, \left[f_{-1},e_{-1} \right] \right]\\
 &= a_{\ell-1,j}\left[e_{-1},\left(\ad e_{-1}\right)^{\ell-2} e_{jk} \right]  + \left[\left(\ad e_{-1} \right)^{\ell-1} e_{jk}, h_1-h_2 \right]\\
 &=\left(a_{\ell-1,j} -(\ell-1+1) +(j-\ell+1)\right)\left(\ad e_{-1} \right)^{\ell-1} e_{jk}.
\end{align*}
So $a_{\ell, k} = a_{\ell-1,j} +j-2\ell+1$ and by induction
$$ a_{\ell,k} = \sum_{p=1}^\ell(j-2p+1)  = \ell j -\ell(\ell+1)+\ell = \ell(j-\ell).$$
Hence
$$\left[f_{-1}, e_{\ell,jk} \right] = \frac{1}{\ell!} \left[f_{-1}, \left(\ad e_{-1} \right)^{\ell-1} e_{jk} \right] =  \frac{\ell(j-\ell)}{\ell!} \left(\ad e_{-1} \right)^{\ell-1} e_{jk}= (j-\ell)e_{\ell-1,jk}.$$
Similarly
\begin{align*}
 \left[e_{-1},\left(\ad f_{-1} \right)^\ell f_{jk} \right] &=  \left[e_{-1}, \left[f_{-1},\left(\ad f_{-1} \right)^{\ell-1} f_{jk} \right] \right]\\
 &=- \left[f_{-1},\left[\left(\ad f_{-1}\right)^{\ell-1} f_{jk},e_{-1} \right] \right] - \left[\left(\ad f_{-1} \right)^{\ell-1} f_{jk}, \left[e_{-1},f_{-1} \right] \right]\\
 &= a'_{\ell-1,j}\left[f_{-1},\left(\ad f_{-1} \right)^{\ell-2} f_{jk} \right]  - \left[\left(\ad f_{-1} \right)^{\ell-1} f_{jk}, h_1-h_2 \right]\\
 &=\left(a'_{\ell-1,j} -(\ell-1+1)+(j-\ell+1)\right)\left(\ad f_{-1} \right)^{\ell-1} f_{jk},
\end{align*}
so $a'_{\ell,j}=a_{\ell,j}$.

For \eqref{L-ef0}, we use the following~\cite[p.\ 534]{MP95}: if $\overline\a=\sum n_{jk}\overline\a_{jk}$, 
$a\in\mathfrak{m}_{\overline\a}$ and $b\in\mathfrak{m}_{-\overline\a}$, then $[a,b] = (a,b)\sum n_{jk}h_{jk}$. So
\begin{align*}
\left[e_{\ell,jk},f_{\ell,jk} \right] &=\left[\e_{\ell,jk},\f_{\ell,jk} \right]+\mathfrak{z}=\left(\e_{\ell,jk},\f_{\ell,jk} \right) \left(\ell  h_{-11}+ h_{jk}+\mathfrak{z} \right) 
=-\left(e_{\ell,jk},f_{\ell,jk} \right)\,\left((j-\ell) h_1 + (\ell+1)h_2\right)\\
&=-\frac{\left(\left(\ad e_{-1} \right)^\ell e_{jk},(\ad f_{-1})^\ell f_{jk}\right)}{(\ell!)^2}\,\left((j-\ell) h_1 + (\ell+1)h_2\right).
\end{align*}
So we need to find $b_{\ell,j}:= \left((\ad e_{-1})^\ell e_{jk},(\ad f_{-1})^\ell f_{jk}\right)$. Now $b_{0,j}=1$ by \eqref{B-gAe-ef}, and for $\ell>0$
\begin{align*}
b_{\ell,j} &= \left( \left[e_{-1},(\ad e_{-1})^{\ell-1} e_{jk} \right],\left(\ad f_{-1} \right)^\ell f_{jk}\right) = - \left( \left[\left(\ad e_{-1} \right)^{\ell-1} e_{jk},e_{-1} \right],\left(\ad f_{-1} \right)^\ell f_{jk}\right)\\
 &=- \left( \left(\ad e_{-1} \right)^{\ell-1} e_{jk},\left[e_{-1},\left(\ad f_{-1} \right)^\ell f_{jk}\right]\right) =- \ell(j-\ell) b_{\ell-1,j} .
\end{align*}
%Alternative:
%\begin{align*}
%b_{\ell,j} &= (e_{\ell,jk},[f_{-1},f_{\ell-1,jk}]) = ( [e_{\ell,jk},f_{-1}],f_{\ell-1,jk})\\
% &=-  ( [f_{-1},e_{\ell,jk}],f_{\ell-1,jk})= -\ell(j-\ell) b_{\ell-1,j} .
%\end{align*}
By induction,
$$b_{\ell,j} := \prod_{p=1}^\ell (-p)(j-p) =(-1)^\ell \frac{\ell!(j-1)!}{(j-\ell-1)!}.$$
Hence
$$\left[e_{\ell,jk},f_{\ell,jk} \right]=-(-1)^\ell \frac{\ell!(j-1)!}{(\ell!)^2(j-\ell-1)!}\,\left((j-\ell) h_1 + (\ell+1)h_2\right) 
  =(-1)^{\ell+1}\binom{j-1}{\ell}  \left( (j-\ell) h_1 +(\ell+1)h_2 \right).$$

For \eqref{L-efe-} and  \eqref{L-eff-}, suppose $\left[e_{\ell,jk}, f_{\ell-1,jk} \right]= c_{\ell, j}e_{-1}$.
Then $c_{1,j} =j-1$ since
\begin{align*}
\left[e_{1,jk}, f_{0,jk} \right]&= \left[\left[e_{-1},e_{jk} \right],f_{jk} \right] =-\left[\left[e_{jk},f_{jk} \right],e_{-1} \right] - \left[\left[f_{jk},e_{-1} \right],e_{jk} \right]\\
 &= \left[jh_1+h_2,e_{-1} \right] - 0 = (j-1)e_{-1}.
 \end{align*}
 And
\begin{align*}
\left[e_{\ell,jk}, f_{\ell-1,jk} \right]&= \frac{1}{\ell}\left[\left[e_{-1},e_{\ell-1,jk} \right],f_{\ell-1,jk} \right]\\
&=-\frac{1}{\ell}\left[\left[e_{\ell-1,jk},f_{\ell-1,jk} \right],e_{-1} \right] -\frac{1}{\ell} \left[\left[f_{\ell-1,jk},e_{-1} \right], e_{\ell-1,jk} \right]\\
&=(-1)^{\ell}\binom{j-1}{\ell}\frac{1}{\ell}\left[(j-\ell)h_1+(\ell+1)h_2, e_{-1} \right] + \frac{j-\ell}{\ell}\left[f_{\ell-2,jk}, e_{\ell-1,jk} \right]\\
&=(-1)^{\ell}\binom{j-1}{\ell}\frac{1}{\ell}\left((j-\ell)e_{-1} -(\ell+1)e_{-1} \right)-c_{\ell-1,jk}\frac{j-\ell}{\ell} e_{-1}.
\end{align*}
Hence
\begin{align*}
c_{\ell,j} &=(-1)^{\ell} \binom{j-1}{\ell}\frac{j-2\ell-1}{\ell}  -c_{\ell-1,jk} \frac{j-\ell}{\ell}
\end{align*}
By induction 
\begin{align*}
c_{\ell,j} &=(-1)^{\ell}\binom{j-1}{\ell}\ell.
\end{align*}

Similarly, suppose $\left[e_{\ell,jk}, f_{\ell+1,jk}\right]= c'_{\ell, j}f_{-1}$.
Then $c'_{1,j} =-(j-1)$ since
\begin{align*}
\left[e_{0,jk}, f_{1,jk} \right]&= \left[e_{jk},\left[f_{-1},f_{jk} \right] \right] =-\left[f_{-1},\left[f_{jk},e_{jk} \right] \right] - \left[f_{jk},\left[e_{jk},f_{-1} \right] \right]\\
 &= -\left[f_{-1},jh_1+h_2 \right] - 0 =- (j-1)f_{-1}.
 \end{align*}
 And
\begin{align*}
\left[e_{\ell-1,jk}, f_{\ell,jk} \right]&= \frac{1}{\ell}\left[e_{\ell-1,jk},\left[f_{-1},f_{\ell,jk} \right] \right]\\
&=-\frac{1}{\ell}\left[f_{-1},\left[f_{\ell-1,jk},e_{\ell-1,jk} \right] \right] -\frac{1}{\ell} \left[f_{\ell-1,jk}, \left[e_{\ell-1,jk},f_{-1} \right] \right]\\
&=-(-1)^{\ell}\binom{j-1}{\ell}\frac{1}{\ell}\left[(j-\ell)h_1+(\ell+1)h_2, e_{-1} \right] + \frac{j-\ell}{\ell}\left[f_{\ell-1,jk}, e_{\ell-2,jk} \right]\\
&=-(-1)^{\ell}\binom{j-1}{\ell}\frac{1}{\ell}\left((j-\ell)e_{-1} -(\ell+1)e_{-1} \right)+c'_{\ell-1,jk}\frac{j-\ell}{\ell} e_{-1}.
\end{align*}
Hence $c'_{\ell, j} =-c_{\ell, j}$.
\end{longver}
\end{proof}

%%%%%%%%%%%%%%%%%%%%%%%%%%%%%%%%%%%%%%%%%%%%%%%%%%%%%%
%%%%%%%%%%%%%%%%%%%%%%%%%%%%%%%%%%%%%%%%%%%%%%%%%%%%%%
\section{Complete pro-unipotent group associated to $\mathfrak m$}\label{S-unip}

In this section, we construct a group $\Uhimp$ associated to the imaginary roots of $\mathfrak m$. We construct  $\Uhimp$ explicitly  as a group of smearing  automorphisms of the positive imaginary subalgebra of a completion~$\widehat{\frak m}$ of $\frak m$. 

In Subsection~\ref{SS-pro}, we show that the group $\Uhimp$ is naturally pro-nilpotent and pro-unipotent.
In Subsection~\ref{SS-PositiveRootSys}, we describe complete pro-unipotent groups for some other completions of~$\mathfrak{m}$.

Complete unipotent groups for  Kac--Moody algebras were constructed in \cite[Section~1.9]{rousseau2012almost}, and in \cite[Section~6.2]{kumar2012kac} in terms of pro-representations.

%%%%%%%%%%%%%%%%%%%%%%%%%%%%%%%%%%%%%%%%%%%%%%%%%%%%%%
\subsection{Lie algebra gradings and pro-summability}\label{SS-graded}
% Graded Lie algebras and topology
Fix a lattice  $\Lambda$ (that is, a free $\Z$-module).
Recall that a $\C$-Lie algebra $\mathfrak{L}$ is \emph{$\Lambda$-graded} if
$$\mathfrak{L} =\bigoplus_{a\in\Lambda} \mathfrak{L}_a,$$
with $\left[\mathfrak{L}_a,\mathfrak{L}_b\right]\subseteq \mathfrak{L}_{a+b}$  and $\dim \mathfrak{L}_a$ finite for all $a,b\in\Lambda$.
Each component $\mathfrak{L}_a$ has the topology of a finite dimensional $\C$-vector space, which induces the product topology on 
$\mathfrak{L}$.
%This in turn induces topologies on $\End(\mathfrak{L})$ and $\Aut(\mathfrak{L})$. 
Given $a\in\Lambda$, the map $\phi\in\End(\mathfrak{L})$ is called \emph{graded of degree $a$} if
$$\phi\left(\mathfrak{L}_b \right)\subseteq \mathfrak{L}_{b+a}$$
for all $b\in\Lambda$.
Suppose $\Pi$ is another $\Z$-module  and  $\lambda: \Lambda\to\Pi$ is a homomorphism with the property that, for each $p\in\Pi$, 
$\mathfrak{L}_a=0$ for all but finitely many $a\in \lambda^{-1}(p)$.
Then we get a $\Pi$-grading 
$\mathfrak{L} =\bigoplus_{p\in\Pi} \mathfrak{L}_p$ by setting
$$ \mathfrak{L}_{p} = \bigoplus_{a\in\lambda^{-1}(p)} \mathfrak{L}_a,$$
for all $p\in\Pi$. 

% Grading of the Monster Lie algebra
The Monster Lie algebra $\frak m$ is $Q$-graded as shown in Proposition~\ref{P-all}.
Let $\lambda$ be the $\Z$-linear map $\lambda: Q \to \Z$ defined by
$\lambda(\a)=a+b$ where $\overline{\a}=(a,b)$ is the specialization of $\a$.
So $\lambda$ induces a $\Z$-grading of $\frak m$ with components
$$\mathfrak{m}_k := \bigoplus_{\a\in\lambda^{-1}(k)\cap\Delta} {\mathfrak m}_{\a}$$
for $k\ne0$ and $\mathfrak{m}_0:=\mathfrak{gl}_2(-1)$.
From Lemma~\ref{L-rts}, it follows that $\Delta_k:= \lambda^{-1}(k)\cap\Delta$ is finite. This ensures that each $\mathfrak{m}_k$ is finite dimensional.
This $\Z$-grading is compatible with the decomposition
$$\frak m = \nimm \oplus \frak{gl}_2(-1)  \oplus \nimp,$$
where
$$\nimpm := \bigoplus_{\a\in\Delta^{\im}_\pm} \frak m_\a = \bigoplus_{k\in\pm\N} \frak m_k.$$

\begin{lemma} If $x\in\frak m_\a$, then $\ad(x)$ is a graded endomorphism of degree $\a$.
If~$x\in\frak m_k$, then $\ad(x)$ is a graded endomorphism of degree $k$.
\end{lemma}

We define the \emph{(positive) formal completion} of $\frak m$ to be
$$\widehat{\frak m} := \nm \oplus \frak{h} \oplus \nhp,$$
where
$$\widehat{\frak n}^+ := \prod_{\a\in\Delta_+} \frak m_\a.$$

Equivalently
$$\widehat{\frak m} = \nimm \oplus \gl_2(-1) \oplus \nhimp,$$ 
where
$$ \nhimp=  \prod_{\a\in\Delta^+} \frak m_{\a}= \prod_{k\in\N} \frak m_k$$ since $\nhp = \m_{\a_{-1}}\oplus \nhimp$.

%% Sequential topology 
We note that $\widehat{\mathfrak{m}}$ has an induced topology as a subset of 
$\mathfrak{X}:=\prod_{k\in\Z} \frak m_k$ with the product topology.
The Lie product on $\mathfrak{m}$ extends to  $\widehat{\mathfrak{m}}$ but does not extend to $\mathfrak{X}$.
Since $\mathfrak{X}$ is a countable product of finite dimensional vector spaces, it is second countable, 
and hence sequential~\cite{Munkres}.   
Recall that a sequential topology is entirely captured by limits of sequences -- we will prove our topological results using limits  without further comment.

%We equip \(\widehat{\mathfrak m}\) with the topology induced from
%\[
%\mathfrak X:=\prod_{k\in\mathbb Z}\mathfrak m_k
%\]
%with its product topology. 
Let
\(\End_{\mathrm{cts}}(\widehat{\mathfrak m})\) denote the vector space of
continuous linear endomorphisms of \(\widehat{\mathfrak m}\). We give
\(\End_{\mathrm{cts}}(\widehat{\mathfrak m})\) the topology of pointwise
convergence: a sequence \(T_n\) converges to \(T\) if and only if, for every
\(x\in \widehat{\mathfrak m}\) and \(k\in\mathbb Z\),
\(
(T_nx)_k\longrightarrow (Tx)_k
\)
in the finite-dimensional space \(\mathfrak m_k\). We give
\(\Aut_{\mathrm{cts}}(\widehat{\mathfrak m})\) the induced topology, where
\(\Aut_{\mathrm{cts}}(\widehat{\mathfrak m})\) denotes the group of 
linear automorphisms of \(\widehat{\mathfrak m}\) that are also homeomorphisms. That is
\[
\Aut_{\mathrm{cts}}(\widehat{\mathfrak m})
=
\left\{
\phi:\widehat{\mathfrak m}\to \widehat{\mathfrak m}
\ \middle|\
\begin{array}{l}
\phi \text{ is a continuous linear bijection},\\[2pt]
\phi^{-1} \text{ is continuous},\\[2pt]
\phi([x,y])=[\phi(x),\phi(y)]
\text{ for all } x,y\in \widehat{\mathfrak m}
\end{array}
\right\}.
\]

Recall that a (formal) series  $\sum_i \vartheta_i$ of operators on a space $\mathfrak{Y}$ is called summable \cite{LL} if, for all $y\in\mathfrak{Y}$,  $\vartheta_i(y)$ is nonzero for only finitely many values of $i$.
We now generalize this concept.
\begin{defn}
Let $I$ and $K$ be countable index sets.
Let $\mathfrak{X}_k$ be a finite dimensional vector space for all $k\in K$, let $\mathfrak Y$ be a subspace of $\prod_k \mathfrak{X}_k$, and write $(y)_k$ for the projection of $y\in\mathfrak Y$ onto $\mathfrak{X}_k$, so that $y=\sum_{k\in K} (y)_k$.
Let $\vartheta_i$ be an operator on $\mathfrak Y$ for all $i\in I$.
We say that $\sum_i \vartheta_i$ is \emph{pro-summable} if $$I_k(y):= \{i\in I\mid \left(\vartheta_i(y) \right)_k\ne 0\}$$
is finite for all $k\in K$ and $y\in \mathfrak{Y}$.
If $\sum_i \vartheta_i$ is pro-summable, then it is a well defined operator on $\mathfrak{Y}$ with
$$ \sum_i \vartheta_i(y) := \sum_{k\in K} \sum_{i\in I_k(y)} \left(\vartheta_i (y) \right)_k.$$
\end{defn}

This definition depends on the choice of direct product decomposition $\prod_{k} \mathfrak{X}_k$. 
Every space we will consider is a subspace of $\prod_{k\in\Z} \mathfrak{m}_k$.
A pro-summable series is well defined without the use of limits, but note that not every well defined series is pro-summable (for example, $\exp\left(\ad\left(h_{1}\right) \right)$ considered in Section~\ref{S-gl2}).
The product of two pro-summable series is well defined, but may not be pro-summable.
Given pro-summable series $\varsigma:= \sum_{j\in J} \varsigma_j$ and $\vartheta:= \sum_{i\in I} \vartheta_i$, the product $\varsigma\vartheta$  is pro-summable if and only if
$$ \left\{(i,j) \in I\times J \mid (\varsigma_j\vartheta_i(y))_k\ne 0 \right\}$$
 is finite for all $k\in K$ and $y\in\mathfrak{Y}$.

Given $x\in\widehat{\frak m}$, there are only
finitely many $k<0$ with $(x)_k\ne 0$.
In other words, our choice of completion ensures that the formal sum  $x=\sum_{k\in\Z}(x)_k$ is finite in the negative direction but potentially infinite in the positive direction.
For $x\in\widehat{\frak m}-\{0\}$, the \emph{order} $N(x)$ is the smallest integer with $(x)_{N(x)}\ne0$, so that $x$  can be written uniquely as a  formal sum
$$x=\sum_{k=N(x)}^\infty (x)_k.$$

\begin{lemma}\label{L-expad}
If $u\in\C$, then
$\exp\left(u\ad\left(e_{-1} \right) \right)$ and $\exp\left(u\ad\left(f_{-1} \right) \right)$ are summable on $\mathfrak{m}$ and pro-summable on 
$\widehat{\mathfrak{m}}$.
If $x\in\widehat{\mathfrak{n}}^+$, then 
$\exp(\ad(x))$ is pro-summable on $\widehat{\mathfrak{m}}$.
\end{lemma}
\begin{proof}
The first statement follows from Proposition~\ref{P-GL-1action} (see Section~\ref{S-gl2}).
A summable series on $\mathfrak{m}$
clearly extends to a pro-summable series on $\widehat{\mathfrak{m}}$.
If $x\in\widehat{\mathfrak{n}}^+$, $y\in \mathfrak{m}_k$, $k\in\Z$, then $x=\sum_{n=0}^\infty (x)_n$ and
$$\left(\exp\left(\ad(x) \right)(y) \right)_k  = (y)_k+\sum_{n=0}^{k-N(y)} \sum_{\substack{m,n_1,\dots,n_m\in \N\text{ s.t.}\\n_1+\cdots+n_m=n}}
\frac{\ad\left(x_{n_i} \right)\cdots\ad\left(x_{n_m} \right)}{n_1!\cdots n_m!} 
 \left(y \right)_{k-n}$$
 is a finite sum. So $ \exp(\ad(x))$ is pro-summable. 
\end{proof}

%%%%%%%%%%%%%%%%%%%%%%%%%%%%%%%%%%%%%%%%%%%%%%%%%%%%%%
\subsection{The complete group $\Uhp_{\im}$ of smearing automorphisms}\label{SS-Uhat}
We now use the $\Z$-grading from the previous subsection to construct a complete  group $\Uhimp $ for the imaginary roots of $\m$, 
which is an analog of a complete unipotent group for a Kac-Moody algebra.
Note however that the choice of grading is mostly for convenience: 
similar methods could be used to construct groups with respect to the original $Q$-grading, or the $\Z\oplus\Z$-grading from Subsection~\ref{SS-mroots}.
We could also use alternative grading maps~$\lambda$, such as the 
$\Z$-grading on $\frak m$ given in \cite{BoInvent}.
These different choices of grading all give the same topological group up to isomorphism.

The completion $\widehat{\frak m}$ has two decompositions
$$\widehat{\frak m} = \nimm \oplus \frak{gl}_2({-1}) \oplus \nhimp 
  =  \frak n^- \oplus \frak h \oplus \widehat{\frak n}^+$$
related by 
$\widehat{\frak n}^+  = \C e_{-1}\oplus\widehat{\frak u}^+$ and
${\frak n}^- = \C f_{-1}\oplus {\frak u}^-$.
Define subspaces
$$ \nhp_k := \prod_{j\ge k} \mathfrak{m}_j, $$
for $k\in \Z$.
Now $\nhp_k$ is a $\widehat{\mathfrak{n}}^+$-module for all integers $k$;
$\nhp_k$ is a pro-Lie algebra for $k\ge0$; and
$\nhp_k$ is  pro-nilpotent Lie algebra for $k\ge0$ (Section~\ref{SS-pro}). 
%In particular $\widehat{\mathfrak{m}}_0=\mathfrak{gl}_2(-1) \oplus \nhimp$ and $\widehat{\mathfrak{m}}_1=\nhimp$.

\begin{defn}  For $n\in\N$, we define
$$
  \Uhp_n := \left\{ \varphi\in\Aut(\widehat{\frak m}) \mid 
  \varphi(y)\in y + \nhp_{k+n}\text{ whenever }y\in{\frak m}_k\;\text{for }\; k\in\mathbb{Z}\right\},
$$
and $\Uhimp:= \Uhp_1$.
\end{defn}
We refer to $\Uhimp$ as a group of \emph{smearing automorphisms} of $\nhimp$ since the action of $ \Uhimp$ on $\widehat{\frak m}$ smears $x_k\in{\frak m}_k$ across the infinite product  $\nhp_{k+1}=\prod_{j\ge k+1} \mathfrak{m}_{j}$. That is
$\varphi(x_k)\in x_k + \nhp_{k+1}$
for $x_k\in{\frak m}_k\;\text{and }\; k\in\mathbb{Z}.$

\begin{lemma}\label{L-Uhat}
For $n\in\N$, both $\Uhp_{\im}$ and $\Uhp_n$ are closed %pro-unipotent 
subgroups of $\Aut(\widehat{\frak m})$.
\end{lemma}

\begin{longver}
\begin{proof} 
Since $\prod_{j\ge k+n} {\mathfrak{m}}_{j}$ is closed in $\widehat{\mathfrak{m}}$, it is straightforward to show that if $\phi_i\to \phi$ and each 
$\phi_i\in\Uhp_n$ then $\phi\in\Uhp_n$.
%Then $\widehat{U}_1$ is a closed pro-unipotent subgroup of $\Aut(\widehat{\frak m})$ 
%with $\widehat{U}_1\ge\widehat{U}_2\ge\widehat{U}_3\ge\cdots$
%and each $\widehat{U}/\widehat{U}_i$ is a finite-dimensional unipotent algebraic group of automorphisms of 
%$\widehat{\mathfrak{n}}/\nhpi$.
\end{proof}
\end{longver}

We now describe the elements of $\Uhp_n$ as  infinite series of operators.  
For convenience, we read all products from the right to the left. For example,
\begin{align*}
\prod_{k=1}^\infty \exp\left(\ad\left(x_k \right) \right)
&:= \cdots  \exp\left(\ad\left(x_3 \right) \right)  \exp\left(\ad\left(x_2 \right) \right) \exp\left(\ad\left(x_1 \right) \right)\\
&= \lim_{n\to\infty}\,  \exp\left(\ad\left(x_n \right) \right) \cdots  \exp\left(\ad\left(x_2 \right) \right) \exp\left(\ad\left(x_1 \right) \right).
\end{align*}

\begin{shortver}
The following lemmas now follow by formal manipulation of the  series:
\end{shortver}

\begin{lemma}\label{L-welldef}
Let $x_k\in \frak m_k$  for all $k\geq1$. Then
$$\prod_{k=n}^\infty \exp\left(\ad\left(x_k \right) \right)$$
expands to a formal sum which is pro-summable. Hence this infinite product converges to
 a well-defined automorphism in $\Uhp_n$.
\end{lemma}

\begin{longver}
\begin{proof}
Let $x_\ell \in \frak m_\ell$ for some $\ell>0$. Now
 $$\exp\left(\ad\left(x_\ell \right) \right) =1+\ad\left(x_\ell \right)+\frac{1}{2!}\left(\ad\left(x_\ell \right) \right)^2+\cdots$$ 
is a well-defined  map $\widehat{\frak m}\to \widehat{\frak m}$ since
$\left(\exp\left(\ad\left(x_\ell \right) \right) \left(y \right) \right)_k=0$ is zero for $k<N(y)$, and for $k\ge N(y)$ it is the finite sum
$$ \left(\exp\left(\ad\left(x_\ell \right) \right) \left(y \right) \right)_k=\left(y_k +\ad\left(x_\ell \right)\left(\left(y \right)_{k-\ell} \right)\right)+\frac{1}{2!}\left(\ad\left(x_\ell \right) \right)^2\left(\left(y \right)_{k-2\ell} \right)+\cdots +\frac{1}{M!}\left(\ad\left(x_\ell \right) \right)^M\left(\left(y \right)_{k-M\ell} \right),$$
where  $M=\left\lfloor \frac{k-N(y)}{\ell} \right\rfloor$.
Furthermore, if $\ell> k-N(y)$ then  $\left(\exp\left(\ad\left(x_\ell \right) \right) \left(y \right)\right)_k=(y)_k$. So,
$$\varphi:=\prod_{\ell=1}^\infty \exp\left(\ad\left(x_\ell \right) \right)$$
is also well-defined, since 
$$\left(\phi \left(y \right) \right)_k = \left(\prod_{\ell=1}^{k-N(y)} \exp\left(\ad\left(x_\ell \right) \right)\left(y \right)\right)_k.$$
Now $\exp\left(\ad\left(x_\ell \right) \right)$ is an automorphism and it is easy to see that 
$\exp\left(\ad\left(x_\ell \right) \right)\in\Uhp_\ell$. 
Hence $\varphi$ is an automorphism and $\prod_{\ell=n}^\infty \exp\left(\ad\left(x_\ell \right) \right)\in \Uhp_n$.
The remaining results are similar.
\end{proof}
\end{longver}

\begin{lemma}\label{L-explim}
Let $x_{i}, x \in \nhimp$ for $i\in\N$. Then $\exp\left(\ad\left(x_i\right) \right)\to \exp(\ad(x))$ as $i\to\infty$ if and only if $x_i\to x$ as $i\to\infty$.
Let $x_{i,\ell}, x_\ell \in \frak m_\ell$  for $i,\ell\in\N$. Then
$$\prod_{\ell=1}^\infty \exp\left(\ad\left(x_{i,\ell} \right) \right)\to \prod_{\ell=1}^\infty \exp\left(\ad\left(x_\ell \right) \right) \qquad\text{as $i\to\infty$}$$
if and only if $x_{i,\ell}\to x_\ell$ as $i\to\infty$ for every $\ell$.
\end{lemma}

\begin{longver}
\begin{proof}
This follows immediately from the formulas for $\left(\exp\left(\ad\left(x_\ell\right) \right) \left(y \right) \right)_k$ in the proof of Lemma~\ref{L-welldef}.
\end{proof}
\end{longver}

The following proposition shows that the smearing automorphisms of $\Uhp_n$ are infinite series with leading term~$1$.
\begin{proposition}\label{P-Uprod}
$$\displaystyle \Uhp_n = \left\{\prod_{\ell=n}^\infty \exp\left(\ad\left(x_\ell \right) \right) \;\middle|\; x_\ell\in{\frak m}_\ell \,
\text{ for all $\ell\ge n$}\right\}.$$
\end{proposition}

\begin{longver}
\begin{proof}
Suppose $\varphi\in\Uhp_n$. We define linear maps $ \varphi_{k,i}:{\frak m}_{k}\to{\frak m}_{k+i}$ via the formula
$$\varphi\left(x_k \right) = x_k + \sum_{i=n}^\infty \varphi_{k,i}\left(x_k \right)$$
for $x_k\in\m_k$.

For $m=1,2$, we can write
$ \varphi_{0,n}\left(h_m \right)= \sum_{\a\in\Delta_n} p_{m,\a},$ with $p_{m,\a}\in{\frak m}_\a$.
Now
\begin{align*}
0&=\left(\varphi(0)\right)_n= \left(\varphi\left(\left[h_1,h_2 \right] \right)\right)_n 
  = \left(\left[\varphi\left(h_1 \right),\phi\left(h_2 \right)\right]\right)_n\\
&= \left(\left[ h_1 + \sum_{k=n}^\infty \varphi_{0,k}\left(h_1 \right),\; h_2+ \sum_{k=n}^\infty \varphi_{0,k}\left(h_2 \right)\right]\right)_n\\
&= \left[ h_1, \varphi_{0,n}\left(h_2 \right)\right] + \left[ \varphi_{0,n}\left(h_1 \right),h_2\right] \\
&=
\sum_{\a\in\Delta_n} \left[h_1,p_{2,\a} \right] -
\sum_{\a\in\Delta_n} \left[h_2,p_{1,\a} \right]\\
 &=
\sum_{\a\in\Delta_n}  \left(\a\left(h_1 \right)p_{2,\a}- \a\left(h_2 \right)p_{1,\a}\right).
\end{align*}
Hence
$\a(h_1) p_{2,\a}=\a(h_2) p_{1,\a},$
and so there are $p_{\a}\in\frak m_\a$ such that
$$p_{1,\a}=\a\left(h_1\right)p_{\a} \quad\text{and}\quad p_{2,\a}=\a\left(h_2\right)p_{\a}.$$
Now define 
$ x_n := -\sum_{\a\in\Delta_n} p_\a,$
so that 
$$ \varphi_{0,n}\left(h_m\right)=\sum_{\a\in\Delta_n} p_{m,\a} =\sum_{\a\in\Delta_n} \a\left(h_m\right)p_{\a}=\sum_{\a\in\Delta_n} \left[h_m,p_{\a}\right]
 =-\left[h_m,x_n\right]=\left[x_n,h_m\right].$$
Hence $ \varphi_{0,n}=\ad\left(x_n\right)|_{{\frak m}_0}$.

Fix  $\b\in\Delta_k$ and $y_{\b}\in\frak m_\b$.
We can write
$ \varphi_{k,n}\left(y_{\b}\right)= \sum_{\g\in\Delta_{n+k}} q_\g,$
with $q_\g\in\frak m_\g$.
Now, for $h\in\frak h$,
\begin{align*}
\b\left(h\right) \varphi_{k,n}\left(y_{\b}\right) &=  \varphi_{k,n}\left(\b(h)y_{\b}\right)=\left(\varphi\left(\left[h,y_{\b}\right]\right)\right)_{k+n} = \left(\left[\varphi\left(h\right),\varphi\left(y_{\b}\right)\right]\right)_{k+n} \\
 &=\left(\left[ h + \sum_{j=n}^\infty \varphi_{0,j}\left(h\right),\; y_{\b}+ \sum_{j=n}^\infty \varphi_{k,j}\left(y_{\b}\right)\right]\right)_{k+n}\\
 &= \left[h, \varphi_{k,n}\left(y_{\b}\right)\right] + \left[ \varphi_{0,k}\left(h\right), y_{\b}\right].
\end{align*}
Hence
\begin{align*}
 \b\left(h\right) \sum_{\g\in\Delta_{n+k}}   q_{\g}
 &= \sum_{\g\in\Delta_{n+k}} \g\left(h\right) q_{\g}  + 
 \sum_{\a\in\Delta_n}  \a\left(h\right) \left[p_{\a},y_{\b}\right],
\end{align*}
so
\begin{align*}
  \sum_{\g\in\Delta_{n+k}}  \left(\b(h)-\g(h)\right) q_{\g}&=  \sum_{\a\in\Delta_n}  \a(h) \left[p_{\a},y_{\b}\right]
  = \sum_{\g\in\Delta_{n+k}}  \left(\g(h)-\b(h)\right)  \left[p_{\g-\b},y_{\b}\right],
\end{align*}
where we take $p_{\g-\b}=0$ if $\g-\b$ is not a root.
Since this holds for arbitrary $h\in\frak h$,  we get $q_{\g} = -\left[p_{\g-\b},y_{\b}\right]$ and so
\begin{align*}
 \varphi_{k,n}\left(y_{\b}\right) &= \sum_{\g\in\Delta_{n+k}}q_\g 
=- \sum_{\g\in\Delta_{n+k}}  \left[p_{\g-\b},y_{\b}\right] 
= - \sum_{\a\in\Delta_{n}} \left[p_{\a}, y_{\b}\right]=\left[x_n,y_{\b}\right].
\end{align*}
That is, $ \varphi_{k,n}=\ad\left(x_n\right)|_{{\frak m}_k}$.

Hence $\varphi = \varphi^{(n+1)}\exp\left(\ad\left(x_n\right)\right)$ where  $\varphi^{(n+1)}\in\Uhp_{n+1}$. 
By induction, we can write
$$\varphi =\varphi^{(N+1)} \prod_{\ell=n}^N \exp\left(\ad\left(x_\ell\right)\right) $$ 
with $\varphi^{(N+1)}\in \Uhp_{N+1}$.
But $\varphi^{(N+1)}\to1$ as $N\to\infty$, and hence
$$\varphi=\lim_{N\to\infty}  \varphi^{(N+1)}\prod_{\ell=n}^N \exp\left(\ad\left(x_\ell\right)\right) =\lim_{N\to\infty} \prod_{\ell=n}^N \exp\left(\ad\left(x_\ell\right)\right)
 = \prod_{\ell=n}^\infty \exp\left(\ad\left(x_\ell\right)\right)$$
as required.
\end{proof}
\end{longver}

\begin{shortver}
\begin{proof}
Suppose $\varphi\in\Uhp_n$. For $x_k\in {\frak m}_k$, we write
$\varphi\left(x_k \right) = x_k + \sum_{i=n}^\infty \varphi_{k,i}\left(x_k \right),$
where $ \varphi_{k,i}$ is a linear map ${\frak m}_{k}\to{\frak m}_{k+i}$.

For $m=1,2$, we can write
$ \varphi_{0,n}\left(h_m \right)= \sum_{\a\in\Delta_n} p_{m,\a},$ with $p_{m,\a}\in{\frak m}_\a$.
Now
\begin{align*}
0&=\left(\varphi(0)\right)_n= \left(\varphi\left(\left[h_1,h_2 \right] \right)\right)_n 
  = \left(\left[\varphi\left(h_1 \right),\phi\left(h_2 \right)\right]\right)_n\\
&= \left(\left[ h_1 + \sum_{k=n}^\infty \varphi_{0,k}\left(h_1 \right),\; h_2+ \sum_{k=n}^\infty \varphi_{0,k}\left(h_2 \right)\right]\right)_n\\
%&= \left[ h_1, \varphi_{0,n}\left(h_2 \right)\right] + \left[ \varphi_{0,n}\left(h_1 \right),h_2\right] \\
&=
\sum_{\a\in\Delta_n} \left[h_1,p_{2,\a} \right] -
\sum_{\a\in\Delta_n} \left[h_2,p_{1,\a} \right]\\
 &=
\sum_{\a\in\Delta_n}  \left(\a\left(h_1 \right)p_{2,\a}- \a\left(h_2 \right)p_{1,\a}\right).
\end{align*}
Hence
$\a(h_1) p_{2,\a}=\a(h_2) p_{1,\a},$
and so we have $p_{\a}\in\frak m_\a$ such that
$p_{1,\a}=\a\left(h_1\right)p_{\a}$ and $ p_{2,\a}=\a\left(h_2\right)p_{\a}$.
Now define 
$ x_n := -\sum_{\a\in\Delta_n} p_\a,$
so that 
%$$ \varphi_{0,n}\left(h_m\right)=\sum_{\a\in\Delta_n} p_{m,\a} =\sum_{\a\in\Delta_n} \a\left(h_m\right)p_{\a}=\sum_{\a\in\Delta_n} \left[h_m,p_{\a}\right]
% =-\left[h_m,x_n\right]=\left[x_n,h_m\right].$$
%Hence 
$ \varphi_{0,n}=\ad\left(x_n\right)|_{{\frak m}_0}$.

Fix  $\b\in\Delta_k$ and $y_{\b}\in\frak m_\b$.
We can write
$ \varphi_{k,n}\left(y_{\b}\right)= \sum_{\g\in\Delta_{n+k}} q_\g,$
with $q_\g\in\frak m_\g$.
Now, for $h\in\frak h$,
\begin{align*}
\b\left(h\right) \varphi_{k,n}\left(y_{\b}\right) &= 
% \varphi_{k,n}\left(\b(h)y_{\b}\right)=
\left(\varphi\left(\left[h,y_{\b}\right]\right)\right)_{k+n} 
%= \left(\left[\varphi\left(h\right),\varphi\left(y_{\b}\right)\right]\right)_{k+n} %\\
% &=\left(\left[ h + \sum_{j=n}^\infty \varphi_{0,j}\left(h\right),\; y_{\b}+ \sum_{j=n}^\infty \varphi_{k,j}\left(y_{\b}\right)\right]\right)_{k+n}\\
% &
 = \left[h, \varphi_{k,n}\left(y_{\b}\right)\right] + \left[ \varphi_{0,k}\left(h\right), y_{\b}\right].
\end{align*}
Hence
\begin{align*}
 \b\left(h\right) \sum_{\g\in\Delta_{n+k}}   q_{\g}
 &= \sum_{\g\in\Delta_{n+k}} \g\left(h\right) q_{\g}  + 
 \sum_{\a\in\Delta_n}  \a\left(h\right) \left[p_{\a},y_{\b}\right],
\end{align*}
so
\begin{align*}
  \sum_{\g\in\Delta_{n+k}}  \left(\b(h)-\g(h)\right) q_{\g}&=  \sum_{\a\in\Delta_n}  \a(h) \left[p_{\a},y_{\b}\right]
  = \sum_{\g\in\Delta_{n+k}}  \left(\g(h)-\b(h)\right)  \left[p_{\g-\b},y_{\b}\right],
\end{align*}
where we take $p_{\g-\b}=0$ if $\g-\b$ is not a root.
Since this holds for arbitrary $h\in\frak h$,  we get $q_{\g} = -\left[p_{\g-\b},y_{\b}\right]$ and so
\begin{align*}
 \varphi_{k,n}\left(y_{\b}\right) &= \sum_{\g\in\Delta_{n+k}}q_\g 
=- \sum_{\g\in\Delta_{n+k}}  \left[p_{\g-\b},y_{\b}\right] 
= - \sum_{\a\in\Delta_{n}} \left[p_{\a}, y_{\b}\right]=\left[x_n,y_{\b}\right].
\end{align*}
That is, $ \varphi_{k,n}=\ad\left(x_n\right)|_{{\frak m}_k}$.

Hence $\varphi = \varphi^{(n+1)}\exp\left(\ad\left(x_n\right)\right)$ where  $\varphi^{(n+1)}\in\Uhp_{n+1}$. 
By induction, we can write
$$\varphi =\varphi^{(N+1)} \prod_{\ell=n}^N \exp\left(\ad\left(x_\ell\right)\right) $$
with $\varphi^{(N+1)}\in \Uhp_{N+1}$.
But $\varphi^{(N+1)}\to1$ as $N\to\infty$, and hence
$$\varphi=\lim_{N\to\infty}  \varphi^{(N+1)}\prod_{\ell=n}^N \exp\left(\ad\left(x_\ell\right)\right) %=\lim_{N\to\infty} \prod_{\ell=n}^N \exp\left(\ad\left(x_\ell\right)\right)
 = \prod_{\ell=n}^\infty \exp\left(\ad\left(x_\ell\right)\right)$$
as required.
\end{proof}
\end{shortver}

\begin{corollary} The group $\Uhimp$ consists of the Lie algebra automorphisms of $\mhat$ that can be written as infinite series with leading term~$1$, with respect to the $\Z$-grading of $\mhat$.
\end{corollary}

\begin{proposition} \label{C-closure}
$\Uhp_n$ is the closure of 
$\left\langle \exp(\ad(x)) \mid x\in{\frak m}_k,\; k\ge n \right\rangle.$
\end{proposition}
\begin{proof} Since
$\prod_{\ell=n}^\infty \exp(\ad(x_\ell)) =\lim_{N\to\infty} \prod_{\ell=n}^N \exp(\ad(x_\ell)).$
\end{proof}

%%%%%%%%%%%%%%%%%%%%%%%%%%%%%%%%%%%%%%%%%%%%%%%%%%%%%%
\subsection{The Baker--Campbell--Hausdorff formula and commutators in $\Uhimp$}
Note that  $(X,Y)$ denotes the group commutator $XYX^{-1}Y^{-1}$. 

Recall the Baker--Campbell--Hausdorff  formula in the completion of the free Lie algebra generated by $X$ and $Y$ (see, for example \cite[p.\ 32]{Stern}):
$$ \exp(X)\exp (Y) = \exp(Z)$$
where $Z$ is given by the formula 
$$Z= X+Y+\frac{1}{2}[X,Y]+\frac{1}{12}\left([X,[X,Y]]-[Y,[X,Y]]\right)+\cdots.$$

\begin{shortver}
This leads immediately to:
\end{shortver}

\begin{lemma}\label{BCH}
If $x,y\in \nhimp$, then
$$ \exp(\ad(x))\exp (\ad(y)) = \exp(\ad(z))$$
for some
$z \in \nhimp$.
\end{lemma}

\begin{longver}
\begin{proof}
The proof is immediate since we have 
\begin{align*}
z&:= x+y+\frac{1}{2}[x,y]+\frac{1}{12}\left([x,[x,y]]-[y,[x,y]]\right)+\cdots\in \nhimp.\qedhere
\end{align*}
\end{proof}
\end{longver}

\begin{lemma}\label{L-nilpexp}   For $n\in\N$, we have: 
\begin{enumerate}
\item\label{L-nilpexp-Uhat}  $\Uhp_n=\exp\left(\ad\left(\widehat{\frak n}^+_n\right)\right)$.
\item\label{L-nilpexp-coset}  $\exp\left(\ad(x)\right) \Uhp_{n+1} = \exp\left(\ad(x + \widehat{\frak n}^+_{n+1})\right)$ for $x\in\mathfrak{m}_n$.
\item\label{L-nilpexp-sum}  $\exp\left(\ad(x+y)\right) \Uhp_{n+1} = \exp\left(\ad(x))\exp(\ad(y)\right) \Uhp_{n+1}$ for $x,y\in\mathfrak{m}_n$.
\item\label{L-nilpexp-prod} 
  $\exp\left(\ad\left([x,y]\right)\right) \in \left( \exp\left(\ad(x)\right), \exp\left(\ad(y)\right) \right)\Uhp_{n+m+1}$ for $x\in\mathfrak{m}_n$ and~$y\in\mathfrak{m}_m$.
\end{enumerate}
\end{lemma}

\begin{longver}
\begin{proof}
Lemma~\ref{L-explim} shows that $\exp\left(\ad\left(\widehat{\frak n}^+_n\right)\right)$ is a closed set and Lemma~\ref{BCH} implies that~$\Uhp_n\subseteq\exp\left(\ad\left(\widehat{\frak n}^+_n\right)\right)$. 

Suppose $x\in\widehat{\frak n}^+_n$. Take  $x_n\in\frak{m}_n$ with $x-x_n\in\widehat{\frak n}^+_{n+1}$.
So Lemma~\ref{BCH} gives us  $$x^{(n+1)} = x-x_n+ \frac{1}{2}\left[x,-x_n\right]+\cdots\in\widehat{\frak n}^+_{n+1}$$ such that
$$\exp\left(\ad\left(x\right)\right) \exp\left(\ad\left(-x_n\right)\right)= \exp\left(\ad\left(x^{(n+1)}\right)\right)$$
and so
$$\exp\left(\ad\left(x\right)\right)= \exp\left(\ad\left(x^{(n+1)}\right)\right) \exp\left(\ad\left(x_n\right)\right).$$
By induction we get
$$\exp\left(\ad(x)\right)= \exp\left(\ad\left(x^{(N+1)}\right)\right) \prod_{\ell=n}^N\exp\left(\ad\left(x_l\right)\right) $$
with $x^{(N+1)}\in\widehat{\frak n}^+_{N+1}$.
Now
$$\exp\left(\ad\left({\widehat{\frak n}}^+_n\right)\right) \supseteq \exp\left(\ad\left({\widehat{\frak n}}^+_{n+1}\right)\right)\supseteq \dots$$
is a sequence of closed sets whose intersection is the trivial group.
Hence $\lim_{N\to\infty}\exp\left(\ad\left(x^{(N+1)}\right)\right)=1$, and so $\exp\left(\ad(x)\right)=\prod_{\ell=n}^\infty\exp\left(\ad\left(x_\ell\right)\right)\in \Uhp_n.$
Thus $\exp\left(\ad\left(\widehat{\frak n}^+_n\right)\right) \subseteq \Uhp_n$ and \eqref{L-nilpexp-Uhat} is proved. 
Parts~\eqref{L-nilpexp-coset} and \eqref{L-nilpexp-sum} are now straightforward, and (\ref{L-nilpexp-prod}) follows from Lemma~\ref{BCH}.
\end{proof}
\end{longver}

\begin{shortver}
\begin{proof}
Part~\eqref{L-nilpexp-Uhat} follows Lemma~\ref{BCH} and the fact that
 $\exp\left(\ad\left(\widehat{\frak n}^+_n\right)\right)$ is a closed set (by
 Lemma~\ref{L-explim}).
 Parts~\eqref{L-nilpexp-coset} and \eqref{L-nilpexp-sum} are now straightforward, and (\ref{L-nilpexp-prod}) follows from Lemma~\ref{BCH}.
\end{proof}
\end{shortver}

\begin{lemma}\label{BCHcomm} For $x,y\in\nhimp$ and $u,v\in\C$ we have
\begin{align*}
\left(\exp(ux),\exp(vy)\right)&=\exp\left(c_{1,1}uv[x,y]+c_{2,1}u^2v[x,[x,y]]+c_{1,2}uv^2[y,[x,y]]\right. \\
&\qquad\left. +c_{3,1}u^3v[x,[x,[x,y]]]+c_{2,2}u^2v^2[y,[x,[x,y]]]+\cdots \right)
\end{align*}
where $c_{i,j}\in\Q$ are constants independent of $x,y,u,v$.
\end{lemma}
\begin{proof} Apply Lemma~\ref{BCH} to $\exp(sx)\exp(ty)$ and $\exp(ty)\exp(sx)$.
\end{proof}

We can now give an infinite dimensional analog of the Chevalley group commutator formula \cite{St}.
Recall the definition of $\Sigma(\a,\b)$ from Subsection~\ref{SS-strings}. 
\begin{theorem}\label{T-commrel}
Suppose $\a,\b\in\Delta^+$, $x_{\a}\in{\frak m}_{\a}$, $y_{\b}\in{\frak m}_{\b}$. Then there are unique $z_{\g}\in{\frak m}_{\g}$ for each
$\g\in\Sigma(\a,\b)$ such that
$$\left(\exp\left(u\ad\left(x_{\a}\right)\right),\exp\left(v\ad\left(y_{\b}\right)\right)\right) =\prod_{\g=i\a+j\b\in\Sigma(\a,\b)} \exp\left(u^i v^j\ad\left(z_{\g}\right)\right).$$
In particular, $z_{\a+\b}=\left[x_{\a},y_{\b}\right]$.
\end{theorem}

\begin{longver}
\begin{proof} 
First write the equation from Lemma~\ref{BCHcomm}  as
\begin{align*}
 \left(\exp\left(u\ad\left(x_\a\right)\right),\exp\left(v\ad\left(y_\b\right)\right)\right) &= 
    \exp\left( \sum_{\g=i\a+j\b\in\Sigma(\a,\b)} u^i v^j\ad\left(x^{(1)}_{\g}\right) \right), 
\end{align*}
for $x^{(1)}_\g\in {\frak m}_{\g}$ with $x^{(1)}_{\a+\b}=\left[x_\a,y_\b\right]$.
Fix a linear ordering `$<$' on the set $\Delta^+$ with $\a<\b$ when $\lambda(a)<\lambda(\b)$ by taking an arbitrary linear ordering on each finite subset $\Delta_k$.
So this is more precisely written
\begin{align*}
\left(\exp\left(u\ad\left(x_\a\right)\right),\exp\left(v\ad\left(y_\b\right)\right)\right) &= 
    \exp\left( \sum_{n=1}^\infty u^{i_n} v^{j_n} \ad\left(x^{(1)}_{\g_n}\right) \right), 
\end{align*}
where $\g_n := i_n\a+j_n\b\in\Delta^+$, $\Sigma(\a,\b)=\left\{\g_n \mid n\in\N\right\}$ and $\g_1< \g_2<\cdots$.
Assume, for induction on $N$, that
$$\exp\left( \sum_{n=1}^\infty u^{i_n} v^{j_n} \ad\left(x^{(1)}_{\g_n}\right) \right) = 
  \prod_{n=1}^{N-1} \exp\left(u^{i_n} v^{j_n} \ad\left(x^{(n)}_{\g_n}\right)\right)
  \exp\left( \sum_{n=N}^\infty u^{i_n} v^{j_n} \ad\!\left(x^{(N)}_{\g_n}\right) \right)$$
for some $x^{(N)}_{\g_n} \in \mathfrak{m}_{\g_n}$ for $n\ge N$, recalling that products are read right-to-left.
We now apply Lemma~\ref{BCH} with
 $$x= - u^{i_N} v^{j_N} \ad\left(x^{(N)}_{\g_N}\right)\quad \text{and}\quad  
     y= \sum_{n=N}^\infty u^{i_n} v^{j_n} \ad\left(x^{(N)}_{\g_n}\right)$$ 
to get
$$\exp\left(- u^{i_N} v^{j_N} \ad\left(x^{(N)}_{\g_N}\right)\right) 
  \exp \left(\sum_{n=N}^\infty u^{i_n} v^{j_n} \ad\left(x^{(N)}_{\g_n}\right)\right) 
= \exp\left(\sum_{n=N+1}^\infty u^{i_n} v^{j_n} \ad\left(x^{(N+1)}_{\g_n}\right)\right), $$
for some $x^{(N+1)}_{\g_n} \in \mathfrak{m}_{\g_n}$,
and so
$$\exp\left(\sum_{n=N}^\infty u^{i_n} v^{j_n} \ad\left(x^{(N)}_{\g_n}\right)\right)  
  = \exp\left( u^{i_N} v^{j_N} \ad\left(x^{(N+1)}_{\g_N}\right)\right)
\exp\left(\sum_{n=N+1}^\infty u^{i_n} v^{j_n} \ad\left(x^{(N+1)}_{\g_n}\right)\right) $$

Now, setting $z_{\g_n}:= x^{(n)}_{\g_n}$, we have inductively proved that
\begin{align*}
  \exp\left( \sum_{n=1}^\infty u^{i_n} v^{j_n} \ad\left(x^{(1)}_{\g_n} \right) \right) 
    &= \prod_{n=1}^{N-1} \exp\left(u^{i_n} v^{j_n} \ad\left(z_{\g_n}\right)\right) \exp\left( \sum_{n=N}^\infty u^{i_n} v^{j_n} \ad\left(x^{(N)}_{\g_n}\right) \right)\\
    &\in  \prod_{n=1}^{N-1} \exp\left(u^{i_n} v^{j_n} \ad\left(z_{\g_n}\right)\right) \widehat{U}^+_{\lambda(\g_N)}
\end{align*}
The result now follows by taking $N\to\infty$, and note that $z_{\g_1}:= x^{(1)}_{\g_1}=\left[x_\a,y_b\right]$.
\end{proof}
\end{longver}

\begin{shortver}
\begin{proof} 
First write the equation from Lemma~\ref{BCHcomm}  as
\begin{align*}
 \left(\exp\left(u\ad\left(x_\a\right)\right),\exp\left(v\ad\left(y_\b\right)\right)\right) &= 
    \exp\left( \sum_{\g=i\a+j\b\in\Sigma(\a,\b)} u^i v^j\ad\left(x^{(1)}_{\g}\right) \right), 
\end{align*}
for $x^{(1)}_\g\in {\frak m}_{\g}$ with $x^{(1)}_{\a+\b}=\left[x_\a,y_\b\right]$.
Recall that the positive roots have a fixed linear ordering `$<$' that respects height, so this is more precisely written
\begin{align*}
\left(\exp\left(u\ad\left(x_\a\right)\right),\exp\left(v\ad\left(y_\b\right)\right)\right) &= 
    \exp\left( \sum_{n=1}^\infty u^{i_n} v^{j_n} \ad\left(x^{(1)}_{\g_n}\right) \right), 
\end{align*}
where $\g_n := i_n\a+j_n\b$, $\Sigma(\a,\b)=\left\{\g_n \mid n\in\N\right\}$ and $\g_1< \g_2<\cdots$.
We can now split the right hand side into a finite  product of exponentials by repeated use of Lemma~\ref{BCH}, and take limits to get the infinite product of exponentials.
\end{proof}
\end{shortver}

% $x\in\widehat{\frak n}^+$, then $\exp(\ad(x))$ is well-defined by an argument similar to Lemma~\ref{L-welldef}.
%\begin{lemma}
%Let $x_i,x\in \widehat{\frak n}^+$. Then
%$$\lim_{i\to\infty}\exp(\ad(x_i)) = \exp(\ad(x)) \qquad\iff\qquad \lim_{n\to\infty}x_i = x.$$
%\end{lemma}

%Recall that for $\varphi\in\Uhp$ and $y_k\in {\frak m}_k$ we have
%$$\varphi(y_k) = y_k + \sum_{i=1}^\infty \varphi_{k,i}(y_k)$$
%where for $x_i\in\frak m_i$, each $\varphi_{k,i}$ is the linear map ${\frak m}_{k}\to{\frak m}_{k+i}$
%$$\varphi_{k,i}(y_k)=\ad(x_i)|_{{\frak m}_k}(y_k).$$
%Thus
%$$\varphi(y_k) =\prod_{i=1}^\infty \exp(\ad(x_i))(y_k).$$
%
%This gives rise to  group composition in $\Uhp$:
%$$(\varphi'\cdot \varphi)(y_k) =\varphi'\cdot(\varphi(y_k))=\varphi'( y_k + \sum_{i=1}^\infty \varphi_{k,i}(y_k)
%)= y_k + \sum_{i=1}^\infty \varphi_{k,i}(y_k)+\sum_{j=1}^\infty \sum_{i=1}^\infty \varphi'_{k,j}(\varphi_{k,i}(y_k))$$
%for all $y_k\in{\frak m}_k$, where for $x_i\in\frak m_i$, each $\varphi_{k,i}$ is the linear map ${\frak m}_{k}\to{\frak m}_{k+i}$
%$$\varphi_{k,i}(y_k)=\ad(x_i)|_{{\frak m}_k}(y_k)$$
%and for $z_j\in\frak m_j$, $\varphi'_{k,j}$ is the linear map ${\frak m}_{k+i}\to{\frak m}_{k+i+j}$.
%$$\varphi'_{k,j}(\varphi_{k,i}(y_k))=\varphi'_{k,j}(\ad(x_i)(y_k))=\ad(z_j)|_{\frak m_{k+i}}((\ad(x_i)(y_k)),$$
%which gives
%\begin{align*}(\varphi'\cdot \varphi)(y_k)&= y_k + \sum_{i=1}^\infty \ad(x_i)(y_k)+\sum_{j=1}^\infty \sum_{i=1}^\infty \ad(z_j)(\ad(x_i)(y_k))\\
%&=\prod_{j=1}^\infty \exp(\ad(z_j))(\prod_{i=1}^\infty \exp(\ad(x_i))(y_k)).
%\end{align*}

%%%%%%%%%%%%%%%%%%%%%%%%%%%%%%%%%%%%%%%%%%%%%%%%%%%%%%
\subsection{$\nhimp$ as a pro-Lie algebra and $\widehat{U}_{\im}^+$ as a pro-Lie group}\label{SS-pro}
In the previous subsections, we have considered $\nhimp$ as a direct product, in order to keep our approach as straightforward as possible.
In this section we relate our approach to the concepts of pro-Lie algebras and pro-Lie groups (see for example \cite{kumar2012kac}).
%We then use this as a tool to find a generating set for $\widehat{U}_{\im}^+$.

Let $\Lambda$ be a directed index set (that is, a poset where, for all $i,j\in\Lambda$, there exists $k\in\Lambda$ with $i\le k$, $j\le k$).
Recall from \cite{kumar2012kac}, that the Lie algebra $\mathfrak{L}$ is a \emph{pro-Lie algebra} if there is a system $(\mathfrak{L}_i)_{i\in \Lambda}$ of ideals with $\mathfrak{L}_i\subseteq \mathfrak{L}_j$ whenever $i<j$ and each quotient $\mathfrak{L}/\mathfrak{L}_i$ is finite dimensional.
Further a pro-Lie algebra $\mathfrak{L}$ is \emph{pro-nilpotent} if each quotient $\mathfrak{L}/\mathfrak{L}_i$ is nilpotent.
Equivalently  $\mathfrak{L}$ is a pro-Lie algebra if and only if it is an inverse limit of finite dimensional Lie algebras; and is pro-nilpotent if and only if it is an inverse limit of finite dimensional nilpotent algebras.
We define pro-solvable Lie algebras similarly.

Each quotient $\mathfrak{L}/\mathfrak{L}_i$ has the topology of a finite dimensional $\C$-vector space, which induces a topology on 
$\mathfrak{L}$. 
This topology will serve a similar role as the pro-topology defined in \cite{kumar2012kac}, although Kumar starts with the discrete topology on vector spaces
in that book (see also
\cite{rousseau2012almost}).

We can now generalize the concept of summability \cite{LL} to this setting.
\begin{defn}
Let $\vartheta_k$ for $k\in\N\cup\{0\}$ be operators on a pro-Lie algebra 
$\mathfrak{L} := \lim_i \mathfrak{L/L}_i$. 
We say that $\vartheta:=\sum_{k=0}^\infty \vartheta_k$ is pro-summable if, for all $x\in \mathfrak{L}$ and $i\in\Lambda$,
$\vartheta_k(x)$ is in $\mathfrak{L}_i$ for all but finitely many values of $k$.
\end{defn}
It follows that $\sum_{k=0}^\infty\vartheta_k(x)+\mathfrak{L}_i$ is a well-defined finite sum in $\mathfrak{L}/\mathfrak{L}_i$, and so $\sum_{k=0}^\infty\vartheta_k(x)$ is well defined in $\mathfrak{L}$

Recall that
\begin{align*}
\nimp &= \bigoplus_{k=1}^\infty \mathfrak{m}_k, & \mathfrak{n}^+_n &= \bigoplus_{k\geq n}^\infty \mathfrak{m}_k,&
\nhimp &= \prod_{k=1}^\infty \mathfrak{m}_k, & \nhp_n &= \prod_{k\geq n}^\infty \mathfrak{m}_k.
\end{align*}
\begin{lemma}\label{nilpotent} 
For all $k\ge0$, $\nhimp/\nh_k= \nimp/{\mathfrak{n}}_k^+$
is nilpotent. Hence $$\nhimp=\varprojlim_k \nhimp/\nh_k\quad\text{ and }\quad \nimp=\varprojlim_k\nimp/{\mathfrak{n}}_k^+$$ are pro-nilpotent pro-Lie algebras, and $\nhimp$ is the completion of $\nimp$.
\end{lemma}
%\begin{proof}
%Clear for imaginary and the real part acts locally $\ad$-nilpotently. {\bf{\color{blue}[EXPAND]}}
%\end{proof}
%Lemma~\ref{nilpotent} makes $\mathfrak{n}^+/\mathfrak{n}_n$ a nilpotent Lie algebra for each $n$, and so the series
%$$\widehat{\mathfrak{n}}^+ = \nhp0 \ge \nhp1 \ge \cdots $$
%makes $\widehat{\mathfrak{n}}^+$ a pro-nilpotent pro-Lie algebra  \cite[check!]{kumar2012kac}.

Similarly a group $G$ is {\it pro-nilpotent}  if there is a system $({G}_i)_{i\in \Lambda}$ of  subgroups satisfying ${G}_i\unlhd{G}_j$ whenever $i<j$ and each quotient ${G}/{G}_i$ is nilpotent.
A group $G$ is {\it pro-unipotent} if there is a system $({G}_i)_{i\in \Lambda}$, with each $G_i$
acting unipotently on a 
finite-dimensional vector space $V_i,$ so that there is an inclusion $G_i\hookrightarrow G_j$ induced by an inclusion 
$V_i \hookrightarrow V_j$ for all $i<j$.
It is easily seen that a pro-unipotent group is pro-nilpotent.

\begin{lemma}\label{L-Uhat-pro-nilp} The group
 $
\Uhimp = \varprojlim_{\,i} \Uhimp/\Uhp_i\
$
is complete, pro-unipotent, and  hence pro-nilpotent. 
\end{lemma}

\begin{longver}
\begin{proof} 
Each group $\Uhimp/\Uhp_n$ is unipotent with central series
$$
 \Uhimp /\Uhp_n\ge\Uhp_2/\Uhp_n\ge\cdots \ge\Uhp_{n-1}/\Uhp_n\ge1,
$$
by Theorem 16.2.6 of \cite{Springer}. 
Hence $\Uhimp $ is complete, pro-unipotent and hence pro-nilpotent.
\end{proof}
\end{longver}

%%%%%%%%%%%%%%%%%%%%%%%%%%%%%%%%%%%%%%%%%%%%%%%%%%%%%%
\subsection{Complete unipotent groups}\label{SS-PositiveRootSys}
In this section, we construct a complete algebra for every positive  subsystem of $\Delta$, and the corresponding complete unipotent groups.
%In this section, we  show that  the remaining relations in the definition of $G(\mathfrak{m})$ hold in automorphism groups of  Lie algebras related to $\frak m$. These Lie algebras are obtained from completions of $\mathfrak{m}$ relative to positive root systems.
Let $\Pi$ be a positive  subsystem as in Subsection~\ref{SS-pos}.
Recall that $ \mathfrak{n}^{\Pi} := \bigoplus_{\a\in\Pi}\m_\a$, $\nh^{\,\Pi}:=\prod_{\a\in\Pi}\m_\a$, and 
$${\frak m}^{\,\Pi} :=    \mathfrak{n}^{-\Pi} \oplus \frak h \oplus \n^{\,\Pi}.$$
We define the completion of $\m^\Pi$ by
$$\widehat{\frak m}^{\,\Pi} :=    \mathfrak{n}^{-\Pi} \oplus \frak h \oplus \nh^{\,\Pi}.$$
In particular, write $\mathfrak{n}^{\pm}:= \mathfrak{n}^{\Delta_{\pm}}$, $\nh^{\pm}:= \nh^{\,\Delta_{\pm}}$, and 
$\mhat^{\pm}:= \mhat^{\,\Delta_{\pm}}$, and note that $\mhat^+$ is identical to $\mhat$.
Note that 
\begin{align*}
\np&=\C e_{-1}\oplus \nimp, &\nm&=\C f_{-1}\oplus \nimm,\\
\nhp&=\C e_{-1}\oplus \nhimp, &\nhm&=\C f_{-1}\oplus \nhimm.
\end{align*}

We can define a complete group $\widehat{U}^{\,\Pi}$ 
 by a slight modification of the methods in Section~\ref{SS-Uhat} choosing a suitable map $\lambda$ for each $\Pi$.
%We can also define  $\widehat{B}^{\,\Pi}$ as the semidirect product $H\ltimes \widehat{U}^{\,\Pi}$.
In particular, write $\widehat{U}^{\pm}:=\widehat{U}^{\Delta_\pm}$, and note that
$\Uhimp$ can be considered as a subgroup of $\widehat{U}^{+}$ in a natural way.

The following result now follows by the same method as Theorem~\ref{T-commrel}:
\begin{theorem}\label{T-commrel2}
Suppose $\a,\b\in\Delta$, $x_{\a}\in{\frak m}_{\a}$, $y_{\b}\in{\frak m}_{\b}$. Let $\Pi$ be a positive  subsystem containing $\a$ and $\b$. Then there are unique $z_{\g}\in{\frak m}_{\g}$ for each
$\g\in\Sigma(\a,\b)$ such that
$$\left(\exp\left(u\ad\left(x_{\a}\right)\right),\exp\left(v\ad\left(y_{\b}\right)\right)\right) =\prod_{\g=i\a+j\b\in\Sigma(\a,\b)} \exp\left(u^i v^j\ad\left(z_{\g}\right)\right)$$
holds in $\UhPi$.
In particular, $z_{\a+\b}=\left[x_{\a},y_{\b}\right]$, and each $z_\g$ can be written as an expression in the free Lie algebra generated by $\left\{x_{\a},y_{\b}\right\}$, independent of the choice of $\,\Pi$.
\end{theorem}

As in Subsection~\ref{SS-pro}, we have:
\begin{lemma}
For every positive subsystem $\Pi$,
$\mathfrak{n}^{\Pi}$ is pro-nilpotent with completion $\nh^{\Pi}$ and
$\UhPi$ is complete pro-unipotent.
\end{lemma}

From the cases in Proposition~\ref{P-Sigma} where $\Sigma(\a,\b)$ is finite, we get:
\begin{corollary}
If $x_{-1}\in \m_{\a_{-1}}, y_{-1}\in \m_{-\a_{-1}}, x_{\ell,j,k}\in\m_{\a_{\ell,jk}}, y_{\ell,j,k}\in\m_{-\a_{\ell,jk}}$ 
and $u,v\in\C$, then
\begin{enumerate}
\item $\left(\exp\left(u\ad\left(x_{-1}\right)\right),\exp\left(v\ad\left(x_{j-1,jk}\right)\right)\right) 
          =1$,\\ $\left(\exp\left(u\ad\left(y_{-1}\right)\right),\exp\left(v\ad\left(y_{j-1,jk}\right)\right)\right) = 1$,
\item $\left(\exp\left(u\ad\left(x_{-1}\right)\right),\exp\left(v\ad\left(y_{0,jk}\right)\right)\right)=1$,\\ 
           $\left(\exp\left(u\ad\left(y_{-1}\right)\right),\exp\left(v\ad\left(x_{0,jk}\right)\right)\right) =1$,
\item $\left(\exp\left(u\ad\left(x_{\ell,jk}\right)\right),\exp\left(v\ad\left(y_{m,pq}\right)\right)\right) =1$
         for $j\ne p$, $k\ne q$, or $|\ell-m|>1$.
\item $\left(\exp\left(u\ad\left(x_{1,2k}\right)\right),\exp\left(v\ad\left(y_{0,2k}\right)\right)\right) =\exp\left(uv\ad\left([x_{1,2k},y_{0,2k}]\right)\right)$,\\
         $\left(\exp\left(u\ad\left(x_{0,2k}\right)\right),\exp\left(v\ad\left(y_{1,2k}\right)\right)\right) =\exp\left(uv\ad\left([x_{0,2k},y_{1,2k}]\right)\right)$.
\end{enumerate}
These relations hold in every $\UhPi$ in which the exponentials are well defined.
\end{corollary}

For $u\in\C$, define
\begin{itemize}
\item $X_{-1}(u)=\exp\left(u\ad\left(e_{-1}\right)\right)$, which can be considered an element of  $\widehat{U}^{\,\Pi}$ 
whenever $\a_{-1}\in\Pi$;
\item $Y_{-1}(u)=\exp\left(u\ad\left(f_{-1}\right)\right)$, which can be considered an element of  $\widehat{U}^{\,\Pi}$ 
whenever $-\a_{-1}\in\Pi$;
\item $X_{\ell,jk}(u)=\exp\left(u\ad\left(e_{\ell,jk}\right)\right)$, which can be considered an element of  
$\widehat{U}^{\,\Pi}$ whenever $\a_{\ell,jk}\in\Pi$; and
\item $Y_{\ell,jk}(u)=\exp\left(u\ad\left(f_{\ell,jk}\right)\right)$, which can be considered an element of  $\widehat{U}^{\,\Pi}$ 
whenever $-\a_{\ell,jk}\in\Pi$.
\end{itemize}
%The particular choice of $\Pi$ will be obvious from context.
\begin{corollary} $\;$\label{C-comm}
The following relations hold in every $\widehat{U}^{\,\Pi}$ in which the elements involved are defined:
\begin{enumerate}
%\item \eqref{GL2-G-XX},  \eqref{GL2-G-YY}, \eqref{G-XX}, \eqref{G-YY};
\item\label{C-comm-XX-YY} $\left(X_{-1}(u),X_{j-1,jk}(v)\right) = \left(Y_{-1}(u),Y_{j-1,jk}(v)\right) = 1$, %ie, \eqref{G-X-1X} and \eqref{G-Y-1Y};
\item $\left(X_{-1}(u),Y_{0,jk}(v)\right) = \left(Y_{-1}(u),X_{0,jk}(v)\right) = 1$, %ie, \eqref{G-X-1Y} and \eqref{G-Y-1X};
\item $\left(X_{\ell,jk}(u),Y_{m,pq}(v)\right) =1$ for $j\ne p$, $k\ne q$, or $\left|\ell-m\right|>1$. %ie, \eqref{G-XYcomm},
\item $\left(X_{1,2k}(u),Y_{0,2k}(v)\right)=X_{-1}(uv)$ and $\left(X_{0,2k}(u),Y_{1,2k}(v)\right)=Y_{-1}(-uv)$.
\end{enumerate}
\end{corollary}
In particular, \eqref{C-comm-XX-YY} holds in $\Uhpm$.

%%%%%%%%%%%%%%%%%%%%%%%%%%%%%%%%%%%%%%%%%%%%%%%%%%%%%%
%%%%%%%%%%%%%%%%%%%%%%%%%%%%%%%%%%%%%%%%%%%%%%%%%%%%%%
\section{Rank 2 linear subgroups}\label{S-gl2}
We have already noted that $\left\{e_{-1},f_{-1},h_1,h_2\right\}$ is a basis for a subalgebra $\mathfrak{gl}_2(-1)\cong \mathfrak{gl}_2$ of $\mathfrak m$ and, for each $(\ell,j,k)\in E$,  $\left\{e_{\ell,jk},f_{\ell,jk},h_1,h_2 \right\}$ is a basis for a subalgebra 
$\mathfrak{gl}_2(\ell,j,k)\cong \mathfrak{gl}_2$ of $\mathfrak m$. 
In this section, we construct groups $\GL_2(-1)$ and $\GL_2(\ell,j,k)$ associated to each of these subalgebras by exponentiation in the adjoint representation. %We show that each of these groups satisfy the corresponding relations listed in Definition~\ref{Defn}. 
Our motivation is to show  that these groups can be realized as groups of automorphisms of $\frak m$ or of a $\frak{gl}_2$ subalgebra of $\frak m$.

Note that we use a fixed branch of the complex logarithm in this section. 
We also define fractional powers in terms of this logarithm ($b^{1/n} := e^{\log(b)/n}$). 
The particular choice is not critical, since we eventually show that all our actions and presentations are independent of the branch chosen.

%%%%%%%%%%%%%%%%%%%%%%%%%%%%%%%%%%%%%%%%%%%%%%%%%%%%%%
\subsection{Standard rank 2 linear groups}\label{SS-stdgl2}
Here we review the presentations of $\gl_2$ and $\GL_2(\C)$.

We define a \emph{standard basis} for a $\gl_2$ Lie algebra to be elements
$\left\{\overline{e},\overline{f},\overline{h_1},\overline{h_2} \right\}$ satisfying the %given by
defining relations 
\begin{align*}
\tag{gl:1e}\label{gl2-h1h2}
  \left[\overline{h}_1,\overline{h}_2 \right]&=0,\\
\tag{gl:4a}\label{gl2-h1e}
  \left[\overline{h}_1,\e \right] &=  \e,\\
\tag{gl:4b}\label{gl2-h2e}
  \left[\overline{h}_2, \e \right] &= -\e,\\
\tag{gl:4c}\label{gl2-h1f}
  \left[\overline{h}_1,\f \right] &= - \f,\\
\tag{gl:4d}\label{gl2-h2f}
  \left[\overline{h}_2,\f \right] &= \f,\\
\tag{gl:5}\label{gl2-ef}
  \left[\e,\f \right]&=\overline{h}_1-\overline{h}_2.
\end{align*}
In the usual matrix representation of $\gl_2$, we have a standard basis
$\overline{e}=\left(\begin{smallmatrix} 0&1\\0&0\end{smallmatrix}\right)$,
$\overline{f}=\left(\begin{smallmatrix} 0&0\\1&0\end{smallmatrix}\right)$,
$\overline{h}_1=\left(\begin{smallmatrix} 1&0\\0&0\end{smallmatrix}\right)$,
$\overline{h}_2=\left(\begin{smallmatrix} 0&0\\0&1\end{smallmatrix}\right)$.

%\begin{comment} Magma code check:
%> M := MatrixAlgebra(Rationals(),2);
%> e := M![0,1,0,0]; f := M![0,0,1,0];    
%> h1 := M![1,0,0,0]; h2 := M![0,0,0,1];
%> e*f-f*e eq h1-h2;  h1*h2-h2*h1 eq 0;
%true
%true
%> h1*e-e*h1 eq e;  h2*e-e*h2 eq -e;     
%true
%true
%> h1*f-f*h1 eq -f;  h2*f-f*h2 eq f;
%true
%true
%
%F<s>:= RationalFunctionField(Rationals(),1);
%X := Matrix(4,4,[F|1,0,0,0, -s^2,1,s,-s, -s,0,1,0, s,0,0,1]);
%Y := Matrix(4,4,[F|1,-s^-2,s^-1,-s^-1, 0,1,0,0, 0,-s^-1,1,0, 0,s^-1,0,1]);
%X;Y;X*Y*X;
%\end{comment}

Similarly, 
for a group isomorphic to $\GL_2(\C)$,
we define \emph{standard generators} to be elements
$\overline{X}(u),$ $
\overline{Y}(u),$ $
\overline{H}_1(s),$ $
\overline{H}_2(s),$ $ \overline{w}(s),$ $\overline{w}$ for $u\in\C$ and $s\in\C^\times$
satisfying the defining relations 
\begingroup\allowdisplaybreaks
\begin{align*}
\tag{GLH:0}\label{GL2-wdefn} 
  \overline{w}(s) &:= \overline{X}(s) \overline{Y}\left(-s^{-1} \right) \overline{X}(s),\quad \overline{w} := \overline{w}(1),\\
\tag{GLH:1a}\label{GL2-H1H1}
   \overline{H}_1(s)\overline{H}_1(t)&=\overline{H}_1(st),\\
\tag{GLH:1b}\label{GL2-H2H2}
   \overline{H}_2(s)\overline{H}_2(t)&=\overline{H}_2(st),\\
\tag{GLH:2}\label{GL2-H1H2}
    \overline{H}_1(s)\overline{H}_2(t)&=\overline{H}_2(t)\overline{H}_1(s),\\
\tag{GL:1a}\label{GL2-XX}
   \overline{X}(u)\overline{X}(v)&=\overline{X}(u+v),\\
\tag{GL:1b}\label{GL2-YY}
   \overline{Y}(u)\overline{Y}(v)&=\overline{Y}(u+v),\\
\tag{GL:2}\label{GL2-XY}
  \overline{Y}(-t)\overline{X}(s)\overline{Y}(t)&= \overline{X}\left(-t^{-1} \right) \overline{Y}\left(-t^{2}s \right) \overline{X}\left(t^{-1} \right),\\
\tag{GL:3}\label{GL2-wwH}
  \overline{w}(s)\overline{w}&=\overline{H}_1\left(-s \right)\overline{H}_2\left(-s^{-1} \right). \\
%\tag{$\GL_2$-G2b}  X_{-1}(-s)Y_{-1}(t)X_{-1}(s) &=Y_{-1}(t), \text{\textcolor{red}{REMOVE}}\\
%\tag{$\GL_2$-G2c}  Y_{-1}(-s)X_{-1}(t)Y_{-1}(s) &=X_{-1}(t), \text{\textcolor{red}{REMOVE}}\\
\intertext{Some useful additional relations are}
\tag{GL:4a}\label{GL2-wX}
  \overline{w}\overline{X}(u)\overline{w}^{-1}&=\overline{Y}(-u),\\
\tag{GL:4b}\label{GL2-wY}
  \overline{w}\overline{Y}(u)\overline{w}^{-1}&=\overline{X}(-u),\\
\tag{GL:5a}\label{GL2-wH1}
  \overline{w} \overline{H}_1(s)\overline{w}^{-1} &= \overline{H}_2(s),\\
\tag{GL:5b}\label{GL2-wH2}
  \overline{w} \overline{H}_2(s)\overline{w}^{-1}  &= \overline{H}_1(s),\\
\tag{GL:6a}\label{GL2-H1X}
  \overline{H}_1(s)\overline{X}(u)\overline{H}_1(s)^{-1}&=\overline{X}(su),\\
\tag{GL:6b}\label{GL2-H2X}
 \overline{H}_2(s)\overline{X}(u)\overline{H}_2(s)^{-1}&=\overline{X}\left(s^{-1}u\right),\\
\tag{GL:6c}\label{GL2-H1Y}
  \overline{H}_1(s)\overline{Y}(u)\overline{H}_1(s)^{-1}&=\overline{Y}\left(s^{-1}u \right),\\
\tag{GL:6d}\label{GL2-H2Y}
  \overline{H}_2(s)\overline{Y}(u)\overline{H}_2(s)^{-1}&=\overline{Y}(su),
\end{align*}
\begin{align*}
\tag{GL:7}\label{GL2-Yelim}
  \overline{Y}(s)& =\overline{X}\left(s^{-1} \right)\overline{H}_1\left(-s^{-1} \right)\overline{H}_2\left(-s \right)\overline{w}\overline{X}\left(s^{-1} \right).
\end{align*}

\endgroup
In the usual matrix representation of $\GL_2(\C)$, we have standard generators
$\overline{X}(u)=\left(\begin{smallmatrix} 1&u\\0&1\end{smallmatrix}\right)$,
$\overline{Y}(u)=\left(\begin{smallmatrix} 1&0\\u&1\end{smallmatrix}\right)$,
$\overline{H}_1(s)=\left(\begin{smallmatrix} s&0\\0&1\end{smallmatrix}\right)$,
$\overline{H}_2(s)=\left(\begin{smallmatrix} 1&0\\0&s\end{smallmatrix}\right)$,
$\overline{w}(s)=\left(\begin{smallmatrix} 0&s\\-s^{-1}&0\end{smallmatrix}\right)$,
$\overline{w}=\left(\begin{smallmatrix} 0&1\\-1&0\end{smallmatrix}\right)$.

\begin{lemma}
Let $\mathfrak{L}$ be a Lie algebra and suppose $\mathfrak{g}\subseteq\mathfrak{L}$
is a $\gl_2$ subalgebra with standard basis $\left\{\overline{e},\overline{f},\overline{h_1},\overline{h_2} \right\}$.
Let $\ad$ denote the adjoint action on $\mathfrak{L}$.
If the series
\begin{align*}
\overline{X}(u) &:= \exp\left({u} \ad\left(\overline{e} \right) \right),&
\overline{Y}(u) &:= \exp\left({u} \ad\left(\overline{f} \right) \right),\\
\overline{H}_1(s) &:= \exp\left(\log(s)\ad\left(\overline{h}_1) \right) \right),&
\overline{H}_2(s) &:= \exp\left(\log(s)\ad\left(\overline{h}_2 \right) \right),
\end{align*}
converge, then they give a standard generating set for a subgroup of $\Aut(\mathfrak{L})$
isomorphic to $\GL_2(\C)$.
\end{lemma}
\begin{proof}
Extend the generating set by taking $\overline{w}(s) := \overline{X}(s) \overline{Y}\left(-s^{-1} \right) \overline{X}(s)$, $\overline{w} := \overline{w}(1)$. The result now follows by the usual properties of exponentials.
\end{proof}

\begin{comment} Magma code check:
> F<t,s> := RationalFunctionField(Rationals(),2);
> M := MatrixAlgebra(F,2);
> X := func< a | M![1,a,0,1] >;
> Y := func< a | M![1,0,a,1] >;
> H1 := func< a | M![a,0,0,1] >;
> H2 := func< a | M![1,0,0,a] >;
> w := func< a | X(a)*Y(-a^-1)*X(a) >;
> //
> // XX, YY
> X(s)*X(t) eq X(s+t);
true
> Y(s)*Y(t) eq Y(s+t);
true
> //
> // HH
> H1(s)*H1(t) eq H1(s*t);
true
> H2(s)*H2(t) eq H2(s*t);
true
> H1(s)*H2(t) eq H2(t)*H1(s);
true
> //
> // wX, wY, XY, wwH
> w(1)*X(s)*w(1)^-1 eq Y(-s);
true
> w(1)*Y(s)*w(1)^-1 eq X(-s);
true
> Y(-t)*X(s)*Y(t) eq X(-t^-1)*Y(-t^2*s)*X(t^-1);
true
> w(s)*w(1) eq H1(-s)*H2(-s^-1);
true
> //
> // wH
> w(1)^-1*H1(s)*w(1) eq H2(s);
true
> w(1)^-1*H2(s)*w(1) eq H1(s);
true
> //
> // HX, HY
> H1(t^-1)*X(s)*H1(t) eq X(t^-1*s);
true
> H2(t^-1)*X(s)*H2(t) eq X(t*s);
true
> H1(t^-1)*Y(s)*H1(t) eq Y(t*s);
true
> H2(t^-1)*Y(s)*H2(t) eq Y(t^-1*s);
true
> Y(t) eq X(t^-1)*H1(-t^-1)*H2(-t)*w(1)*X(t^-1);
true
\end{comment}

%%%%%%%%%%%%%%%%%%%%%%%%%%%%%%%%%%%%%%%%%%%%%%%%%%%%%%
\subsection{The group $\GL_2(-1)$ corresponding to the real roots}\label{SS-gl2re}

\begin{lemma} The group $\GL_2(-1) \cong \GL_2(\C)$ has standard generators
\begin{align*}
X_{-1}(u) &:= \exp\left({u} \ad\left(e_{-1} \right) \right),&
Y_{-1}(u) &:= \exp\left({u} \ad\left(f_{-1} \right) \right),\\
{H}_1(s) &:= \exp\left(\log(s)\ad\left({h}_1) \right) \right),&
{H}_2(s) &:= \exp\left(\log(s)\ad\left({h}_2 \right) \right),\\
\widetilde{w}_{-1}(s) &:=X_{-1}(s)Y_{-1}(-s^{-1})X_{-1}(s),& \widetilde{w}_{-1} &:= \widetilde{w}_{-1}(1).
\end{align*}

\end{lemma}
\begin{proof}
The set $\left\{e_{-1},f_{-1},h_1,h_2\right\}$ is a
standard basis for $\mathfrak{gl}_2(-1)$ by \eqref{Mhh}, \eqref{Mhe-}, \eqref{Mhf-}, and \eqref{Me-f-}.
%Here $\exp(X) = \sum_{n=0}^\infty \frac{X^n}{n!}$ is the usual series expansion. 
It is easily checked that these particular series converge, since $\ad(e_{-1})$ and $\ad(f_{-1})$ are locally nilpotent on $\m$, while $\ad(h_{1})$ and $\ad(h_{2})$ act  diagonally with respect to the root space decomposition of $\m$.
(See Subsection~\ref{SS-graded} on the topology used for this convergence). 
 \end{proof}

The rest of this section is devoted to determining the action of $\GL_2(-1)$ on $\mathfrak{m}$,
which is needed in Section~\ref{S-related}.
%Note that $X_{-1}(u)$ and $Y_{-1}(u)$ are well defined elements of  $\Aut(\mathfrak{m})$ since $e_{-1}$ and $f_{-1}$ are locally ad-nilpotent on~$\mathfrak{m}$.
The following result gives explicit formulas for the action of $\GL_2(-1)$ on $e_{\ell,jk},f_{\ell,jk}\in\mathfrak{m}$.

\begingroup\allowdisplaybreaks
\begin{proposition}\label{P-GL-1action}
 For $u\in\C$, $s\in\C^\times$, and $(\ell,j,k)\in E$, we have:
\begin{align*}
&X_{-1}(u):& 
 e_{\ell,jk} \mapsto \sum_{m=\ell}^{j-1} \binom{m}{\ell} u^{m-\ell}e_{m,jk},\; f_{\ell,jk}   \mapsto \sum_{m=0}^\ell   \binom{j-1-m}{j-1-\ell} u^{\ell-m}f_{m,jk},\\
&Y_{-1}(u):&  
  e_{\ell,jk}   \mapsto \sum_{m=0}^\ell   \binom{j-1-m}{j-1-\ell} u^{\ell-m}e_{m,jk}, \; f_{\ell,jk} \mapsto \sum_{m=\ell}^{j-1} \binom{m}{\ell} u^{m-\ell}f_{m,jk},
\end{align*}
\begin{align*}
&H_1(s):&   e_{\ell,jk}&\mapsto s^{(\ell+1)} e_{\ell,jk}, &  f_{\ell,jk}& \mapsto s^{-(\ell+1)} f_{\ell,jk},\\
&H_2(s):&     e_{\ell,jk}& \mapsto s^{(j-\ell)} e_{\ell,jk},  &  f_{\ell,jk}& \mapsto s^{-(j-\ell)} f_{\ell,jk},\\
&\widetilde{w}_{-1}:&   e_{\ell,jk} & \mapsto  (-1)^{j-1-\ell} e_{j-1-\ell,jk}, &  f_{\ell,jk} & \mapsto  (-1)^{j-1-\ell} f_{j-1-\ell,jk}.
\end{align*}
\end{proposition}
\begin{proof} 
By Theorem~\ref{T-identities-v2}, the action of $\mathfrak{gl}_2(-1)$ on $\bigoplus_{(j,k)\in I^{\im}} V^+_{jk}$ is equivalent to the action of $\mathfrak{gl}_2$ on $\left(\C[X,Y] \right)^*$ with
$e_{\ell,jk}\mapsto \left(X^\ell Y^{j-1-\ell} \right)^*$. Here $\left(X^a Y^b \right)^*$ denotes the map $X^cY^d\mapsto \delta_{ac}\delta_{bd}X^a Y^b$ in $\left(\C[X,Y] \right)^*$. Similarly the action on $\bigoplus_{(j,k)\in I^{\im}} V^-_{jk}$ is equivalent to the action on $\left(\C[X,Y] \right)^*$ with
$f_{\ell,jk}\mapsto \left(X^{j-1-\ell}Y^\ell \right)^*$.
The formulas given are now the standard action of $\GL_2(\C)$ on $\left(\C[X,Y] \right)^*$.
\begin{longver}
Explicitly
\begin{align*}
\ad(e_{-1})^i e_{\ell,jk} &= \begin{cases} \frac{(\ell+i)!}{\ell!} % \prod_{q=1}^p(\ell+q) 
  e_{\ell+i,jk}&\text{if $0\le i< j-\ell$,}\\ 0&\text{if $i\ge j-\ell$,}\end{cases} \\
\ad(f_{-1})^i e_{\ell,jk} &= \begin{cases}  \frac{(j-\ell-1+i)!}{(j-\ell-1)!} e_{\ell-i,jk}&\text{if $0\le i\le \ell$,}\\ 0&\text{if $i> \ell$,}\end{cases} \\
\ad(h_1)^i e_{\ell,jk}&= (\ell+1)^i e_{\ell,jk},\\
\ad(h_2)^i e_{\ell,jk}&= (j-\ell)^i e_{\ell,jk},\\
\ad(e_{-1})^i f_{\ell,jk} &= \begin{cases} \frac{(j-\ell-1+i)!}{(j-\ell-1)!} % \prod_{q=1}^p(\ell+q) 
  f_{\ell-i,jk}&\text{if $0\le i\le \ell$,}\\ 0&\text{if $i >\ell$,}\end{cases} \\
\ad(f_{-1})^i f_{\ell,jk} &= \begin{cases}  \frac{(\ell+i)!}{\ell!} e_{\ell+i,jk}&\text{if $0\le i< j-\ell$,}\\ 0&\text{if $i\ge j-\ell$,}\end{cases} \\
\ad(h_1)^i f_{\ell,jk}&= (-1)^i(\ell+1)^i f_{\ell,jk},\\
\ad(h_2)^i f_{\ell,jk}&= (-1)^i(j-\ell)^i f_{\ell,jk}.
\end{align*}
Hence
\begin{align*}
\exp(a\ad(e_{-1})) e_{\ell,jk} &=\sum_{i=0}^{j-\ell-1} \binom{\ell+i}{i} a^{i}e_{\ell+i,jk}
  =\sum_{m=\ell}^{j-1} \binom{m}{\ell}   a^{m-\ell}e_{m,jk}, \\
\exp(a\ad(f_{-1})) e_{\ell,jk} &=\sum_{i=0}^{\ell} \binom{j-\ell-1+i}{i} a^{i}e_{\ell-i,jk}
  = \sum_{m=0}^\ell   \binom{j-1-m}{j-1-\ell} a^{\ell-m}e_{m,jk} ,\\
\exp(a\ad(h_1)) e_{\ell,jk}&= e^{(\ell+1)a} e_{\ell,jk},\\
\exp(a\ad(h_2)) e_{\ell,jk}&= e^{(j-\ell)a} e_{\ell,jk},\\
\exp(a\ad(e_{-1})) f_{\ell,jk} &=\sum_{i=0}^{\ell} \binom{j-\ell-1+i}{i} a^{i}f_{\ell-i,jk}
  =\sum_{m=\ell}^{j-1} \binom{m}{\ell}   a^{m-\ell}f_{m,jk}, \\
\exp(a\ad(f_{-1})) f_{\ell,jk} &=\sum_{i=0}^{j-\ell-1} \binom{\ell+i}{i} a^{i}f_{\ell+i,jk}
  = \sum_{m=0}^\ell   \binom{j-1-m}{j-1-\ell} a^{\ell-m}f_{m,jk}, \\
\exp(a\ad(h_1)) f_{\ell,jk}&= f^{-(\ell+1)a} e_{\ell,jk},\\
\exp(a\ad(h_2)) f_{\ell,jk}&= f^{-(j-\ell)a} e_{\ell,jk}.
\end{align*}

Now 
\begin{align*}
\widetilde{w}_{-1} e_{\ell,jk} &= \sum_{m=\ell}^{j-1} \binom{m}{\ell}   \sum_{p=0}^m   \binom{j-1-p}{j-1-m} (-1)^{m-p} 
  \sum_{q=p}^{j-1}\binom{q}{p}  e_{q,jk}\\
&= \sum_{m=\ell}^{j-1}  \sum_{p=0}^m \sum_{q=p}^{j-1}  \binom{m}{\ell} \binom{j-1-p}{m-p} \binom{q}{p} (-1)^{m-p} e_{q,jk}\\
&=  \sum_{m=\ell}^{j-1} \sum_r \sum_{q=m-r}^{j-1}   \binom{m}{\ell} \binom{j-1+r-m}{r} \binom{q}{m-r} (-1)^{r} e_{q,jk}\\
&= (-1)^{\ell} e_{j-1-\ell,jk}.
\end{align*}
Similarly $\widetilde{w}_{-1} f_{\ell,jk} =(-1)^{j-1-\ell} f_{j-1-\ell,jk}$.
\end{longver}
\end{proof}

This proposition also gives the explicit action of $\GL_2(-1)$ on the modules $V^+_{jk} = \bigoplus_{\ell=1}^{j-1} \C e_{\ell,jk}$ and $V^-_{jk} = \bigoplus_{\ell=1}^{j-1} \C f_{\ell,jk}$ \cite{JurJPAA}.
Another immediate consequence is that the elements $H_1(s),H_2(s)\in\Aut(\mathfrak{m})$ are independent of the branch of the complex logarithm chosen.

%%%%%%%%%%%%%%%%%%%%%%%%%%%%%%%%%%%%%%%%%%%%%%%%%%%%%%
\subsection{The imaginary linear groups $\GL_2(\ell,j,k)$}\label{SS-gl2im}
\begin{proposition}
For $(\ell,j,k)\in E$,  
$\mathfrak{gl}_2(\ell,j,k)$ %:= \C\{e_{\ell,jk}, f_{\ell,jk}, h_1, h_2\}$
has standard basis  ${e_{\ell,jk}}$, ${f_{\ell,jk}}/{{c_{\ell j}}}$, ${h_1}/{(\ell+1)}$,
 ${-h_2}/{(j-\ell)}$, where $c_{\ell j} := (-1)^{\ell+1} \binom{j-1}{\ell} \, (\ell+1)(j-\ell).$
\end{proposition}

\begin{longver}
\begin{proof}  Let
\begin{align*}
 \overline{e}  &:= {e_{\ell,jk}}&
 \overline{f} &:= \frac{f_{\ell,jk}}{{c_{\ell j}}}, &
 \overline{h}_1 &:= \frac{h_1}{\ell+1},& 
 \overline{h}_2 &:=  \frac{-h_2}{j-\ell}.
\end{align*}
Then
\begin{align*}
 \left[\overline{h}_1,\overline{h}_2 \right] &= \frac{-\left[h_1.h_2 \right]}{(\ell+1)(j-\ell)}=0,\\
 \left[\overline{h}_1,\overline{e} \right] &=  \frac{\left[h_1,e_{\ell,jk} \right]}{(\ell+1)}   = \frac{(\ell+1)e_{\ell,jk}}{(\ell+1) }=\e,\\
 \left[\overline{h}_2,\overline{e} \right] &=  \frac{-\left[h_2,e_{\ell,jk} \right]}{(j-\ell)}   = \frac{-(j-\ell)e_{\ell,jk}}{(j-\ell) }=-\e,\\
 \left[\overline{h}_1,\overline{f} \right] &=  \frac{\left[h_1,f_{\ell,jk} \right]}{(\ell+1){c_{\ell,j}}}   = \frac{-(\ell+1)f_{\ell,jk}}{(\ell+1){c_{\ell,j}} }=-\f,\\
 \left[\overline{h}_2,\overline{f} \right] &=  \frac{-\left[h_2,f_{\ell,jk} \right]}{(j-\ell){c_{\ell,j}}}   = \frac{(j-\ell)f_{\ell,jk}}{(j-\ell){c_{\ell,j}} }=\f,\\ 
 \left[\overline{e},\overline{f} \right] &=  \frac{\left[e_{\ell,jk},f_{\ell,jk} \right]}{c_{\ell,j}}   =
 \frac{(j-\ell) h_1 +(\ell+1)h_2}{(\ell+1)(j-\ell)}=
 \frac{h_1}{\ell+1}+\frac{h_2}{j-\ell}=
  \overline{h}_1-\overline{h}_2.\qedhere
\end{align*}
\end{proof}
\end{longver}

Write
$$ X_{\ell,jk}\left(u\right) := \exp\left({u}\ad( e_{\ell,jk})\right) \quad\text{ and }\quad
 Y_{\ell,jk}\left(u\right) := \exp\left({u} \ad(f_{\ell,jk})\right),$$
where $\ad$ is the adjoint action on $\mathfrak{gl}_2(\ell,j,k)$.
\begin{corollary}
The group $\GL_2(\ell,j,k) \cong \GL_2(\C)$ has standard generators 
\begin{align*}
X_{\ell,jk}\left(u\right) &= \exp\left({u} \ad(e_{\ell,jk})\right),\qquad
&Y_{\ell,jk}\left(\frac{u}{{c_{\ell j}}}\right) &=\exp\left(\frac{u \ad(f_{\ell,jk})}{{c_{\ell j}}} \right) ,\\
H_1\left(s^{1/(\ell+1)}\right)&= \textstyle\exp \left(\frac{\log(s)\ad(h_1)}{(\ell+1)} \right),\;
&H_2\left(s^{-1/(j-\ell)}\right)&= \textstyle \exp \left(\frac{-\log(s)\ad(h_2)}{(j-\ell)} \right),\\
\widetilde{w}_{\ell,jk}(s) &=
 X_{\ell,jk}\left({s }\right) Y_{\ell,jk}\left(\frac{-s^{-1}}{{c_{\ell j}}}\right) 
X_{\ell,jk}\left({s}\right), \qquad &\widetilde{w}_{\ell,jk} &=\widetilde{w}_{\ell,jk}(1)
\end{align*}
for $u\in\C$ and $s\in\C^\times$.
\end{corollary}
Note that $\GL_2(\ell,j,k)$ acts on the subgroup $\mathfrak{gl}_2(\ell,j,k)$ of $\mathfrak{m}$ as automorphisms, but  this cannot be extended to a well-defined action on all of $\mathfrak{m}$. % by automorphisms.  

%CHECK THIS FOR EVEN VALUES OF j:
%\begin{comment} Magma code check:
%% Action of es
%j := 5;
%F<s> := RationalFunctionField(Rationals(),1);
%
%X := ZeroMatrix(F,j,j);
%for m in [0..j-1] do for l in [0..j-1] do
%  X[m+1,l+1] := Binomial(m,l)*s^(m-l);
%end for;  end for;
%X;
%
%Y := ZeroMatrix(F,j,j);
%for m in [0..j-1] do for l in [0..j-1] do
%  Y[m+1,l+1] := Binomial(j-1-m,j-1-l)*(-s^(-1))^(l-m);
%end for;  end for;
%Y;
%
%% Action of fs
%j := 5;
%F<s> := RationalFunctionField(Rationals(),1);
%
%X := ZeroMatrix(F,j,j);
%for m in [0..j-1] do for l in [0..j-1] do
%  X[m+1,l+1] := Binomial(j-1-m,j-1-l)*s^(l-m);
%end for;  end for;
%X;
%
%Y := ZeroMatrix(F,j,j);
%for m in [0..j-1] do for l in [0..j-1] do
%  Y[m+1,l+1] := Binomial(m,l)*(-s^(-1))^(m-l);
%end for;  end for;
%Y;
%
%X*Y*X;
%\end{comment}

We should note that there are many other $\frak{gl}_2$ subalgebras associated to the imaginary roots.
Suppose $\mathfrak{m}_{(\ell+1,j-\ell)}$ has dimension $d=c\left((\ell+1)(j-\ell) \right)>1$ with basis $e_{\ell,jk}=e_1$ and $e_2, \dots , e_d$. Then the Cartan involution gives a basis $f_1,\dots,f_d$ for $\mathfrak{m}_{-(\ell+1,j-\ell)}$ so that each  $\left\{e_i,f_i,h_1,h_2 \right\}$ generates a $\mathfrak{gl}_2$ subalgebra. However, $\left\{e_1,\dots,e_d,f_1,\dots,f_d,h_1,h_2 \right\}$ could generate an algebra that is much larger than a cartesian  product of $d$ copies of a $\mathfrak{gl}_2$ algebra, because we could have $\left[e_i,e_j \right]$ being a nonzero element of $\mathfrak{m}_{2(\ell+1,j-\ell)}$. 
So we have just restricted ourselves to constructing a single  $\mathfrak{gl}_2$ subalgebra corresponding to each free generator of our free  Lie subalgebras $\nimpm$.

%%%%%%%%%%%%%%%%%%%%%%%%%%%%%%%%%%%%%%%%%%%%%%%%%%%%%%
\subsection{Inner automorphism groups}\label{SS-inner}

For a Lie algebra $\frak L$,  we recall that %, as usual,
 $$\Inn(\frak L)=\left\langle \exp(\ad x)\mid x\in\frak L\text{  is locally $\ad$-nilpotent} \right\rangle.$$
\begin{proposition}
 $x\in\m$ is locally $\ad$-nilpotent if and only if $x$ is a nilpotent element of $\mathfrak{gl}_2(-1)$.
\end{proposition}
\begin{proof} 
Let $x$ be $\ad$-nilpotent and write $x=x_-+x'+x_+$ with $x_\pm\in\nimpm$ and $x'\in \gl_2(-1)$.
Then $x_+$ is a (finite) linear combination of Lie products of the free generators $f_{\ell,jk}$ for $(\ell,j,k)\in I^{\im}$. Choose $m$ so that $m-1$ is greater than every $\ell$ from a nonzero term in this combination, and $p>m, q=1$.
Then $[x_{-},e_{m,pq}]=0$ by Theorem~\ref{T-identities-v2}(\ref{L-ef0}).
Now $[x,e_{m,pq}] = [x_++x',e_{m,pq}]$.
From properties of free Lie algebras, we can easily choose $(m,p,q)$ so that $\ad(x_+)^n(e_{m,pq})$ is nonzero for all $n\ge 1$. Hence $x_+=0$. Similarly $x_-=0$ and so $x=x'\in\gl_2(-1)$.
The converse is straightforward.
\end{proof}
\begin{corollary}
$\Inn(\frak m) = \left\langle X_{-1}(u), Y_{-1}(u) \mid u\in\C \right\rangle \cong \SL_2(\C).$
\end{corollary}

In contrast, the center of $\GL_2(\C)$ acts trivially on $\gl_2$, which gives the following.
\begin{corollary} $\Inn\left(\frak{gl}_2(\ell,j,k) \right)\cong\PGL_2(\C)$.
\end{corollary}

%%%%%%%%%%%%%%%%%%%%%%%%%%%%%%%%%%%%%%%%%%%%%%%%%%%%%%
%%%%%%%%%%%%%%%%%%%%%%%%%%%%%%%%%%%%%%%%%%%%%%%%%%%%%%
\section{Complete parabolic subgroups and the Adjoint representation}\label{S-related}
In this section, we define   complete  parabolic groups $\Php$ associated to $\frak m$ and we establish an analog of the adjoint representation $\Ad$ on $\Php$. We also prove the relations for $\Php$.

%%%%%%%%%%%%%%%%%%%%%%%%%%%%%%%%%%%%%%%%%%%%%%%%%%%%%%
\subsection{Definition of the complete parabolic subgroup}\label{SS-parab}
We consider a negative completion of $\m$:
$$\widehat{\frak m}^- = \nhimm \oplus \frak{gl}_2({-1}) \oplus \nimp 
  =  \nhm \oplus \frak h \oplus \np$$
where
$\nhm = \mathfrak{m}_{-\a_{-1}} \oplus\nhimm$ and $\np = \mathfrak{m}_{\a_{-1}} \oplus\nimp$.
In Subsection~\ref{SS-PositiveRootSys}, we  defined $\Uhm\subseteq \Aut\left(\widehat{\frak m}^-\right)$. 

Recall that $\GL_2(-1)$ consists of automorphism of $\mathfrak{m}$, which can be extended to automorphisms of either $\mhat=\mhat^+$ or $\mhat^-$. Recall also the definitions of $H_1(s)$, $H_2(s)$, $X_{-1}(u)$, $Y_{-1}(u)$, $\widetilde{w}_{-1}\in\GL_2(-1)$.
We define subgroups
$$ H:= \left\{H_1(s),H_2(t) \mid s,t\in\C^\times\right\},$$
$$  \Up_{-1} := \left\{X_{-1}(u)\mid u\in\C\right\},\qquad
\Um_{-1} := \left\{Y_{-1}(u)\mid u\in\C\right\}.$$
We can take
\begin{align*}
 \Uhpm &:= \left\langle \Upm_{-1}, \Uhimpm\right\rangle,\\
%\intertext{Borel groups}
 %\Bhpm &:= \left\langle H, \Uhpm\right\rangle,\\
\intertext{and we define parabolic groups}
 \Phpm &:= \left\langle \GL_2(-1), \Uhpm\right\rangle.
\end{align*}

The Cartan involution $\eta$ extends to an isomorphism $\eta: \mhat\to \mhat^-$, which induces an isomorphism
$$\eta: \Php \to \Phm.$$
Hence, for every property of $\Php$ that we prove, there is an analogous property of $\Phm$.
Note that $\UhPi$ is not generally closed under the action $\GL_2(-1)$, so we cannot define parabolic subgroups corresponding to an arbitrary positive subsystem $\Pi$.

%\begin{proposition}\label{C-UBP} The group
%$\Uhp$ is a complete pro-unipotent group.% and $\Uhp = \Up_{-1} \ltimes \Uhimp$.
%$\Php$ is a complete pro-algebraic group. Furthermore,
%$$\Uhp = \Up_{-1} \ltimes \Uhimp,\qquad \Php = \GL_2(-1)\ltimes \Uhimp.$$
%\end{proposition}
% \Bhp = H\ltimes \Uhp,

%%%%%%%%%%%%%%%%%%%%%%%%%%%%%%%%%%%%%%%%%%%%%%%%%%%%%%
\subsection{Adjoint representation}\label{SS-adjoint}

In this subsection, we construct an analog of the adjoint representation $\Ad$ for $\Php$.
It is easily checked using Theorem~\ref{T-identities-v2} that $\nhp_n$ for $n\in\N$, $\nhimp$, and $\nhp$ are ideals
in $$\widehat{\mathfrak{p}}^+ := \gl_2(-1)\oplus\ \nhimp.$$
We also define
$$\widehat{\mathfrak p}^{-}
:=
\widehat{\mathfrak n}^{-}_{\mathrm{im}}
\oplus
\mathfrak{gl}_2(-1)$$
$${\mathfrak{p}}^+ := \gl_2(-1)\oplus\ \nimp$$
and
$${\mathfrak p}^{-}
=
{\mathfrak n}^{-}_{\mathrm{im}}
\oplus
\mathfrak{gl}_2(-1).$$
The subalgebra $\nhp$ is pro-nilpotent and $\widehat{\mathfrak{p}}^+$ is a pro-Lie algebra.

Let
 $$\Ad_n:\Php / \Uhp_n \to\Aut\left(\php/\nhp_n\right)$$
be the adjoint representation of the finite dimensional linear algebraic group $\Php/ \widehat{U}^+_n$ and let
$$\Exp_n: \php/\nhp_n  \to \Php / \Uhp_n$$
be the exponential map. 
Then the following is a standard result:
$$\Exp_n\left(\Ad_n(g)\right)(x)=g\Exp_n(x) g^{-1}$$
for $x\in  \php/\nhp_n$, and $g \in  \Php / \Uhp_n$.

Taking inverse limits, as in Definition~4.4.25 of \cite{kumar2012kac}, allows us to define $\Ad$ and $\Exp$ as the unique maps making the following diagrams commute:
\begin{center}
\begin{tikzcd}
\php \arrow[r] \arrow[d, "\Exp", swap]
& \php/\nhp_n \arrow[d, "\Exp_n"] \\
\Php \arrow[r]
& \Php/\Uhp_n
\end{tikzcd}
$\qquad\qquad$
\begin{tikzcd}
\Php \arrow[r] \arrow[d, "\Ad", swap]
& \Php/\Uhp_n \arrow[d, "\Ad_n"] \\
\Aut\left(\php\right) \arrow[r]
& \Aut\left(\php/\nhp_n\right).
\end{tikzcd}
\end{center}

Now take $x\in \widehat{\mathfrak{n}}$ and $g\in \widehat{{P}}^+$.
So we have
$$\Exp\left(\Ad(g)\right)(x)=g\Exp(x) g^{-1}.$$
Since $\Php / \Uhp_1 \cong \GL_2(\C)$ is connected, we can identify it with the finite dimensional algebraic group 
$$\exp\left(\ad_{\php/\nhp_n}\left(\php / \nhp_n\right)\right)\subseteq\End\left(\php/\nhp_n\right).$$
With this identification, we get
$$ \Exp_n\left(x+\nhp_n\right) = \exp\left(\ad_{\php/\nhp_n}\left(x+\nhp_n\right)\right)\quad\text{and}$$
$$\quad \Ad_n\left(g\cdot\Uhp_n\right)\left(x\cdot\nhp_n\right)=\left(g+\Uhp_n\right)\left(x+\nhp_n\right),$$
and so in the limit
$$ \Exp(x) = \exp\left(\ad_{\php}(x)\right)\quad\text{and}\quad \Ad(g)(x)=gx.$$
Thus we have
$$ \exp\left(\ad_{\php}(gx)\right)= g\left(\exp\left(\ad_{\php}(x)\right)\right)g^{-1}$$
It is easily checked that an element of $\Php$ that acts as the identity on $\php$ must be the identity on all of $\mhat$.
Hence we have proved:
\begin{theorem}\label{T-adjoint}
If $g\in\Php$ and $x\in \nhp$, then
$$ \exp\left(\ad(gx)\right)= g\left(\exp\left(\ad(x)\right)\right)g^{-1}.$$
%holds in $\Php\subseteq\Aut\left(\mhat\right)$.
\end{theorem}

We now prove some relations in $\Php$:% which uses  the definition of the roots $\a_{\ell,jk}$.
\begin{corollary}
If $x_{\ell,jk}\in\m_{\a_{\ell,jk}}$, $s\in\C^\times$ $u\in\C$, then
\begin{enumerate}
%item $(X_{-1}(s),\exp(t\ad(x_{j-1,jk}))) =(Y_{-1}(s),\exp(t\ad(x_{0,jk}))) = 1$,
\item $H_1(s)\exp\left(u\ad\left(x_{\ell,jk}\right)\right)H_1(s)^{-1} = \exp\left(s^{\ell-1}u\ad\left(x_{\ell,jk}\right)\right)$,
\item $H_2(s)\exp\left(u\ad\left(x_{\ell,jk}\right)\right)H_2(s)^{-1} = \exp\left(s^{j-\ell}u\ad\left(x_{\ell,jk}\right)\right)$.
\item $\widetilde{w}_{-1}\exp(t\ad(x_{\ell,jk}))\widetilde{w}_{-1}^{-1} = \exp(t\ad(\widetilde{w}_{-1}x_{\ell,jk}))$.
\end{enumerate}
\end{corollary}

%%%%%%%%%%%%%%%%%%%%%%%%%%%%%%%%%%%%%%%%%%%%%%%%%%%%%%
In summary, we have hierarchies of Lie algebras and groups as in Figure~\ref{fig:lie-algebra-hierarchy} and Figure~\ref{fig:complete-parabolic-groups}.
The uncompleted
parabolic subalgebras \(\mathfrak p^{+}\) and \(\mathfrak p^{-}\) embed
respectively in the completed parabolic subalgebras
\(\widehat{\mathfrak p}^{+}\) and \(\widehat{\mathfrak p}^{-}\). These
completed parabolic subalgebras embed in the corresponding one-sided
completions \(\widehat{\mathfrak m}^{+}\) and
\(\widehat{\mathfrak m}^{-}\). The subalgebra \(\mathfrak{gl}_2(-1)\)
embeds as a common Levi factor of \(\mathfrak p^{+}\) and
\(\mathfrak p^{-}\).
The groups
\(\widehat U^{+}\) and \(\widehat U^{-}\) embed respectively in
\(\widehat P^{+}\) and \(\widehat P^{-}\), while
\(\mathrm{GL}_2(-1)\) embeds as a common subgroup of both complete
parabolic groups.
\begin{figure}[ht]
\centering
\[
\begin{tikzpicture}[
  every node/.style={inner sep=2pt, align=center},
  incl/.style={->, thick}
]

\node (Mhatp) at (-4.6,5.8)
{$\widehat{\mathfrak m}^{+}
=
\mathfrak n^- \oplus \mathfrak h \oplus \widehat{\mathfrak n}^{+}$};

\node (Mhatm) at (4.6,5.8)
{$\widehat{\mathfrak m}^{-}
=
\widehat{\mathfrak n}^{-} \oplus \mathfrak h \oplus \mathfrak n^{+}$};

\node (m) at (0,4.1)
{$\mathfrak m
=
\mathfrak n^-_{\mathrm{im}}
\oplus
\mathfrak{gl}_2(-1)
\oplus
\mathfrak n^+_{\mathrm{im}}$};

\node (Phatp) at (-4.6,2.6)
{$\widehat{\mathfrak p}^{+}
=
\mathfrak{gl}_2(-1)
\oplus
\widehat{\mathfrak n}^{+}_{\mathrm{im}}$};

\node (Phatm) at (4.6,2.6)
{$\widehat{\mathfrak p}^{-}
=
\widehat{\mathfrak n}^{-}_{\mathrm{im}}
\oplus
\mathfrak{gl}_2(-1)$};

\node (Pp) at (-4.6,1.1)
{$\mathfrak p^{+}
:=
\mathfrak{gl}_2(-1)
\oplus
\mathfrak n^{+}_{\mathrm{im}}$};

\node (Pm) at (4.6,1.1)
{$\mathfrak p^{-}
=
\mathfrak n^{-}_{\mathrm{im}}
\oplus
\mathfrak{gl}_2(-1)$};

\node (gl2) at (0,-0.6)
{$\mathfrak{gl}_2(-1)$};

\draw[incl] (m) -- (Mhatp);
\draw[incl] (m) -- (Mhatm);

\draw[incl] (Phatp) -- (Mhatp);
\draw[incl] (Phatm) -- (Mhatm);

\draw[incl] (Pp) -- (Phatp);
\draw[incl] (Pm) -- (Phatm);

\draw[incl] (Pp) -- (m);
\draw[incl] (Pm) -- (m);

\draw[incl] (gl2) -- (Pp);
\draw[incl] (gl2) -- (Pm);

\end{tikzpicture}
\]
\caption{Hierarchy of Lie subalgebras and completions }
\label{fig:lie-algebra-hierarchy}
\end{figure}

\begin{figure}[ht]
\centering
\[
\begin{tikzpicture}[
  every node/.style={inner sep=2pt},
  incl/.style={->, thick}
]
\node (Pp) at (-2.7,4.2) {$\widehat P^{+}$};
\node (Pm) at ( 2.7,4.2) {$\widehat P^{-}$};

\node (Up) at (-2.7,2.6) {$\widehat U^{+}$};
\node (Um) at ( 2.7,2.6) {$\widehat U^{-}$};

\node (Uimp) at (-1.8,1.2) {$\widehat U^{+}_{\mathrm{im}}$};
\node (Uimm) at ( 1.8,1.2) {$\widehat U^{-}_{\mathrm{im}}$};

\node (G) at (0,0) {$\mathrm{GL}_2(-1)$};

\draw[incl] (Up) -- (Pp);
\draw[incl] (Um) -- (Pm);

\draw[incl] (Uimp) -- (Up);
\draw[incl] (Uimm) -- (Um);

\draw[incl] (G) -- (Pp);
\draw[incl] (G) -- (Pm);
\end{tikzpicture}
\]
\caption{Hierarchy of the complete parabolic  and unipotent groups}
\label{fig:complete-parabolic-groups}
\end{figure}
%%%%%%%%%%%%%%%%%%%%%%%%%%%%%%%%%%%%%%%%%%%%%%%%%%%%%%
\section{Compatibility of $\M$-actions}\label{S-compat}
In this section, we define the action of the Monster finite simple group $\M$ on ~$\widehat{\mathfrak p}^+$ and on the complete parabolic group~$\Php$. 

%{\color{red}\sout{We also show that these actions are compatible with each other.}}
Let $V^\natural$ be the Moonshine module of \cite{FLM}. Borcherds proved  \cite{BoInvent} that $$\mathfrak m_{(m,n)}\cong V^\natural_{mn+1}$$ as  $\mathbb{M}$-modules for $(m,n)\neq (0,0)$, and it follows easily that
$$\m_{(1,-1)}\oplus \m_{(0,0)}\oplus \m_{(-1,1)}= \gl_2(-1)$$
 is a trivial $\mathbb{M}$-module. This gives an action $\mathbb{M}   \to    \Aut(\mathfrak m)$ acting trivially on $\gl_2(-1)$ and preserving root spaces.

\begin{theorem}\label{action} The action of $\M$ on $\mathfrak m$ induces an action of $\M$ on $\mhat$. This in turn induces an action of $\M$ on~$\Php$. 
For $g\in\M$, $u\in\Php$, $x\in\mhat$, we have
$$ g\cdot(u\cdot x) = (g\cdot u)\cdot (g\cdot x).$$
\end{theorem}

\begin{proof} As in Section~\ref{S-unip}, we use the grading
${\frak{m}}=\bigoplus_{n\in\Z} \frak{m}_{n}$, where $$\mathfrak{m}_n := \bigoplus_{i\in\Z} \mathfrak{m}_{(i,n-i)}.$$
%, so that $\mathfrak{m}_0=\gl_2$.
The group $\M$ acts on each $\frak{m}_{n}$. %and this  action is trivial on $\mathfrak{m}_0$.
This induces an action of  $\mathbb{M}$  on $\mhat$:
every $x \in \mhat$ can be written as
$x =\sum_{n=N(x)}^\infty x_n$ for $x_n\in \frak{m}_{n}$ and
the action of $g\in\M$ is
  $$g\cdot x :=  \sum_{n=N(x)}^\infty g\cdot x_n. $$
This in turn induces an action of  $\mathbb{M}$  on $\Uhimp$:
every $u\in \Uhimp$ can be written as
$u =\prod_{n=1}^\infty \exp(\ad(y_n))$ for $y_n\in \frak{m}_{n}$, and
the action of $g\in\M$ is
  $$g\cdot  u :=  \prod_{n=1}^\infty \exp(\ad(g\cdot y_n)). $$
Since $\mathbb{M}$ acts trivially on $\m_0=\gl_2(-1)$, the action on $\GL_2(-1)$ must also  be trivial, and so the action of $\M$ on $\Uhimp$ extends to $\Php$.
The formula for $g\cdot(u\cdot x)$ is now clear.
\end{proof}

\begin{corollary}\label{compat} For every $g\in\mathbb{M}$, the following diagram commutes:
\begin{center}
\begin{tikzcd}
\php \arrow[r, "g"] \arrow[d, "\Exp", swap]
& \php \arrow[d, "\Exp"] \\
\Php \arrow[r,"g"]
& \Php
\end{tikzcd}
\end{center}
where $\Exp$ is defined as in Subsection~\ref{SS-adjoint}.
\end{corollary}

%%%%%%%%%%%%%%%%%%%%%%%%%%%%%%%%%%%%%%%%%%%%%%%%%%%%%%
%%%%%%%%%%%%%%%%%%%%%%%%%%%%%%%%%%%%%%%%%%%%%%%%%%%%%%
\section{A Lie group analog for $\mathfrak{m}$}\label{S-Grels}
In this section, we define an analog of a  Kac--Moody group associated to $\frak m$. This group, denoted~$G(\frak m)$, will be given by an infinite presentation.
%generators and defining relations. 
%{\bf{\color{blue}In Section~\ref{S-unip} and Section~\ref{S-related} we show that these relations hold in the automorphism group of a completion of $\m$.}}

To motivate our   construction of the group $G(\frak m)$, recall that the adjoint Kac--Moody group corresponding to a symmetrizable Kac--Moody algebra $\mathfrak{g}$ can be constructed as automorphisms of  $\mathfrak{g}$. However, as explained in the introduction, it is not possible to construct all the group elements of $G(\frak m)$ as automorphisms of $\frak m$. 
%We will see in Section~\ref{S-unip}, that we can construct the analog of the positive unipotent group as automorphisms of a completion $\widehat{\mathfrak{m}}$, and this group is pro-unipotent. %{\sout{However the negative unipotent group doesn't have a well-defined action on $\widehat{\mathfrak{m}}$. Instead it acts on a different completion $\widehat{\mathfrak{m}}^-$.} 
%We will see in Section~\ref{} that the subgroup $U^+_{\im}$ acts on a completion $\widehat{\frak n}+$ of ${\frak n}+$
%Using this completion we construct an analog of a parabolic group $P^+$. Other completions allow us to construct other unipotent and parabolic subgroups in Section~\ref{S-related}.
We give generators and relations for a group $\GL_2(-1)$ (constructed in Subsection~\ref{SS-gl2re}), corresponding to the unique  index $(-1,1)\in I^{\re}$. The unipotent elements of this group act locally nilpotently on the whole Lie algebra~$\mathfrak{m}$, hence $\GL_2(-1)$ acts as automorphisms of $\frak m$. We also give generators and relations for the groups $\GL_2(\ell,j,k)$ constructed in Subsection~\ref{SS-gl2im} that are automorphisms of the subalgebras $\frak{gl}_2(\ell,j,k)$, corresponding to each extended index $(\ell,j,k)\in E$.

We conjecture that the group $G(\frak m)$ is an amalgam of the $\GL_2$ subgroups and pro-unipotent subgroups discussed in the previous two sections. We hope to investigate this in upcoming work. %(see  \cite{ACJM1}).

\begin{defn}\label{Defn}
Let $\mathcal{X}$ be the set of symbols
$$\mathcal{X} :=\{H_1(s), H_2(s), X_{-1}(u),  Y_{-1}(u), X_{\ell,jk}(u), Y_{\ell,jk}(u) \mid$$ 
$$\qquad \qquad \qquad \qquad\qquad \qquad \qquad \qquad s\in \C^\times, u\in \C, (\ell,j,k)\in E\},$$
where $E=  \left\{(\ell,j,k)\mid j\in\N,\, 1 \leq k\leq c(j),\, 0\leq \ell<j \right\}$ as in Subsection~\ref{SS-decomp}. Let $F(\mathcal{X})$ denote the free group on~$\mathcal{X}$.

We  call  $X_{-1}(u)$, $Y_{-1}(u)$, $X_{\ell,jk}(u)$ and $Y_{\ell,jk}(u)$ \emph{unipotent generators}.

Define the constant
$$c_{\ell j} := (-1)^{\ell+1} \binom{j-1}{\ell} \, (\ell+1)(j-\ell).$$
Define the set of relations $\mathcal{R}\subseteq F(\mathcal{X})$  and additional symbols $ \widetilde{w}_{-1}(s)$ and   $\widetilde{w}_{\ell,jk}(s)$ as follows, for all $s,t\in \C^\times$, $u,v\in \C$, $(\ell,j,k),(m,p,q) \in E$.
\begingroup\allowdisplaybreaks
\begin{align*}
\tag{Re:0}\label{GL2-G-wdefn} 
  \widetilde{w}_{-1}(s) &:= X_{-1}(s) Y_{-1}\left(-s^{-1} \right) X_{-1}(s), \\ \widetilde{w}_{-1} &:= \widetilde{w}_{-1}(1),\\
\tag{Im:0}\label{G-wdefn}
  \widetilde{w}_{\ell,jk}(s) &:= X_{\ell,jk}\left({s}\right) Y_{\ell,jk}\left(\frac{-s^{-1}}{{c_{\ell j}}}\right) 
  X_{\ell,jk}\left({s}\right),\\  \widetilde{w}_{\ell,jk} &:= \widetilde{w}_{\ell,jk}(1),\\
%\intertext{Relations that define a toral subgroup $H$:}
\intertext{Relations from the group $\GL_2(-1)$:}
\tag{H:1a}\label{GL2-G-H1H1}
   H_1(s)H_1(t)&=H_1(st),\\
\tag{H:1b}\label{GL2-G-H2H2}
   H_2(s)H_2(t)&=H_2(st),\\
\tag{H:2}\label{GL2-G-H1H2}
    H_1(s)H_2(t)&=H_2(t)H_1(s),\\
%\intertext{Relations that define a subgroup $\GL_2(-1)$:}
\tag{Re:1a}\label{GL2-G-XX}
   X_{-1}(u)X_{-1}(v)&=X_{-1}(u+v),\\
\tag{Re:1b}\label{GL2-G-YY}
   Y_{-1}(u)Y_{-1}(v)&=Y_{-1}(u+v),\\
\tag{Re:2}\label{GL2-G-XY}
  Y_{-1}(-t)X_{-1}(s)Y_{-1}(t)&= X_{-1}\left(-t^{-1}\right) Y_{-1}\left(-t^{2}s \right)X_{-1}\left(t^{-1} \right),\\ 
\tag{Re:3}\label{GL2-G-wwH}
  \widetilde{w}_{-1}(s)\widetilde{w}_{-1}&=H_1(-s)H_2\left(-s^{-1} \right), \\
\tag{Re:4a}\label{GL2-G-wX}
  \widetilde{w}_{-1}X_{-1}(u)\widetilde{w}_{-1}^{-1}&=Y_{-1}(-u),\\
\tag{Re:4b}\label{GL2-G-wY}
  \widetilde{w}_{-1}Y_{-1}(u)\widetilde{w}_{-1}^{-1}&=X_{-1}(-u),\\
%\intertext{}
%\tag{$\GL_2$-G2b}  X_{-1}(-s)Y_{-1}(t)X_{-1}(s) &=Y_{-1}(t), \text{\textcolor{red}{REMOVE}}\\
%\tag{$\GL_2$-G2c}  Y_{-1}(-s)X_{-1}(t)Y_{-1}(s) &=X_{-1}(t), \text{\textcolor{red}{REMOVE}}\\
\tag{Re:5a}\label{GL2-G-wH1}
  {\widetilde{w}_{-1}} H_1(s)\widetilde{w}_{-1}^{-1} &= H_2(s),\\
\tag{Re:5b}\label{GL2-G-wH2}
  {\widetilde{w}_{-1}} H_2(s)\widetilde{w}_{-1}^{-1} &= H_1(s),\\
\tag{Re:6a}\label{GL2-G-H1X}
  H_1(s)X_{-1}(u)H_1(s)^{-1}&=X_{-1}(su),\\
\tag{Re:6b}\label{GL2-G-H2X}
  H_2(s)X_{-1}(u)H_2(s)^{-1}&=X_{-1}\left(s^{-1}u \right),\\
\tag{Re:6c}\label{GL2-G-H1Y}
  H_1(s)Y_{-1}(u)H_1(s)^{-1}&=Y_{-1}\left(s^{-1}u \right),\\
\tag{Re:6d}\label{GL2-G-H2Y}
  H_2(s)Y_{-1}(u)H_2(s)^{-1}&=Y_{-1}(su),\\
\intertext{Relations from the groups $\GL_2(\ell,j,k)$:}
\tag{Im:1a}\label{G-XX}
  X_{\ell,jk}(u)X_{\ell,jk}(v)&=X_{\ell,jk}(u+v),\\
\tag{Im:1b}\label{G-YY}
  Y_{\ell,jk}(u)Y_{\ell,jk}(v)&=Y_{\ell,jk}(u+v),\\
\tag{Im:2}\label{G-XY}
  Y_{\ell,jk}(-t)X_{\ell,jk}(s)Y_{\ell,jk}(t)&= \\
  \qquad \qquad X_{\ell,jk}&\left(\frac{-t^{-1}}{c_{\ell j}}\right)Y_{\ell,jk}\left(-c_{\ell j}t^2s\right)X_{\ell,jk}\left(\frac{t^{-1}}{c_{\ell j}}\right),\\
\tag{Im:3}\label{G-wwH}
  \widetilde{w}_{\ell,jk}\left(s^{(\ell+1)(j-\ell)} \right)\widetilde{w}_{\ell,jk}&=H_1\left((-s)^{j-\ell}\right) H_2\left((-s)^{\ell+1}\right), \\
\tag{Im:4a}\label{G-wX}
  \widetilde{w}_{\ell,jk} X_{\ell,jk}(u) {\widetilde{w}_{\ell,jk}}^{-1}&= Y_{\ell,jk}\left(\frac{-u}{c_{\ell j}}\right),\\
\tag{Im:4b}\label{G-wY}
  \widetilde{w}_{\ell,jk} Y_{\ell,jk}(u) {\widetilde{w}_{\ell,jk}}^{-1}&= X_{\ell,jk}\left(-c_{\ell j}u \right),\\
\tag{Im:5a}\label{G-wH1}
  \widetilde{w}_{\ell,jk} H_1\left(s^{j-\ell} \right) {\widetilde{w}_{\ell,jk}}^{-1}&=  H_2\left(s^{-(\ell+1)}\right),\\
\tag{Im:5b}\label{G-wH2}
  \widetilde{w}_{\ell,jk} H_2\left(s^{\ell+1} \right) {\widetilde{w}_{\ell,jk}}^{-1}&=  H_1\left(s^{-(j-\ell)}\right),\\
\tag{Im:6a}\label{G-H1X}
  H_1(s) X_{\ell,jk}(u) H_1(s)^{-1}&=X_{\ell,jk}\left(s^{\ell+1}u \right),\\
\tag{Im:6b}\label{G-H2X}
  H_2(s) X_{\ell,jk}(u) H_2(s)^{-1}&=X_{\ell,jk}\left(s^{j-\ell}u \right),\\
\tag{Im:6c}\label{G-H1Y}
  H_1(s) Y_{\ell,jk}(u) H_1(s)^{-1}&=Y_{\ell,jk}\left(s^{-(\ell+1)}u \right),\\
\tag{Im:6d}\label{G-H2Y}
  H_2(s) Y_{\ell,jk}(u) H_2(s)^{-1}&=Y_{\ell,jk}\left(s^{-(j-\ell)}u \right),\\
\intertext{Relations between the unipotent generators:}
\tag{U:1a}\label{G-X-1X}
  (X_{-1}(u), X_{j-1,jk}(v))&=1,\\
\tag{U:1b}\label{G-Y-1Y} 
  (Y_{-1}(u), Y_{j-1,jk}(v))&=1.\\
\tag{U:1c}\label{G-Y-1X}
  (Y_{-1}(u), X_{0,jk}(v))&=1,\\
\tag{U:1d}\label{G-X-1Y}
  (X_{-1}(u), Y_{0,jk}(v))&=1,\\
\tag{U:2}\label{G-XYcomm}
  (X_{\ell,jk}(u), Y_{m,pq}(v))&=1\qquad\text{for $j\neq p$, or $k\neq q$, or $|\ell-m|>1$,}\\
\tag{U:3a}\label{G-XYj=2a} (X_{1,2k}(u),Y_{0,2k}(v))&=X_{-1}(uv),\\
\tag{U:3b}\label{G-XYj=2b}(X_{0,2k}(u),Y_{1,2k}(v))&=Y_{-1}(-uv),\\
\intertext{Action of $\widetilde{w}_{-1}$ on the unipotent generators:}
\tag{U:4a}\label{G-w-X}
  \widetilde{w}_{-1}X_{\ell,jk}(u)\widetilde{w}_{-1}^{-1}&=X_{j-1-\ell,jk}\left((-1)^{\ell} u \right), \\
\tag{U:4b}\label{G-w-Y}
  \widetilde{w}_{-1}Y_{\ell,jk}(u)\widetilde{w}_{-1}^{-1}&=Y_{j-1-\ell,jk}\left((-1)^{j-\ell-1} u \right), 
\end{align*}
\begin{align*} 
\tag{U:5a}\label{GL2-G-Yelim}
 Y_{-1}(s)& =X_{-1}\left(s^{-1} \right)H_1 \left(-s^{-1} \right)H_2 \left(-s \right)\widetilde{w}_{-1}X_{-1} \left(s^{-1} \right).\\
\tag{U:5b}\label{G-Yelim}
Y_{\ell,jk}(s) &=
X_{\ell,jk}\left(\frac{ s^{-1}}{c_{\ell j}}\right)
H_1\left(\left[-{{c_{\ell j}}}s\right]^{-1/(\ell+1)}\right)
H_2\left(\left[-{{c_{\ell j}}}s\right]^{-1/(j-\ell)}\right)\widetilde{w}_{\ell,jk}X_{\ell,jk}\left(\frac{ s^{-1}}{c_{\ell j}}\right).
\end{align*}

\endgroup

We now define $G(\mathfrak{m})$ as the group given by this presentation, that is,
$${G}(\mathfrak{m})= \left\langle \mathcal{X}\mid \mathcal{R}\right\rangle={F}(\mathcal{X})/N_\mathcal{R},$$
where $N_\mathcal{R}$ is the normal closure of the relations $\mathcal{R}$.
\end{defn}

We do not claim that this presentation is  minimal -- some of the relations 
follow from the other relations but are included for their usefulness. 

\begin{defn}\label{subgps}  We define the following subgroups of $G(\m)$:
\begin{align*}
\GL_2(-1) &:= \langle X_{-1}(u), Y_{-1}(u), H_1(s), H_2(s)\mid u\in\C,\,s\in\C^\times\rangle,\\
\GL_2(\ell,j,k) &:= \langle X_{\ell,jk}(u), Y_{\ell,jk}(u), H_1(s), H_2(s)\mid u\in\C,\,s\in\C^\times\rangle,\\
\intertext{for $(\ell,j,k)\in E$,}
\Pp &:= \left\langle \GL_2(-1), X_{\ell,jk}(u) \mid (\ell,j,k)\in E,\,u\in\C\right\rangle,\\
\Pm &:= \left\langle \GL_2(-1), Y_{\ell,jk}(u) \mid  (\ell,j,k)\in E,u\in\C\right\rangle,\\
U_{\ell,jk;\,m,pq} &:= \left\langle X_{\ell,jk}(u), Y_{m,pq}(u) \mid u\in\C\right\rangle,
\end{align*}
for distinct $(\ell,j,k),\,(m,p,q)\in E$.
\end{defn}
%where $\Pi$ is a positive root system as defined in Subsection~\ref{SS-pos}.

The following is a restatement of Corollary~\ref{C-comm} and represents the commutator relations in $G(\mathfrak m)$ where the product on 
the right hand side is finite.
\begin{corollary} $\;$\label{C-comm}
The following relations hold in every $\widehat{U}^{\,\Pi}$ in which the elements involved are defined:
\begin{enumerate}
\item \eqref{GL2-G-XX},  \eqref{GL2-G-YY}, \eqref{G-XX}, \eqref{G-YY};
\item $\left(X_{-1}(u),X_{j-1,jk}(v)\right) = \left(Y_{-1}(u),Y_{j-1,jk}(v)\right) = 1$, that is, \eqref{G-X-1X} and \eqref{G-Y-1Y};
\item $\left(X_{-1}(u),Y_{0,jk}(v)\right) = \left(Y_{-1}(u),X_{0,jk}(v)\right) = 1$, that is, \eqref{G-X-1Y} and \eqref{G-Y-1X};
\item $\left(X_{\ell,jk}(u),Y_{m,pq}(v)\right) =1$ for $j\ne p$, $k\ne q$, or $\left|\ell-m\right|>1$, that is, \eqref{G-XYcomm}.
\item $\left(X_{1,2k}(u),Y_{0,2k}(v)\right)=X_{-1}(uv)$ and $\left(X_{0,2k}(u),Y_{1,2k}(v)\right)=Y_{-1}(-uv)$.
\end{enumerate}
\end{corollary}
In particular, (2) holds in $\Uhpm$ and \eqref{GL2-G-XX},  \eqref{G-XX}.

We can now show that $\GL_2(-1)$ satisfies the relations given in Definition~\ref{Defn} for this subgroup. To be precise:
\begin{lemma}\label{L-GL2-1}
The subgroup $\GL_2({-1})\subseteq \Aut(\mathfrak{m})$ has generators
$H_1(s)$, $H_2(s)$, $X_{-1}(u)$,  $Y_{-1}(u)$ for  $s\in\C^\times,\,u\in\C $,
and  relations \eqref{GL2-G-H1H1}, \eqref{GL2-G-H2H2}, \eqref{GL2-G-H1H2}, \eqref{GL2-G-XX}, 
\eqref{GL2-G-YY}, \eqref{GL2-G-XY}, \eqref{GL2-G-wwH}, \eqref{GL2-G-wX}, \eqref{GL2-G-wY}, \eqref{GL2-G-wH1},  
\eqref{GL2-G-wH2}, \eqref{GL2-G-H1X}, \eqref{GL2-G-H2X}, \eqref{GL2-G-H1Y},  \eqref{GL2-G-H2Y} and \eqref{GL2-G-Yelim}.
\end{lemma}

\begin{longver}
\begin{proof} These follow from the corresponding relations \eqref{GL2-H1H1}, \eqref{GL2-H2H2}, \eqref{GL2-H1H2}, \eqref{GL2-XX}, 
\eqref{GL2-YY}, \eqref{GL2-XY}, \eqref{GL2-wwH}, \eqref{GL2-wX}, \eqref{GL2-wY}, \eqref{GL2-wH1},  
\eqref{GL2-wH2}, \eqref{GL2-H1X}, \eqref{GL2-H2X}, \eqref{GL2-H1Y}, \eqref{GL2-H2Y} and \eqref{GL2-Yelim}
\end{proof}
\end{longver}

We can now show that each $\GL_2(\ell,j,k)$ satisfies those relations given in Definition~\ref{Defn} that  involve the generators $X_{\ell,jk}(u)$ and  $Y_{\ell,jk}(u)$.
\begin{lemma}\label{L-GL2ljk}
The subgroup $\GL_2({\ell,j,k})\subseteq\Aut(\mathfrak{gl}_2(\ell,j,k))$ has generators
$H_1(s)$, $H_2(s)$, $X_{\ell,jk}(u)$,  $Y_{\ell,jk}(u)$ for  $s\in\C^\times,\ u\in\C $,
and defining relations \eqref{GL2-G-H1H1}, \eqref{GL2-G-H2H2}, \eqref{GL2-G-H1H2}, \eqref{G-XX}, 
\eqref{G-YY}, \eqref{G-XY}, \eqref{G-wwH}, \eqref{G-wX}, \eqref{G-wY}, \eqref{G-wH1},  
\eqref{G-wH2}, \eqref{G-H1X}, \eqref{G-H2X}, \eqref{G-H1Y}, \eqref{G-H2Y}, and \eqref{G-Yelim}
\end{lemma}
\begin{proof}
Now \eqref{G-XX} is equivalent to \eqref{GL2-XX}; 
\eqref{G-YY} is equivalent to \eqref{GL2-YY} with substitutions $u\to c_{\ell j} u$, $v\to c_{\ell j} v$;
\eqref{G-wwH} is equivalent to \eqref{GL2-wwH};
\eqref{G-XY} is equivalent to \eqref{GL2-XY} with substitution $t\to c_{\ell j} t$;
\eqref{G-wH1} is equivalent to \eqref{GL2-wH1};
\eqref{G-wH2} is equivalent to \eqref{GL2-wH2};
 \eqref{G-wX} is equivalent to \eqref{GL2-wX};
 \eqref{G-wY} is equivalent to \eqref{GL2-wY} with substitution $u\to c_{\ell j} u$.
 
Now
\eqref{GL2-H1X}  is equivalent to
\begin{align*}
{H}_1\left(t^{1/(\ell+1)} \right)X_{\ell,jk}\left(u \right){H}_1\left(t^{-1/(\ell+1)} \right)&=X_{\ell,jk}(tu),
\end{align*}
which gives \eqref{G-H1X} in the particular case where $s=t^{\ell+1}$ and $\left(t^{\ell+1} \right)^{1/(\ell+1)}=t$ for our chosen branch of the logarithm.
When $\left(t^{\ell+1} \right)^{1/(\ell+1)}\ne t$,  \eqref{G-H1X} is a redundant equation that also holds in  $\GL_2({\ell,j,k})$.
Note that these equations ensure that this relation is independent of the chosen  branch of logarithm.
Similar arguments give
\eqref{G-wwH} from \eqref{GL2-wwH} with substitution $s\to s^{(\ell+1)(j-\ell)}$;
\eqref{G-H1Y} from \eqref{GL2-H1Y} with substitutions $u\to c_{\ell j} u$, $s\to s^{\ell+1}$;
\eqref{G-H2X} from \eqref{GL2-H2X}  with substitution $s\to s^{j-\ell}$; and
\eqref{G-H2Y} from \eqref{GL2-H2Y} with substitutions $u\to c_{\ell j} u$, $s\to s^{j-\ell}$, and \eqref{G-Yelim} from \eqref{GL2-Yelim} with substitution $s\to c_{\ell j}s$.
\end{proof}

From Corollary~\ref{C-comm}\eqref{G-X-1X} and \eqref{G-Y-1Y}:
 $\left(X_{-1}(u),X_{j-1,jk}(v)\right) =1$ and $\left(Y_{-1}(u),Y_{j-1,jk}(v)\right) = 1$ hold in $\Uhpm$ respectively.

It now follows that $\Phpm$ satisfies the relations in Corollary~\ref{C-comm}, in addition to the following relations.
\begin{corollary}\label{Prelns} The relations
\eqref{G-H1X}, \eqref{G-H2X}, and \eqref{G-w-X} in $G(\frak m)$ hold in $\Php$ and 
\eqref{G-H1Y}, \eqref{G-H2Y}, and \eqref{G-w-Y} hold in $\Phm$. Moreover,  \eqref{G-Y-1X}:  $(Y_{-1}(u), X_{0,jk}(v)) = 1$ holds in $\widehat{P}^+$ and   \eqref{G-X-1Y}: $(X_{-1}(u), Y_{0,jk}(v)) = 1$ holds in $\widehat{P}^-$.
\end{corollary}

We have  shown that each relation in Definition~\ref{Defn} holds in whichever of the subgroups
$\GL_2(-1)\subseteq\Aut(\mathfrak{m})$, $\GL_2(\ell,j,k)\subseteq\Aut(\gl_2(\ell,j,k))$, $\widehat{{P}}^+\subseteq\Aut(\mhat)$, $\widehat{{P}}^-\subseteq\Aut(\mhat^-)$,  or
$\UhPi\subseteq\Aut(\widehat{\frak m}^\Pi)$ contains all the 
generators that appear in that relation. 

This implies that there are homomorphisms $$\GL_2(-1)\to\Aut(\mathfrak{m}),
\GL_2(\ell,j,k)\to\Aut(\gl_2(\ell,j,k)),$$
$$\widehat{{P}}^+\to\Aut(\mhat),\widehat{{P}}^-\to \Aut(\mhat^-).$$
 
In particular, we have constructed a group $\GL_2(-1)$ acting as automorphisms of $\mathfrak{m}$,
groups $\GL_2(\ell,j,k)$ acting as automorphisms on $\gl_2(\ell,j,k)$ %:=\C f_{\ell,jk}\oplus \mathfrak{h} \oplus \C e_{\ell,jk}$ that correspond to 
  and the corresponding subgroups of $\Ppm$ and $U_{\ell,jk;\,m,pq}$ acting as automorphisms of various completions of $\mathfrak{m}$.
For each of the groups $\GL_2(-1), \GL_2(\ell,j,k), \Ppm$, and~$ U_{\ell,jk;\,m,pq}$,  we have shown %(Sections~\ref{PositiveRootSys} and ~\ref{S-related}
that  the relations from Definition~\ref{Defn} involving the generators of the subgroup hold in the corresponding automorphism group $\Aut(\mathfrak{L})$.

We conjecture that the groups $\GL_2(-1), \GL_2(\ell,j,k), \Ppm$ and $ U_{\ell,jk;\,m,pq}$  are  isomorphic to the corresponding group of automorphisms in $\Aut(\mathfrak{L})$,  where
 $\mathfrak{L}$ is either $\mathfrak{m}$, or some $\gl_2$ subalgebra of $\mathfrak{m}$, or some completion of $\mathfrak{m}$.

%%%%%%%%%%%%%%%%%%%%%%%%%%%%%%%%%%%%%%%%%%%%%%%%%%%%%%
%%%%%%%%%%%%%%%%%%%%%%%%%%%%%%%%%%%%%%%%%%%%%%%%%%%%%%
\section{A homomorphism from $G(\mathfrak{m})$}\label{rep}
We conjecture that the subgroups $\GL_2(-1), \GL_2(\ell,j,k), \Ppm$ and $ U_{\ell,jk;\,m,pq}$  in $G(\m)$ are isomorpohic to the corresponding groups of
automorphisms of Lie algebras.
%Nor have we found a faithful representation for $G(\mathfrak{m})$. 
As a step in this direction,  we now give a detailed construction of  the negative completion $\widehat{\mathfrak{m}}^-$, and then construct a homomorphism from  $G(\mathfrak{m})$ to a quotient of $\text{Aut}(\widehat{\mathfrak{m}}^+) *_{\GL_2(-1)} \text{Aut}(\widehat{\mathfrak{m}}^-).$

\begin{longver}
%%%%%%%%%%%%%%%%%%%%%%%%%%%%%%%%%%%%%%%%%%%%%%%%%%%%%%
\subsection{Negative completion $\widehat{\mathfrak{m}}^-$}

We use the $\mathbb{Z}$-grading of  $\mathfrak{m}$ induced by the specialization map $\lambda: Q \to \mathbb{Z}$, which yields the decomposition:
$\mathfrak{m} = \bigoplus_{k \in \mathbb{Z}} \mathfrak{m}_k $, 
where $\mathfrak{m}_k = \bigoplus_{\alpha \in \lambda^{-1}(k) \cap \Delta} \mathfrak{m}_\alpha$. The zero-graded component is $\mathfrak{m}_0 = \mathfrak{gl}_2(-1)$. Recall that  $\mathfrak{n}_{\im}^{-} = \bigoplus_{k \in -\mathbb{N}} \mathfrak{m}_k$.
The completions of $\mathfrak{n}^-$ and $\mathfrak{n}_{\im}^-$ are now
\[ 
\widehat{\mathfrak{n}}^- := \prod_{\alpha \in \Delta_-} \mathfrak{m}_\alpha \quad \text{and} \quad \widehat{\mathfrak{n}}_{\im}^- := \prod_{\alpha \in \Delta_-^{\im}} \mathfrak{m}_\alpha  = \prod_{k=1}^\infty \mathfrak{m}_{-k} ,
\]
respectively. The negative formal of $\m$ is
\[ 
\widehat{\mathfrak{m}}^- := \widehat{\mathfrak{n}}^- \oplus \mathfrak{h} \oplus \mathfrak{n}^+.
\]
Since $\mathfrak{gl}_2(-1) = \mathfrak{m}_{-\alpha_{-1}} \oplus \mathfrak{h} \oplus \mathfrak{m}_{\alpha_{-1}}$, $\widehat{\mathfrak{n}}^- = \mathfrak{m}_{-\alpha_{-1}} \oplus \widehat{\mathfrak{n}}_{\im}^-$, and $\mathfrak{n}^+ = \mathfrak{m}_{\alpha_{-1}} \oplus \mathfrak{n}_{\im}^+$, we can also decompose $\widehat{\mathfrak{m}}^-$ as:
\[ 
\widehat{\mathfrak{m}}^- = \widehat{\mathfrak{n}}_{\im}^- \oplus \mathfrak{gl}_2(-1) \oplus \mathfrak{n}_{\im}^+ 
\]

We define a filtration of  subspaces 
\(
\widehat{\mathfrak{n}}_k^- := \prod_{j \le k} \mathfrak{m}_{j} 
\)
for each $k \in \mathbb{Z}$.
Then each $\widehat{\mathfrak{n}}_k^-$ is a module under the action of $\widehat{\mathfrak{n}}^-$ and $\widehat{\mathfrak{n}}_k^-$ is a pro-nilpotent Lie algebra for $k\leq -1$.
Note that  $\widehat{\mathfrak{n}}_{-1}^- = \widehat{\mathfrak{n}}_{\im}^-$.
For each $n \in \mathbb{N}$, we define a subgroup $\widehat{U}_n^-$ of $\operatorname{Aut}(\widehat{\mathfrak{m}}^-)$ consisting of automorphisms:
\[ 
\widehat{U}_n^- := \left\{ \varphi \in \operatorname{Aut}(\widehat{\mathfrak{m}}^-) \;\middle|\; \varphi(y) \in y + \widehat{\mathfrak{n}}_{k-n}^- \text{ whenever } y \in \mathfrak{m}_k \text{ for some } k \in \mathbb{Z} \right\} .
\]
In particular,
\(
\widehat{U}_{\im}^- := \widehat{U}_1^- .
\)

We set
\begin{align*}
\prod_{k=1}^\infty \exp\left(\ad\left(x_{-k} \right) \right)
&:= \cdots  \exp\left(\ad\left(x_{-3} \right) \right)  \exp\left(\ad\left(x_{-2} \right) \right) \exp\left(\ad\left(x_{-1} \right) \right)\\
&= \lim_{n\to\infty}\,  \exp\left(\ad\left(x_{-n} \right) \right) \cdots  \exp\left(\ad\left(x_{-2} \right) \right) \exp\left(\ad\left(x_{-1} \right) \right).
\end{align*}

\begin{lemma} Let $n\in\N$.
Let $x_{-k}\in \frak m_{-k}$  for all $k\geq1$. Then
$$\prod_{k=n}^\infty \exp\left(\ad\left(x_{-k} \right) \right)$$
expands to a formal sum which is pro-summable. Hence this infinite product converges to
 a well-defined automorphism in $\widehat{U}^-_{n}$.
\end{lemma}

For any $k \ge 1$, the quotient $\widehat{\mathfrak{n}}_{\im}^- / \widehat{\mathfrak{n}}_{-k}^-$ is  the nilpotent Lie algebra $\mathfrak{n}_{\im}^- / \mathfrak{n}_{-k}^-$. Thus
\[ 
\widehat{\mathfrak{n}}_{\im}^- = \varprojlim_{k} \left( \widehat{\mathfrak{n}}_{\im}^- / \widehat{\mathfrak{n}}_{-k}^- \right) 
\]
and so $\widehat{\mathfrak{n}}_{\im}^-$ is a {pro-nilpotent pro-Lie algebra}. Similarly,  $\widehat{U}_{\im}^- / \widehat{U}_{n}^-$ are finite-dimensional pro-unipotent groups. The group $\widehat{U}_{\im}^-$ is the topological inverse limit:
\[ 
\widehat{U}_{\im}^- = \varprojlim_{n} \left( \widehat{U}_{\im}^- / \widehat{U}_n^- \right) 
\]
Thus, $\widehat{U}_{\im}^-$ is  a {complete, pro-unipotent, and pro-nilpotent group}, acting  as smearing automorphisms on  $\widehat{\mathfrak{m}}^-$.

%%%%%%%%%%%%%%%%%%%%%%%%%%%%%%%%%%%%%%%%%%%%%%%%%%%%%%
\subsection{The homomorphism from $G(\mathfrak m)$}
\end{longver}

Consider the free product with amalgamation
$$\widehat{G} := \text{Aut}(\widehat{\mathfrak{m}}^+) *_{\GL_2(-1)} \text{Aut}(\widehat{\mathfrak{m}}^-).$$

We form the quotient of $\widehat{G}$ by the following  relations
\begin{align*}
\tag{Im:2}\label{G-XY}
  Y_{\ell,jk}(-t)X_{\ell,jk}(s)Y_{\ell,jk}(t)&= \\
  \qquad \qquad X_{\ell,jk}&\left(\frac{-t^{-1}}{c_{\ell j}}\right)Y_{\ell,jk}\left(-c_{\ell j}t^2s\right)X_{\ell,jk}\left(\frac{t^{-1}}{c_{\ell j}}\right),\\
\tag{Im:3}\label{G-wwH}
\widetilde{w}_{\ell,jk}\left(s^{(\ell+1)(j-\ell)} \right)\widetilde{w}_{\ell,jk}&=H_1\left((-s)^{j-\ell}\right) H_2\left((-s)^{\ell+1}\right), \\
\tag{Im:5a}\label{G-wH1}
  \widetilde{w}_{\ell,jk} H_1\left(s^{j-\ell} \right) {\widetilde{w}_{\ell,jk}}^{-1}&=  H_2\left(s^{-(\ell+1)}\right),\\
\tag{Im:5b}\label{G-wH2}
  \widetilde{w}_{\ell,jk} H_2\left(s^{\ell+1} \right) {\widetilde{w}_{\ell,jk}}^{-1}&=  H_1\left(s^{-(j-\ell)}\right),\\
\tag{U:2}\label{G-XYcomm}
  (X_{\ell,jk}(u), Y_{m,pq}(v))&=1\qquad\text{for $j\neq p$, or $k\neq q$, or $|\ell-m|>1$,}\\
\tag{U:3a}\label{G-XYj=2a} (X_{1,2k}(u),Y_{0,2k}(v))&=X_{-1}(uv),\\
\tag{U:3b}\label{G-XYj=2b}(X_{0,2k}(u),Y_{1,2k}(v))&=Y_{-1}(-uv)\\
\tag{Im:4a}\label{G-wX}
  \widetilde{w}_{\ell,jk} X_{\ell,jk}(u) {\widetilde{w}_{\ell,jk}}^{-1}&= Y_{\ell,jk}\left(\frac{-u}{c_{\ell j}}\right),\\
\tag{Im:4b}\label{G-wY}
  \widetilde{w}_{\ell,jk} Y_{\ell,jk}(u) {\widetilde{w}_{\ell,jk}}^{-1}&= X_{\ell,jk}\left(-c_{\ell j}u \right),\\
 \tag{U:5b} \label{G-Yelim}
Y_{\ell,jk}(s) =
X_{\ell,jk}\left(\frac{ s^{-1}}{c_{\ell j}}\right)
&H_1\left(\left[-{{c_{\ell j}}}s\right]^{-1/(\ell+1)}\right)
H_2\left(\left[-{{c_{\ell j}}}s\right]^{-1/(j-\ell)}\right)\widetilde{w}_{\ell,jk}X_{\ell,jk}\left(\frac{ s^{-1}}{c_{\ell j}}\right).
\end{align*}
We denote this quotient by $\widehat{G}/N_{XY}$, where $N_{XY}$ is the normal closure of the relations  \eqref{G-XY}, \eqref{G-wwH}, \eqref{G-wX}, \eqref{G-wY},
\eqref{G-wH1}, \eqref{G-wH2},  \eqref{G-XYcomm}, \eqref{G-XYj=2a}, \eqref{G-XYj=2b}, and \eqref {G-Yelim}.

\begin{theorem}\label{homom}
There is a homomorphism 
$$G(\mathfrak m)\to \widehat{G}/N_{XY}$$
\end{theorem}
\begin{proof} 
Let $\mathcal{X}$ be the set of abstract free generators for $G(\mathfrak{m})$ from Definition 7.1, consisting of $X_{-1}(u)$, $Y_{-1}(u)$, $H_1(s)$, $H_2(s)$, $X_{\ell,jk}(u)$, and $Y_{\ell,jk}(u)$ for $s \in \mathbb{C}^\times$ and $u \in \mathbb{C}$. Let $F(\mathcal{X})$ be the free group on these generators, so that $G(\mathfrak{m}) = F(\mathcal{X})/N_{\mathcal{R}}$ where $N_{\mathcal{R}}$ is the normal closure of the defining relations $\mathcal{R}$.

Let $\iota^+ \colon \text{Aut}(\widehat{\mathfrak{m}}^+) \to \widehat{G}$ and $\iota^- \colon \text{Aut}(\widehat{\mathfrak{m}}^-) \to \widehat{G}$ be the canonical inclusion homomorphisms into $\widehat{G}$. The inclusions agree on the shared $\GL_2(-1)$ subgroup: $\iota^+(g) = \iota^-(g)$ for all $g \in GL_2(-1)$. Let $\pi \colon \widehat{G} \to \widehat{G}/N_{XY}$ be the  quotient projection.

We define a map $\Phi$ on the generators $\mathcal{X}$ of $G(\mathfrak m)$ as follows:
\begin{align*}
\Phi(H_i(s)) &= \pi(\iota^+(H_i(s))) = \pi(\iota^-(H_i(s))) \quad \text{for } i \in \{1,2\},\\
\Phi(X_{-1}(u)) &= \pi(\iota^+(X_{-1}(u))) = \pi(\iota^-(X_{-1}(u))),\\
\Phi(Y_{-1}(u)) &= \pi(\iota^+(Y_{-1}(u))) = \pi(\iota^-(Y_{-1}(u))),\\
\Phi(X_{\ell,jk}(u)) &= \pi(\iota^+(X_{\ell,jk}(u))),\\
\Phi(Y_{\ell,jk}(u)) &= \pi(\iota^-(Y_{\ell,jk}(u))).
\end{align*}

Then $\Phi$ extends to a unique group homomorphism $\Phi \colon F(\mathcal{X}) \to \widehat{G}/N_{XY}$. This induces a well-defined homomorphism from $G(\mathfrak{m})$ if and only if $\Phi(r) = 1$ for all relations $r \in \mathcal{R}$.

In Corollary ~\ref{C-comm}, we showed that the relations 
 \eqref{GL2-G-XX}, \eqref{G-XX}, and \eqref{G-X-1X}
are satisfied in $\widehat{U}^+\subseteq \text{Aut}(\widehat{\mathfrak{m}}^+)$ and the relations \eqref{GL2-G-YY}, \eqref{G-YY}, and \eqref{G-Y-1Y} are satisfied in $\widehat{U}^-\subseteq \text{Aut}(\widehat{\mathfrak{m}}^-)$.

In Lemma~\ref{L-GL2-1}  we showed that the relations \eqref{GL2-G-H1H1}, \eqref{GL2-G-H2H2}, \eqref{GL2-G-H1H2}, \eqref{GL2-G-XX}, 
\eqref{GL2-G-YY}, \eqref{GL2-G-XY}, \eqref{GL2-G-wwH}, \eqref{GL2-G-wX}, \eqref{GL2-G-wY}, \eqref{GL2-G-wH1},  
\eqref{GL2-G-wH2}, \eqref{GL2-G-H1X}, \eqref{GL2-G-H2X}, \eqref{GL2-G-H1Y}, \eqref{GL2-G-H2Y}, and \eqref{GL2-G-Yelim}  are satisfied in  $\GL_2({-1})$ which is contained in $\text{Aut}(\widehat{\mathfrak{m}}^+)$ and $\text{Aut}(\widehat{\mathfrak{m}}^-)$.

In Corollary ~\ref{Prelns} we showed that the relations
\eqref{G-H1X} and \eqref{G-H2X} in $G(\frak m)$ hold in $\Php\subseteq\Aut(\mhat^+)$ and 
\eqref{G-H1Y} and \eqref{G-H2Y} hold in $\Phm\subseteq\Aut(\mhat^-)$, \eqref{G-w-X} is satisfied in $\Php$ and \eqref{G-w-Y} is satisfied in $\Phm$.   Also, \eqref{G-Y-1X} holds in $\widehat{P}^+$ and \eqref{G-X-1Y} holds in $\widehat{P}^-$.

The remaining relations of $G(\mathfrak{m})$ generate the normal closure $N_{XY}$. Since $\Phi(r) = 1$ for every relation $r \in \mathcal{R}$, the normal closure $N_{\mathcal{R}}$ is contained in the kernel of $\Phi$, hence there is a uniquely defined group homomorphism $G(\mathfrak{m}) \to \widehat{G}/N_{XY}$.
\end{proof}

We make the following conjecture.
\begin{conjecture}\label{NXY} In the group
$
\widehat{G}
=
\Aut(\widehat{\mathfrak m}^{+})
*_{\GL_2(-1)}
\Aut(\widehat{\mathfrak m}^{-})
$
we have
\[
N_{XY}\cap \GL_2(-1)=\{1\}.
\]
\end{conjecture}

\begin{theorem} Assuming that  Conjecture~\ref{NXY} is true, the group $\widehat{G}/N_{XY}$ is nontrivial. That is, there is an injective homomorphism $$\GL_2(-1)\to \widehat{G}/N_{XY}.$$
\end{theorem}

\begin{proof}
Let
$
q:\widehat{G}\longrightarrow
\widehat{G}/N_{XY}
$
be the quotient homomorphism. The canonical embeddings
\[
\GL_2(-1)\hookrightarrow \Aut(\widehat{\mathfrak m}^{+}),
\qquad
\GL_2(-1)\hookrightarrow \Aut(\widehat{\mathfrak m}^{-})
\]
induce an injective homomorphism
\[
\iota:\GL_2(-1)\hookrightarrow
\widehat{G}.
\]
We claim that \(q\circ \iota\) is injective. We have
\[
\ker(q\circ \iota)
=
\{g\in \GL_2(-1)\mid \iota(g)\in N_{XY}\}
=
\iota^{-1}\bigl(N_{XY}\cap \iota(\GL_2(-1))\bigr).
\]
By Conjecture~\ref{NXY}, this kernel is trivial.
Therefore \(q\circ \iota\) is an injective homomorphism. 
In particular, $\widehat{G}/N_{XY}$ is nontrivial.
\end{proof}

 conjecture  that the homomorphism $G(\mathfrak m)\to \widehat{G}/N_{XY}$ of Theorem~\ref{homom} is injective, which would establish   that $G(\mathfrak m)$ is not a presentation of the trivial group.

We also conjecture that $G(\mathfrak m)$ can be constructed as a generalized amalgamated product, which would also prove that $G(\mathfrak m)$ is nontrivial. We hope to address this in an upcoming work.

\bibliographystyle{amsalpha}
\bibliography{GPmath}{}

\end{document}